\documentclass[preprint,11pt]{elsarticle}
\usepackage{amsmath,amsfonts,amssymb,amsthm}
\usepackage{comment}
\usepackage{enumerate}
\usepackage[hmargin={28mm,28mm},vmargin={30mm,35mm}]{geometry}
\usepackage{paralist}
\usepackage{pgfplots}
\usepackage{tikz}
\usepackage{subcaption}
\usepackage{xcolor}
\usepackage{booktabs}

\usepackage[bold]{hhtensor}

\usetikzlibrary{spy}

\newtheorem{theorem}{Theorem}
\newtheorem{lemma}[theorem]{Lemma}
\newtheorem{proposition}[theorem]{Proposition}

\theoremstyle{remark}
\newtheorem{remark}[theorem]{Remark}
\theoremstyle{definition}

\newcommand{\lsim}[1][]{\lesssim_{#1}}
\newcommand{\esim}[1][]{\simeq_{#1}}

\newcommand{\coloneq}{\mathrel{\mathop:}=}
\newcommand{\eqcolon}{=\mathrel{\mathop:}}

\newcommand{\Real}{\mathbb{R}}

\DeclareMathOperator{\tr}{tr}

\newcommand{\SCAL}{{\cdot}}
\newcommand{\SSCAL}{:}
\newcommand{\GRAD}{\boldsymbol{\nabla}}
\newcommand{\DIV}{\boldsymbol{\nabla}{\cdot}}
\newcommand{\LAPL}{{\Delta}}
\newcommand{\GRADh}{\GRAD_h}

\newcommand{\st}{\; : \;}

\newcommand{\Id}[1][d]{\matr{I}_{#1}}

\newcommand{\Mh}[1][h]{\mathcal{M}_{#1}}
\newcommand{\Th}[1][h]{\mathcal{T}_{#1}}
\newcommand{\Fh}[1][h]{\mathcal{F}_{#1}}
\newcommand{\Fhi}[1][h]{\mathcal{F}_{#1}^{{\rm i}}}
\newcommand{\Fhb}[1][h]{\mathcal{F}_{#1}^{{\rm b}}}

\newcommand{\normal}{\vec{n}}

\newcommand{\Poly}[1]{\mathbb{P}^{#1}}

\newcommand{\norm}[2][]{\|#2\|_{#1}}
\newcommand{\seminorm}[2][]{|#2|_{#1}}

\newcommand{\term}{\mathfrak{T}}

\newcommand{\Pe}[1][TF]{\mathrm{Re}_{#1}}
\newcommand{\Reynolds}{\mathrm{Re}}

\newcommand{\UT}{\underline{\boldsymbol{U}}_T^k}
\newcommand{\Uh}[1][]{\underline{\boldsymbol{U}}_{h#1}^k}
\newcommand{\Ph}{P_h^k}
\newcommand{\Uhz}{\underline{\boldsymbol{U}}_{h,0}^k}

\newcommand{\hvec}[1]{\hat{\vec{#1}}}
\newcommand{\uvec}[1]{\underline{\vec{#1}}}
\newcommand{\uhvec}[1]{\hat{\underline{\vec{#1}}}}

\newcommand{\rT}[1][T]{\boldsymbol{r}_{#1}^{k+1}}
\newcommand{\rh}[1][h]{\boldsymbol{r}_{#1}^{k+1}}

\newcommand{\lproj}[2][h]{\pi_{#1}^{#2}}
\newcommand{\vlproj}[2][h]{\boldsymbol{\pi}_{#1}^{#2}}
\newcommand{\mlproj}[2][h]{\boldsymbol{\pi}_{#1}^{#2}}

\newcommand{\IT}[1][k]{\underline{\boldsymbol{I}}_T^{#1}}
\newcommand{\Ih}[1][k]{\underline{\boldsymbol{I}}_h^{#1}}

\newcommand{\jump}[2][F]{[#2]_{#1}}

\newcommand{\GwT}[2][\uvec{w}_T]{\vec{\mathcal{G}}_T^k(#1;#2)}
\newcommand{\GT}[1][k]{\matr{G}_T^{#1}}

\newcommand{\DT}[1][k]{D_T^{#1}}

\newcommand{\flux}{\vec{\Phi}}
\newcommand{\visc}{{\rm visc}}
\newcommand{\conv}{{\rm conv}}

\begin{document}

\begin{frontmatter}

\title{A Hybrid High-Order method for the incompressible Navier--Stokes equations based on Temam's device}

\author[label1]{Lorenzo Botti}
\author[label2]{Daniele A. Di Pietro}
\author[label3]{J\'er\^ome Droniou}

\address[label1]{Department of Engineering and Applied Sciences, University of Bergamo (Italy), \emph{lorenzo.botti@unibg.it}}
\address[label2]{IMAG, Univ Montpellier, CNRS, Montpellier, France, \emph{daniele.di-pietro@umontpellier.fr}}
\address[label3]{School of Mathematical Sciences, Monash University, Melbourne (Australia), \emph{jerome.droniou@monash.edu}}

\begin{abstract}
  In this work we propose a novel Hybrid High-Order method for the incompressible Navier--Stokes equations based on a formulation of the convective term including Temam's device for stability.
  The proposed method has several advantageous features:
  it supports arbitrary approximation orders on general meshes including polyhedral elements and non-matching interfaces;
  it is inf-sup stable;
  it is locally conservative;
  it supports both the weak and strong enforcement of velocity boundary conditions;
  it is amenable to efficient computer implementations where a large subset of the unknowns is eliminated by solving local problems inside each element.
  Particular care is devoted to the design of the convective trilinear form, which mimicks at the discrete level the non-dissipation property of the continuous one.
  The possibility to add a convective stabilisation term is also contemplated, and a formulation covering various classical options is discussed.
  The proposed method is theoretically analysed, and an energy error estimate in $h^{k+1}$ (with $h$ denoting the meshsize) is proved under the usual data smallness assumption.
  A thorough numerical validation on two and three-dimensional test cases is provided both to confirm the theoretical convergence rates and to assess the method in more physical configurations (including, in particular, the well-known two- and three-dimensional lid-driven cavity problems).
\end{abstract}


\end{frontmatter}



\section{Introduction}

In this work we propose a novel Hybrid High-Order (HHO) method for the incompressible Navier--Stokes equations based on a formulation of the convective term including Temam's device for stability \cite{Temam:79}.

Introduced in \cite{Di-Pietro.Ern:15}, HHO methods are new generation discretisation methods for PDEs based on discrete unknowns that are broken polynomials on the mesh and on its skeleton.
Unlike classical finite element methods, the notion of reference element is not present in HHO, and basis functions are not explicitly defined.
Instead, the discrete unknowns are combined to reconstruct relevant quantities inside each element by mimicking integration by parts formulas.
These local reconstructions are used to formulate consistent Galerkin terms, while stability is achieved by stabilisation terms devised at the element level.
The HHO approach has several advantages:
it is dimension-independent;
it supports arbitrary approximation orders on general meshes including polyhedral elements, non-matching junctions and, possibly, curved faces \cite{Botti.Di-Pietro:18};
it is locally conservative;
it is amenable to efficient (parallel or serial) computer implementations.
The HHO method proposed in this work has additional advantageous features specific to the incompressible Navier--Stokes problem: it satisfies a uniform inf-sup condition, leading to a stable pressure-velocity coupling; it behaves robustly for large Reynolds numbers; it supports both weakly and strongly enforced boundary conditions; at each nonlinear iteration, it requires to solve a linear system where the only globally coupled unknowns are the face velocities and the mean value of the pressure inside each element.

The literature on the numerical approximation of the incompressible Navier--Stokes equations is vast, and giving a detailed account lies out of the scope of the present work.
We therefore mention here only those numerical methods which share similar features with the HHO approach.
Discontinuous Galerkin (DG) methods, which hinge on discrete unknowns that are broken polynomials on the mesh and have the potential to support polyhedral elements, have gained significant popularity in computational fluid mechanics.
Their application to the discretisation of incompressible flows has been considered, e.g., in \cite{Karakashian.Katsaounis:00,Cockburn.Kanschat.ea:05,Girault.Riviere.ea:05,Bassi.Crivellini.ea:06,Cockburn.Kanschat.ea:07,Mozolevski.Suli.ea:07,Di-Pietro.Ern:10,Botti.Di-Pietro:11,Riviere.Sardar:14,Tavelli.Dumbser:14}.
We particularly mention here \cite{Di-Pietro.Ern:10}, where a DG trilinear form based on Temam's device was proposed which, unlike the one considered here, leads to a non-conservative discretisation of the momentum equation.
Linked to DG methods are Cell Centered Galerkin methods, which hinge on incomplete polynomial spaces; see  \cite[Section 4]{Di-Pietro:12} for their application to the incompressible Navier--Stokes problem.
A second family of discretisation methods worth mentioning here are Hybridisable Discontinuous Galerkin (HDG) methods, which can be regarded as an evolution of DG methods where both face and element unknowns are present.
Their application to the incompressible Navier--Stokes equations has been considered, e.g., in \cite{Nguyen.Peraire.ea:11,Giorgiani.Fernandez-Mendez.ea:14,Qiu.Shi:16,Ueckermann.Lermusiaux:16,Cesmelioglu.Cockburn.ea:17}; see also \cite{Di-Pietro.Krell:18} (and, in particular, Remark 5 therein) for a comparison with HHO methods.
More recent polyhedral technologies have also been applied to the discretisation of the incompressible Navier--Stokes equations.
We cite here, in particular, the two-dimensional Virtual Element Method (VEM) of \cite{Beirao-da-Veiga.Lovadina.ea:18}; see also the related work \cite{Beirao-da-Veiga.Lovadina.ea:17}.
For a study of the relations among HDG, HHO, and VEM in the context of scalar diffusion problems, we refer the reader to \cite{Boffi.Di-Pietro:18,Di-Pietro.Droniou.ea:18}.
Finally, as regards HHO methods, applications to incompressible flows have been considered in \cite{Aghili.Boyaval.ea:15,Di-Pietro.Ern.ea:16*1,Di-Pietro.Krell:18,Botti.Di-Pietro.ea:18}.

For a given integer $k\ge 0$, the HHO method proposed here hinges on discrete velocity unknowns that are vector-valued polynomials of total degree $\le k$ on mesh elements and faces, and discontinuous pressure unknowns of total degree $\le k$.
Based on the discrete velocity unknowns, we reconstruct inside each element:
(i) a velocity one degree higher than element unknowns, leading to the characteristic $\mathcal{O}(h^{k+1})$-approximation of the viscous term;
(ii) a divergence in the space of scalar-valued polynomials of total degree $\le\ell$ whose purpose its twofold: with $\ell=k$, it is used in the discretisation of the pressure-velocity coupling; with $\ell=2k$, it appears in Temam's contribution to the convective trilinear form;
(iii) a directional derivative used to formulate the consistent contribution in the convective term.
The convective trilinear form embedding Temam's device is the first main novelty of this paper, and it mimicks at the discrete level the non-dissipation property valid at the continuous level.
Stability in the convection-dominated regime can be strenghtened by introducing a convective stabilisation term, for which a variety of classical options are here adapted to the HHO framework.
The main source of inspiration is \cite{Di-Pietro.Droniou.ea:15}, where HHO methods for locally degenerate scalar diffusion-advection-reaction problems are developed.
The second important novelty of this work is the extension of Nitsche's technique to weakly enforce boundary conditions on the velocity in the HHO scheme.
The weak enforcement of boundary conditions can improve the resolution of boundary layers and simplifies the practical implementation.

Theoretical justification and numerical validation of the proposed method are provided.
From the theoretical point of view, we prove an error estimate in $h^{k+1}$ for the discrete $H_0^1$-like norm of the error on the velocity and the $L^2$-norm of the error on the pressure.
As customary for the Navier--Stokes equations, this error estimate is derived under a data smallness assumption.
Following the ideas in \cite{Di-Pietro.Krell:18}, one could also prove convergence without any smallness assumption on the data nor additional regularity on the exact solution. These developments are omitted for the sake of brevity.
From the numerical point of view, we provide a thorough assessment of the proposed method using well-known benchmark problems from the literature.
Specifically, the convergence rates with or without convective stabilisation and with strongly or weakly enforced boundary conditions are assessed using Kovasznay's solution \cite{Kovasznay:48}.
The robustness in the convection-dominated regime is assessed, on the other hand, simulating the two- and three-dimensional lid-driven cavity flows for Reynolds numbers up to $\pgfmathprintnumber{20000}$ and polynomial degrees up to $k{=}8$.
For these test cases, no analytical solution is available, so we compare with results from the literature.

The rest of this work is organised as follows.
In Section \ref{sec:setting.cont} we discuss the continuous problem and formulate a key remark that will inspire the design of the discrete trilinear form.
In Section \ref{sec:setting} we discuss the discrete setting, with particular focus on the local reconstruction operators at the heart of the HHO method.
In Section \ref{sec:method} we discuss the discretisation of the various terms, highlight the properties relevant for the analysis, formulate the discrete problem, and carry out its convergence analysis.
Section \ref{sec:tests} contains a thorough numerical validation on tests commonly used in the literature.
The flux formulation of the method highlighting its local conservation properties is derived in Section \ref{sec:flux}.
Finally, \ref{sec:proofs} contains the proofs of intermediate technical results.
Readers mainly interested in the numerical recipe and results can skip this technical appendix at first reading.


\section{Continuous setting and a key remark}\label{sec:setting.cont}

Let $\Omega\subset\Real^d$, $d\in\{2,3\}$, denote a bounded connected polyhedral domain with boundary $\partial\Omega$.
We consider here a Newtonian fluid with constant density.
Denote by $\vec{f}\in L^2(\Omega)^d$ a force per unit volume, by $\nu>0$ a real number representing the kinematic viscosity, and set $L^2_0(\Omega)\coloneq\left\{q\in L^2(\Omega)\st\int_\Omega q = 0\right\}$.
The weak formulation of the incompressible Navier--Stokes equations reads:
Find $(\vec{u},p)\in H^1_0(\Omega)^d\times L^2_0(\Omega)$ such that
\begin{subequations}\label{eq:weak}
  \begin{alignat}{2}\label{eq:weak:momentum}
    \nu a(\vec{u},\vec{v})
    + t(\vec{u},\vec{u},\vec{v})
    + b(\vec{v},p)
    &= \int_\Omega \vec{f}\SCAL\vec{v}
    &\qquad&\forall \vec{v}\in H^1_0(\Omega)^d,
    \\ \label{eq:weak:mass}
    -b(\vec{u},q)&=0 &\qquad&\forall q\in L^2(\Omega),
  \end{alignat}
\end{subequations}
with bilinear forms $a$ and $b$ and trilinear form $t$ such that
$$
a(\vec{w},\vec{v})\coloneq\int_\Omega\GRAD\vec{w}\SSCAL\GRAD\vec{v},\qquad
b(\vec{v},q)\coloneq -\int_\Omega (\DIV\vec{v}) q,\qquad
t(\vec{w},\vec{v},\vec{z})\coloneq\int_\Omega (\vec{w}\SCAL\GRAD)\vec{v}\SCAL\vec{z}
$$
Here, $\vec{w}\SCAL\GRAD$ is the differential operator $\vec{w}\SCAL\GRAD=\sum_{j=1}^d w_j\partial_j$,
so that, if $\vec{w}=(w_i)_{i=1,\ldots,d}$, $\vec{u}=(u_i)_{i=1,\ldots,d}$, and $\vec{v}=(v_i)_{i=1,\ldots,d}$, then $(\vec{w}\SCAL\GRAD)\vec{u}\SCAL\vec{v}=\sum_{i,j=1}^d w_j(\partial_j u_i) v_i$.
In \eqref{eq:weak}, we have considered wall boundary conditions for the sake of simplicity: other usual boundary conditions can be considered.

The well-posedness of problem \eqref{eq:weak} for small data hinges on the coercivity of the viscous bilinear form $a$, on the inf-sup stability of the pressure-velocity coupling bilinear form $b$, and on the non-dissipativity of the convective trilinear form $t$.
These same stability properties, along with suitable consistency requirements, are key to the design of a convergent HHO method.
A coercive discrete counterpart of the bilinear form $a$ has been proposed in \cite{Di-Pietro.Ern.ea:14}; see also \cite{Di-Pietro.Droniou.ea:15} concerning the weak enforcement of boundary conditions.
Inf-sup stable HHO counterparts of the bilinear form $b$, on the other hand, can be found, e.g., in \cite{Aghili.Boyaval.ea:15,Di-Pietro.Ern.ea:16*1,Boffi.Botti.ea:16}.
Here, we generalise them to incorporate weakly enforced boundary conditions.

Let us examine the non-dissipativity property of $t$ in order to illustrate the strategy used to mimick it at the discrete level.
We start by noting the following integration by parts formula:
For all $\vec{w},\vec{v},\vec{z}\in H^1(\Omega)^d$,
\begin{equation}\label{eq:ibp.cont}
  \int_\Omega(\vec{w}\SCAL\GRAD)\vec{v}\SCAL\vec{z}
  + \int_\Omega(\vec{w}\SCAL\GRAD)\vec{z}\SCAL\vec{v}
  + \int_\Omega(\DIV\vec{w})(\vec{v}\SCAL\vec{z})
  = \int_{\partial\Omega}(\vec{w}\SCAL\normal)(\vec{v}\SCAL\vec{z}),
\end{equation}
where $\normal$ denotes the outward unit normal vector to $\partial\Omega$.
Using \eqref{eq:ibp.cont} with $\vec{w}=\vec{v}=\vec{z}=\vec{u}$ ($\vec{u}$ being the velocity solution to \eqref{eq:weak}), we get
\begin{equation}\label{eq:t.non-dissipativity}
  t(\vec{u},\vec{u},\vec{u})
  = \int_\Omega (\vec{u}\SCAL\GRAD)\vec{u}\SCAL\vec{u}
  = -\frac12\int_\Omega(\DIV\vec{u})(\vec{u}\SCAL\vec{u})
  + \frac12\int_{\partial\Omega}(\vec{u}\SCAL\normal)(\vec{u}\SCAL\vec{u})
  = 0,
\end{equation}
where we have used \eqref{eq:weak:mass} to infer $\DIV\vec{u}=0$ and cancel the first term, and the fact that $\vec{u}$ vanishes on $\partial\Omega$ to cancel the second.
This relation expresses the fact that the convective term does not contribute to the kinetic energy balance, obtained taking $\vec{v}=\vec{u}$ in \eqref{eq:weak:momentum}.

When attempting to reproduce property \eqref{eq:t.non-dissipativity} at the discrete level, a difficulty arises: the discrete counterparts of the terms in the right-hand side of \eqref{eq:t.non-dissipativity} may not vanish, since the discrete solution may not be ``sufficiently'' divergence-free (see Remark \ref{rem:incompressibility}) and/or it may not be zero on $\partial\Omega$.
To overcome this difficulty, the following modified expression for $t$ can be used as a starting point, an idea which can be traced back to Temam \cite{Temam:79}:
\begin{equation}\label{eq:tilde.t}
  \tilde{t}(\vec{w},\vec{v},\vec{z})
  = \int_\Omega(\vec{w}\SCAL\GRAD)\vec{v}\SCAL\vec{z}
  + \frac12\int_\Omega(\DIV\vec{w})(\vec{v}\SCAL\vec{z})
  - \frac12\int_{\partial\Omega}(\vec{w}\SCAL\normal)(\vec{v}\SCAL\vec{z}).
\end{equation}
Repeating the above reasoning, it is a simple matter to check that this trilinear form satisfies the following generalised version of property \eqref{eq:t.non-dissipativity}:
For all $\vec{w},\vec{v}\in H^1(\Omega)^d$, $\tilde{t}(\vec{w},\vec{v},\vec{v})=0$.
This means, in particular, that $\tilde{t}$ is non-dissipative even if $\vec{w}$ is not divergence free and $\vec{v}$ does not vanish on $\partial\Omega$ (as may be the case for the discrete velocity).



\section{Discrete setting}\label{sec:setting}

In this section we establish the discrete setting. After briefly recalling the notion of mesh and introducing projectors on local polynomial spaces, we define the spaces of discrete unknowns and the local reconstructions upon which the HHO method is built.

\subsection{Mesh}\label{sec:setting:mesh}

Throughout the paper, we will use for the sake of simplicity the three-dimensional nomenclature also when $d=2$, i.e., we will speak of polyhedra and faces rather than polygons and edges.
We define a mesh as a couple $\Mh\coloneq(\Th,\Fh)$, where $\Th$ is a finite collection of polyhedral elements $T$ such that $h\coloneq\max_{T\in\Th}h_T>0$ with $h_T$ denoting the diameter of $T$, while $\Fh$ is a finite collection of planar faces $F$. It is assumed henceforth that the mesh $\Mh$ matches the geometrical requirements detailed in \cite[Definition 7.2]{Droniou.Eymard.ea:18}; see also~\cite[Section 2]{Di-Pietro.Tittarelli:18}.
Boundary faces lying on $\partial\Omega$ and internal faces contained in
$\Omega$ are collected in the sets $\Fhb$ and $\Fhi$, respectively.
For every mesh element $T\in\Th$, we denote by $\Fh[T]$ the subset of $\Fh$ containing the faces that lie on the boundary $\partial T$ of $T$.
Symmetrically, for every $F\in\Fh$, we denote by $\Th[F]$ the subset of $\Th$ containing the elements that share $F$ (one if $F\in\Fhb$, two if $F\in\Fhi$).
For each mesh element $T\in\Th$ and face $F\in\Fh[T]$, $\normal_{TF}$ is the (constant) unit normal
vector to $F$ pointing out of $T$.
For any boundary face $F\in\Fhb$, we denote by $T_F$ the unique element of $\Th$ such that $F\in\Fh[T_F]$ and we let $\normal_F\coloneq\normal_{T_FF}$.

Our focus is on the so-called $h$-convergence analysis, so we consider a sequence of refined meshes that is regular in the sense of~\cite[Definition~3]{Di-Pietro.Tittarelli:18} with regularity parameter uniformly bounded away from zero.
The mesh regularity assumption implies, in particular, that the diameters of a mesh element and its faces are uniformly comparable and that the number of faces in $\Fh[T]$ is bounded above by an integer independent of $h$.

\subsection{Notation and basic results}\label{sec:setting:projectors}

We abridge into $a\lsim[C] b$ the inequality $a\le Cb$ with constant $C>0$ independent of $h$, $\nu$, and, for local inequalities, on the mesh element or face.
We also write $a\esim[C] b$ for $C^{-1}a\le b\le Ca$ with $C$ as above.
When the constant name is not relevant, we simply write $a\lsim b$ and $a\simeq b$.

Let $X$ denote a mesh element or face and, for an integer $l\ge 0$, denote by $\Poly{l}(X)$ the space spanned by the restrictions to $X$ of polynomials in the space variables of total degree $\le l$.
We denote by $\lproj[X]{l}:L^1(X)\to\Poly{l}(X)$ the $L^2$-orthogonal projector such that, for all $v\in L^1(X)$,
\begin{equation}\label{eq:lproj}
  \int_X(v-\lproj[X]{l}v)w = 0\qquad\forall w\in\Poly{l}(X).
\end{equation}
The vector- and matrix-valued $L^2$-orthogonal projectors, both denoted by $\vlproj[X]{l}$, are obtained applying $\lproj[X]{l}$ component-wise.
The following optimal $W^{s,p}$-approximation properties are proved in \cite[Appendix A.2]{Di-Pietro.Droniou:17} using the classical theory of \cite{Dupont.Scott:80} (cf. also \cite[Chapter 4]{Brenner.Scott:08}).
Let $s\in\{0,\ldots,l+1\}$ and $p\in[1,\infty]$.
It holds with hidden constant only depending on $d$, $l$, $s$, $p$, and the mesh regularity parameter:
For all $T\in\Th$, all $v\in W^{s,p}(T)$, and all $m\in\{0,\ldots,s\}$,
\begin{subequations}\label{eq:lproj.approx}
\begin{equation}\label{eq:lproj.approx:T}
  \seminorm[W^{m,p}(T)]{v-\lproj[T]{l}v} \lsim h_T^{s-m}\seminorm[W^{s,p}(T)]{v},
\end{equation}
and, if $s\ge 1$ and $m\le s-1$,
\begin{equation}\label{eq:lproj.approx:FT}
  h_T^{\frac1p}\seminorm[{W^{m,p}(\Fh[T]})]{v-\lproj[T]{l}v} \lsim h_T^{s-m}\seminorm[W^{s,p}(T)]{v},
\end{equation}
\end{subequations}
where $W^{m,p}(\Fh[T])$ is the space spanned by functions that are in $W^{m,p}(F)$ for all $F\in\Fh[T]$, endowed with the corresponding broken norm.

At the global level, the space of broken polynomial functions on $\Th$ of total degree $\le l$ is denoted by $\Poly{l}(\Th)$, and $\lproj{l}$ is the corresponding $L^2$-orthogonal projector.
Broken polynomial spaces form subspaces of the broken Sobolev spaces $W^{s,p}(\Th)\coloneq\left\{ v\in L^p(\Omega)\st v_{|T}\in W^{s,p}(T)\quad\forall T\in\Th\right\}$, which will be used to express the regularity requirements in consistency estimates.
We additionally set, as usual, $H^s(\Th)\coloneq W^{s,2}(\Th)$.

\subsection{Discrete velocity space}\label{sec:setting:dofs}

Let a polynomial degree $k\ge 0$ be fixed.
We define the following space of hybrid discrete velocity unknowns:
\begin{equation*}
  \label{eq:UT}
  \Uh\coloneq\left\{
  \uvec{v}_h=( (\vec{v}_T)_{T\in\Th}, (\vec{v}_F)_{F\in\Fh} )\st
  \vec{v}_T\in\Poly{k}(T)^d\quad\forall T\in\Th\mbox{ and }
  \vec{v}_F\in\Poly{k}(F)^d\quad\forall F\in\Fh  
  \right\}.
\end{equation*}
For all $\uvec{v}_h\in\Uh$, we denote by $\vec{v}_h\in\Poly{k}(\Th)^d$ the vector-valued broken polynomial function obtained patching element-based unknowns, that is
\begin{equation}\label{eq:vh}
  \vec{v}_{h|T}\coloneq\vec{v}_T\qquad\forall T\in\Th.
\end{equation}
The restrictions of $\Uh$ and $\uvec{v}_h\in\Uh$ to a generic mesh element $T\in\Th$ are respectively denoted by $\UT$ and $\uvec{v}_T=(\vec{v}_T,(\vec{v}_F)_{F\in\Fh[T]})$.
The vector of discrete variables corresponding to a smooth function on $\Omega$ is obtained via the global interpolation operator $\Ih:H^1(\Omega)^d\to\Uh$ such that, for all $\vec{v}\in H^1(\Omega)^d$,
$$
\Ih\vec{v} \coloneq ( (\vlproj[T]{k}\vec{v}_{|T})_{T\in\Th}, (\vlproj[F]{k}\vec{v}_{|F})_{F\in\Fh}).
$$
Its restriction to a generic mesh element $T\in\Th$ is $\IT:H^1(T)^d\to\UT$ such that, for all $\vec{v}\in H^1(T)^d$,
\begin{equation}\label{eq:IT}
  \IT\vec{v} = ( \vlproj[T]{k}\vec{v}, (\vlproj[F]{k}\vec{v}_{|F})_{F\in\Fh[T]} ).
\end{equation}
We equip $\Uh$ with the discrete $H_0^1$-like norm such that, for all $\uvec{v}_h\in\Uh$,
\begin{equation}\label{eq:norm.1h}
  \norm[1,h]{\uvec{v}_h}\coloneq\Bigg(
  \sum_{T\in\Th}\norm[1,T]{\uvec{v}_T}^2
  + \sum_{F\in\Fhb}h_F^{-1}\norm[L^2(F)^d]{\vec{v}_F}^2
  \Bigg)^{\frac12},
\end{equation}
where, for all $T\in\Th$,
\begin{equation}\label{eq:norm.1T}
  \norm[1,T]{\uvec{v}_T}\coloneq\left(
  \norm[L^2(T)^{d\times d}]{\GRAD\vec{v}_T}^2
  + \sum_{F\in\Fh[T]}h_F^{-1}\norm[L^2(F)^d]{\vec{v}_F-\vec{v}_T}^2
  \right)^{\frac12}.
\end{equation}
In the analysis, we will frequently invoke the following discrete Sobolev embeddings in $\Uh$, which state that, up to a certain $q$ depending on the space dimension, the $L^q$-norms of the broken polynomial function \eqref{eq:vh} obtained patching element unknowns are controlled (uniformly with respect to the meshsize) by the discrete $H_0^1$-like norm.
\begin{proposition}[Discrete Sobolev embeddings]\label{prop:sobolev.embeddings}
  Let $1\le q<\infty$ if $d=2$ and $1\le q\le 6$ if $d=3$.
  Then, it holds with hidden constant depending only on $\Omega$, $k$, $q$, and the mesh regularity parameter:
  \begin{equation}\label{eq:sobolev.embeddings}
    \forall\uvec{v}_h\in\Uh,\qquad
    \norm[L^q(\Omega)^d]{\vec{v}_h}\lsim\norm[1,h]{\uvec{v}_h}.
  \end{equation}
\end{proposition}
\begin{proof}
  See \ref{sec:convergence.analysis:preliminary.results:sobolev.embeddings}.
  We also refer to \cite[Proposition 5.4]{Di-Pietro.Droniou:17} for the proof of a similar result in the subspace \eqref{eq:Uhz} of $\Uh$ with strongly enforced boundary conditions.  
\end{proof}
An immediate consequence of Proposition \ref{prop:sobolev.embeddings} is that the map $\norm[1,h]{{\cdot}}$ is a norm on $\Uh$.

\subsection{Local reconstructions}

We next introduce the local reconstructions at the core of the HHO method along with their relevant properties.
Throughout the rest of this section, we work on a fixed mesh element $T\in\Th$.

\subsubsection{Gradient}

Inspired by the principles of \cite[Sections 4.3.1.1 and 4.4.2]{Di-Pietro.Tittarelli:18}, for any integer $\ell\ge 0$ we define the local gradient reconstruction $\GT[\ell]:\UT\to\Poly{\ell}(T)^{d\times d}$ such that, for all $\uvec{v}_T\in\UT$ and all $\matr{\tau}\in\Poly{\ell}(T)^{d\times d}$,
\begin{subequations}
  \begin{align}\label{eq:GT.bis}
    \int_T\GT[\ell]\uvec{v}_T\SSCAL\matr{\tau}
    &= -\int_T\vec{v}_T\SCAL(\DIV\matr{\tau})
    + \sum_{F\in\Fh[T]}\int_F\vec{v}_F\SCAL\matr{\tau}\normal_{TF}
    \\ \label{eq:GT}
    &= \int_T\GRAD\vec{v}_T\SSCAL\matr{\tau}
    + \sum_{F\in\Fh[T]}\int_F (\vec{v}_F-\vec{v}_T)\SCAL\matr{\tau}\normal_{TF}.
  \end{align}
\end{subequations}
The right-hand side of \eqref{eq:GT.bis} mimicks an integration by parts formula where the role of the function represented by $\uvec{v}_T$ is played by $\vec{v}_T$ in the volumetric term and by $\vec{v}_F$ in the boundary term.
The reformulation \eqref{eq:GT}, obtained integrating by parts the first term in \eqref{eq:GT.bis}, shows that $\GT\uvec{v}_T$ stems from two contributions: the gradient of the element-based unknown and a boundary correction involving the differences between element- and face-based unknowns.
If $\ell\ge k$, it follows from \cite[Eq. (22)]{Di-Pietro.Krell:18} that, for all $\vec{v}\in H^1(T)^d$,
\begin{equation}\label{eq:GT.orth}
  \int_T(\GT[\ell]\IT\vec{v}-\GRAD\vec{v})\SSCAL\matr{\tau}=0\qquad\forall\matr{\tau}\in\Poly{k}(T)^{d\times d}.
\end{equation}
Comparing with the definition \eqref{eq:lproj} of the $L^2$-orthogonal projector, this gives, in particular,
\begin{equation}\label{eq:GT:commuting}
  \GT[k]\IT\vec{v}=\mlproj[T]{k}(\GRAD\vec{v}).
\end{equation}
The following approximation properties for $(\GT[\ell]\circ\IT)$ have been proved in \cite[Proposition 1]{Di-Pietro.Krell:18}:
For all $\vec{v}\in H^m(T)^d$ with $m=\ell+2$ if $\ell\le k$, $m=k+1$ otherwise,
\begin{equation}\label{eq:approx.GT}
  \norm[L^2(T)^{d\times d}]{\GT[\ell]\IT\vec{v}-\GRAD\vec{v}}
  + h_T^{\frac12}\norm[L^2(\partial T)^{d\times d}]{\GT[\ell]\IT\vec{v}-\GRAD\vec{v}}
  \lsim h_T^{m-1}\seminorm[H^m(T)^d]{\vec{v}}.
\end{equation}

\subsubsection{Velocity}

From $\GT$, we can obtain a reconstruction $\rT:\UT\to\Poly{k+1}(T)^d$ of the velocity one order higher than element-based unknowns as follows:
For all $\uvec{v}_T\in\UT$, $\rT\uvec{v}_T$ is such that
\[
  \int_T(\GRAD\rT\uvec{v}_T-\GT\uvec{v}_T)\SSCAL\GRAD\vec{w}=0\quad\forall\vec{w}\in\Poly{k+1}(T)^d,\qquad
  \int_T(\rT\uvec{v}_T-\vec{v}_T)=\vec{0},
\]
that is, the gradient of $\rT\uvec{v}_T$ is the $L^2$-orthogonal projection of $\GT\uvec{v}_T$ on $\GRAD\Poly{k+1}(T)^d$, and $\rT\uvec{v}_T$ has the same mean value over $T$ as the element-based unknown $\vec{v}_T$.
We note for further use the following characterisation of $\GRAD\rT\uvec{v}_T$, obtained writing \eqref{eq:GT} for $\matr{\tau}=\GRAD\vec{w}$:
\begin{equation}\label{eq:characterisation.rT}
  \int_T\GRAD\rT\uvec{v}_T\SSCAL\GRAD\vec{w}
  = \int_T\GRAD\vec{v}_T\SSCAL\GRAD\vec{w} + \sum_{F\in\Fh[T]}\int_F(\vec{v}_F-\vec{v}_T)\SCAL\GRAD\vec{w}\normal_{TF}.
\end{equation}
This velocity reconstruction will play a key role in the approximation of viscous terms, as detailed in Section \ref{sec:diff}.

\subsubsection{Divergence}

For any integer $\ell\ge 0$, a discrete divergence reconstruction $\DT[\ell]:\UT\to\Poly{\ell}(T)$ is obtained setting, for all $\uvec{v}_T\in\UT$,
$$
\DT[\ell]\uvec{v}_T\coloneq\tr(\GT[\ell]\uvec{v}_T).
$$
This discrete divergence will be used with $\ell=k$ in the pressure-velocity coupling (see Section \ref{sec:disc.div}) and with $\ell=2k$ to incorporate Temam's device for stability in the convection term (see Section \ref{sec:disc.prob:adv.rea}).
An immediate consequence of \eqref{eq:approx.GT} is that, for any $\vec{v}\in H^m(T)^d$ with $m$ as in that equation, it holds that
\begin{equation}\label{eq:approx.DT}
  \norm[L^2(T)]{\DT[\ell]\IT\vec{v}-\DIV\vec{v}}
  \lsim h_T^{m-1}\seminorm[H^m(T)^d]{\vec{v}}.
\end{equation}
For further use, we record the following characterisation of $\DT[\ell]$, obtained from \eqref{eq:GT} with $\matr{\tau}=q\Id$:
For all $\uvec{v}_T\in\UT$,
\begin{equation}\label{eq:DT}
  \int_T\DT[\ell]\uvec{v}_T~q
  =\int_T(\DIV\vec{v}_T) q
  + \sum_{F\in\Fh[T]}\int_F (\vec{v}_F-\vec{v}_T)\SCAL\normal_{TF}q
  \qquad\forall q\in\Poly{\ell}(T).
\end{equation}
Writing this formula for $\ell=2k$ and $\ell=k$ and letting, in both cases, $q$ span $\Poly{k}(T)$, it is inferred that, for any $\uvec{v}_T\in\UT$,
\begin{equation}\label{eq:DT2k.DTk}
  \lproj[T]{k}(\DT[2k]\uvec{v}_T) = \DT\uvec{v}_T.
\end{equation}
Another consequence of \eqref{eq:GT:commuting} together with the linearity of the $L^2$-orthogonal projector is the following relation:
For any $\vec{v}\in H^1(T)^d$,
\begin{equation}\label{eq:DT:commuting}
  \DT\IT\vec{v} = \lproj[T]{k}(\DIV\vec{v}).
\end{equation}

\subsubsection{Directional derivative}

For the discretisation of the convective term, we also need a reconstruction of directional derivatives inspired by \cite[Eq. (12)]{Di-Pietro.Droniou.ea:15}.
The main novelty is here that the advective velocity field is represented by a vector of discrete velocity unknowns $\uvec{w}_T\in\UT$ rather than a smooth field on $T$; see also \cite{Anderson.Droniou:18} for similar developments in the context of miscible displacements in porous media.
Specifically, the directional derivative reconstruction $\GwT{{\cdot}}:\UT\to\Poly{k}(T)^d$ is such that, for all $\uvec{v}_T\in\UT$,
\begin{equation}\label{eq:GwT}
  \int_T\GwT{\uvec{v}_T}\SCAL\vec{z}
  = \int_T(\vec{w}_T\SCAL\GRAD)\vec{v}_T\SCAL\vec{z}
  + \sum_{F\in\Fh[T]}\int_F(\vec{w}_F\SCAL\normal_{TF})(\vec{v}_F-\vec{v}_T)\SCAL\vec{z}
  \qquad\forall\vec{z}\in\Poly{k}(T)^d.
\end{equation}%
In the above expression, the role of the advective velocity inside the element and on its faces is played by $\vec{w}_T$ and $\vec{w}_F$, respectively.
For all $\vec{z}\in\Poly{k}(T)^d$, writing \eqref{eq:GT} for $\ell=2k$ and $\matr{\tau}=\vec{z}\otimes\vec{w}_T\coloneq(z_i w_{T,j})_{1\le i,j\le d}$ and comparing with \eqref{eq:GwT}, one can see that it holds
\begin{equation}\label{eq:GwT.GT}
  \int_T\GwT{\uvec{v}_T}\SCAL\vec{z}
  = \int_T(\vec{w}_T\SCAL\GT[2k])\uvec{v}_T\SCAL\vec{z}
  + \sum_{F\in\Fh[T]}\int_F(\vec{w}_F-\vec{w}_T)\SCAL\normal_{TF}(\vec{v}_F-\vec{v}_T)\SCAL\vec{z},
\end{equation}
where, recalling that $(\GT[2k]\uvec{v}_T)_{ij}$ approximates the partial derivative with respect to the $j$th space variable of the $i$th component of the function represented by $\uvec{v}_T$, we have set $(\vec{w}_T\SCAL\GT[2k])\uvec{v}_T\coloneq\left(\sum_{j=1}^d w_{T,j}(\GT[2k]\uvec{v}_T)_{ij}\right)_{i=1,\ldots,d}$. This shows that $\GwT{\uvec{v}_T}$ differs from $(\vec{w}_T\SCAL\GT[2k])\uvec{v}_T$ in that $\vec{w}_F$ replaces $\vec{w}_T$ in the boundary term.
The approximation properties of the discrete directional derivative relevant in the analysis are summarised in the following proposition.
\begin{proposition}[Approximation properties of the discrete directional derivative]\label{prop:GwT}
  It holds:
  For all $T\in\Th$, all $\vec{v}\in W^{1,4}(T)\cap H^{k+1}(T)^d$, and all $\vec{z}\in\Poly{k}(T)^d$,
  \begin{equation}\label{eq:approx.GwT}
    \left|
    \int_T\left(\GwT[\IT\vec{v}]{\IT\vec{v}}-(\vec{v}\SCAL\GRAD)\vec{v}\right)\SCAL\vec{z}
    \right|
    \lsim h_T^{k+1}\seminorm[H^{k+1}(T)^d]{\vec{v}}\seminorm[W^{1,4}(T)^d]{\vec{v}}\norm[L^4(T)^d]{\vec{z}}.
  \end{equation}
\end{proposition}
\begin{proof}
  See \ref{sec:convegence.analysis:preliminary.results:GwT}.
\end{proof}
\subsection{A discrete integration by parts formula}\label{sec:ibp}

In this section we prove a global discrete integration by parts formula which plays the role of \eqref{eq:ibp.cont} at the discrete level. It also establishes a link between the discrete directional derivative and divergence reconstructions.
As in the continuous setting, this formula plays a central role in proving a non-dissipation property for the convective trilinear form (see point (i) in Lemma \ref{lem:th}) and justifies the specific formulation adopted for Temam's device in \eqref{eq:tT} below.
\begin{proposition}[Discrete integration by parts formula]\label{prop:ibp}
  It holds, for all $\uvec{w}_h,\uvec{v}_h,\uvec{z}_h\in\Uh$,
  \begin{multline}\label{eq:ibp}
    \sum_{T\in\Th}\int_T\left(
      \GwT{\uvec{v}_T}\SCAL\vec{z}_T + \vec{v}_T\SCAL\GwT{\uvec{z}_T} + \DT[2k]\uvec{w}_T(\vec{v}_T\SCAL\vec{z}_T)
      \right)
    \\
    =
    -\sum_{T\in\Th}\sum_{F\in\Fh[T]}\int_F (\vec{w}_F\SCAL\normal_{TF})(\vec{v}_F-\vec{v}_T)\SCAL(\vec{z}_F-\vec{z}_T)
    +\sum_{F\in\Fhb}\int_F(\vec{w}_F\SCAL\normal_F)\vec{v}_F\SCAL\vec{z}_F.
  \end{multline}
\end{proposition}
\begin{remark}[Comparison with \eqref{eq:ibp.cont}]
  Compared with its continuous counterpart \eqref{eq:ibp.cont}, formula \eqref{eq:ibp} contains one additional term in the right-hand side where the differences between face and element unknowns in $\uvec{v}_h$ and $\uvec{z}_h$ appear. This term reflects the non-conformity of the HHO space.
\end{remark}
\begin{proof}[Proof of Proposition \ref{prop:ibp}]
  Let an element $T\in\Th$ be fixed.
  Expanding first $\GwT{\uvec{v}_T}$ according to its definition \eqref{eq:GwT} with $\vec{z}=\vec{z}_T$, then integrating by parts the first volumetric term, we obtain 
  \begin{equation}\label{eq:ibp:basic}
    \begin{aligned}
      \int_T \GwT{\uvec{v}_T}\SCAL\vec{z}_T
      &= \int_T(\vec{w}_T\SCAL\GRAD)\vec{v}_T\SCAL\vec{z}_T
      + \sum_{F\in\Fh[T]}\int_F(\vec{w}_F\SCAL\normal_{TF})(\vec{v}_F-\vec{v}_T)\SCAL\vec{z}_T
      \\
      &=
      - \int_T\vec{v}_T\SCAL(\vec{w}_T\SCAL\GRAD)\vec{z}_T
      - \int_T (\DIV\vec{w}_T)\vec{v}_T\SCAL\vec{z}_T
      \\
      &\quad + \sum_{F\in\Fh[T]}\int_F \left[
        (\vec{w}_F\SCAL\normal_{TF})(\vec{v}_F\SCAL\vec{z}_T)
        - (\vec{w}_F\SCAL\normal_{TF})(\vec{v}_T\SCAL\vec{z}_T)
        + (\vec{w}_T\SCAL\normal_{TF})(\vec{v}_T\SCAL\vec{z}_T)
        \right]
      \\
      &\eqcolon\term_1+\term_2+\term_3.
    \end{aligned}
  \end{equation}
  Using again \eqref{eq:GwT} this time with $\uvec{v}_T=\uvec{z}_T$ and $\vec{z}=\vec{v}_T$, we obtain for the first term
  \begin{equation}\label{eq:ibp:T1}
    \term_1 =
    - \int_T\GwT{\uvec{z}_T}\SCAL\vec{v}_T
    + \sum_{F\in\Fh[T]}\int_F(\vec{w}_F\SCAL\normal_{TF})(\vec{z}_F-\vec{z}_T)\SCAL\vec{v}_T.
  \end{equation}
  Invoking the characterisation \eqref{eq:DT} of the discrete divergence reconstruction with $\ell=2k$ and $q=\vec{v}_T\SCAL\vec{z}_T$, we get for the second term
  \begin{equation}\label{eq:ibp:T2}
    \term_2 = -\int_T\DT[2k]\uvec{w}_T(\vec{v}_T\SCAL\vec{z}_T)
    + \sum_{F\in\Fh[T]}\int_F (\vec{w}_F-\vec{w}_T)\SCAL\normal_{TF}(\vec{v}_T\SCAL\vec{z}_T).
  \end{equation}
  Plugging \eqref{eq:ibp:T1}--\eqref{eq:ibp:T2} into \eqref{eq:ibp:basic} and rearranging, we obtain
  \begin{multline*}
    \int_T\left(
    \GwT{\uvec{v}_T}\SCAL\vec{z}_T + \vec{v}_T\SCAL\GwT{\uvec{z}_T} + \DT[2k]\uvec{w}_T(\vec{v}_T\SCAL\vec{z}_T)
    \right)
    \\
    = \sum_{F\in\Fh[T]}\int_F(\vec{w}_F\SCAL\normal_{TF})(\vec{z}_F\SCAL\vec{v}_T - \vec{z}_T\SCAL\vec{v}_T+\vec{z}_T\SCAL\vec{v}_F).
  \end{multline*}
  Summing the above equality over $T\in\Th$ and adding the quantity
  \begin{equation}\label{eq:ibp:bnd}
    -\sum_{T\in\Th}\sum_{F\in\Fh[T]}\int_F(\vec{w}_F\SCAL\normal_{TF})(\vec{v}_F\SCAL\vec{z}_F)
    +\sum_{F\in\Fhb}\int_F(\vec{w}_F\SCAL\normal_F)(\vec{v}_F\SCAL\vec{z}_F) = 0,
  \end{equation}
  the conclusion follows after observing that $\vec{z}_F\SCAL\vec{v}_T - \vec{z}_T\SCAL\vec{v}_T+\vec{z}_T\SCAL\vec{v}_F-\vec{z}_F\SCAL\vec{v}_F=-(\vec{v}_F-\vec{v}_T)\SCAL(\vec{z}_F-\vec{z}_T)$.
  Formula \eqref{eq:ibp:bnd} is justified observing that, for any internal face $F\in\Fhi$ such that $F\in\Fh[T_1]\cap\Fh[T_2]$ for distinct mesh elements $T_1,T_2\in\Th$, it holds that $(\vec{w}_F\SCAL\normal_{T_1F})(\vec{v}_F\SCAL\vec{z}_F)+(\vec{w}_F\SCAL\normal_{T_2F})(\vec{v}_F\SCAL\vec{z}_F)=0$ owing to the single-valuedness of $\vec{w}_F$, $\vec{v}_F$, and $\vec{z}_F$ and the fact that $\normal{T_1F}+\normal_{T_2F}=\vec{0}$ by definition.
\end{proof}


\section{The Hybrid High-Order method}\label{sec:method}

In this section we discuss the discretisation of the viscous, pressure-velocity coupling, and convective terms. We then formulate the discrete problem.

\subsection{Viscous term}\label{sec:diff}

To discretise the viscous term, we introduce the bilinear form $\mathrm{a}_h:\Uh\times\Uh\to\Real$ such that
\begin{equation}\label{eq:ah}
	\begin{aligned}
  \mathrm{a}_h(\uvec{w}_h,\uvec{v}_h)
  \coloneq{}&\sum_{T\in\Th}\mathrm{a}_T(\uvec{w}_T,\uvec{v}_T)
  \\
  &+ \sum_{F\in\Fhb}  
  \int_F\left(
  -\GRAD\rT[T_F]\uvec{w}_{T_F}\normal_F\SCAL\vec{v}_F
  +\vec{w}_F\SCAL\GRAD\rT[T_F]\uvec{v}_{T_F}\normal_F
  + h_F^{-1}\vec{w}_F\SCAL\vec{v}_F
  \right),
	\end{aligned}
\end{equation}
where the terms in the second line account for the weakly enforced boundary conditions \emph{\`a la} Nitsche, while the element contribution $\mathrm{a}_T:\UT\times\UT\to\Real$ is such that
\begin{equation}\label{eq:aT}
  \mathrm{a}_T(\uvec{w}_T,\uvec{v}_T)
  \coloneq(\GRAD\rT\uvec{w}_T,\GRAD\rT\uvec{v}_T)_T + \mathrm{s}_T(\uvec{w}_T,\uvec{v}_T).
\end{equation}
In the right-hand side of \eqref{eq:aT}, the first term is the usual Galerkin contribution responsible for consistency, while $\mathrm{s}_T:\UT\times\UT\to\Real$  is the following stabilisation bilinear form that penalises the difference between the interpolate of the velocity reconstruction and the discrete unknowns:
\begin{equation}\label{eq:def.sT}
\mathrm{s}_T(\uvec{w}_T,\uvec{v}_T)\coloneq
\sum_{F\in\Fh[T]}
\frac{1}{h_F}\int_F(\vec{\delta}_{TF}^k\uvec{w}_T-\vec{\delta}_T^k\uvec{w}_T)\SCAL
(\vec{\delta}_{TF}^k\uvec{v}_T-\vec{\delta}_T^k\uvec{v}_T),
\end{equation}
where, for all $\uvec{v}_T\in\UT$, the difference operators are such that, recalling the definition \eqref{eq:IT} of $\IT$,
\begin{equation*}
  (\vec{\delta}_T^k\uvec{v}_T,(\vec{\delta}_{TF}^k\uvec{v}_T)_{F\in\Fh[T]})
  \coloneq\IT\rT\uvec{v}_T-\uvec{v}_T
  =\Big(\vlproj[T]{k}\rT\uvec{v}_T-\vec{v}_T,(\vlproj[F]{k}(\rT\uvec{v}_T)_{|F}-\vec{v}_F)_{F\in\Fh[T]}\Big).
\end{equation*}
The role of this stabilisation bilinear form is to ensure the following uniform local norm equivalence (see, e.g., \cite[Lemma 4]{Di-Pietro.Ern.ea:14}, where the scalar case is considered):
For all $T\in\Th$ and all $\uvec{v}_T\in\UT$,
\begin{equation}\label{eq:aT:stability}
  \mathrm{a}_T(\uvec{v}_T,\uvec{v}_T)
  \esim[C_{\mathrm{a}}^{-1}]\norm[1,T]{\uvec{v}_T}^2,
\end{equation}
with $\norm[1,T]{{\cdot}}$ defined by \eqref{eq:norm.1T}.
It holds (see, e.g., \cite[Proposition 3.1]{Di-Pietro.Tittarelli:18}):
For all $T\in\Th$ and all $\vec{v}\in H^{k+2}(T)^d$,
\begin{equation}\label{eq:sT:consistency}
  \mathrm{s}_T(\IT\vec{v},\IT\vec{v})^{\frac12}\lsim h_T^{k+1}\seminorm[H^{k+2}(T)^d]{\vec{v}}.
\end{equation}
The properties of the bilinear forms $\mathrm{a}_h$ relevant for the analysis are summarised in the following lemma.
\begin{lemma}[Properties of $\mathrm{a}_h$]\label{lem:ah}
  The bilinear form $\mathrm{a}_h$ has the following properties:
  \begin{enumerate}[(i)]
  \item \emph{Stability and boundedness.} It holds with $\norm[1,h]{{\cdot}}$ defined by \eqref{eq:norm.1h}:
    For all $\uvec{v}_h\in\Uh$,
    \begin{equation}\label{eq:ah:stability}
      \mathrm{a}(\uvec{v}_h,\uvec{v}_h)
      \esim[C_{\mathrm{a}}^{-1}]\norm[1,h]{\uvec{v}_h}^2.
    \end{equation}
  \item \emph{Consistency.} It holds:
    For all $\vec{w}\in H_0^1(\Omega)^d\cap H^{k+2}(\Th)^d$ such that $\LAPL\vec{w}\in L^2(\Omega)^d$,
    \begin{equation}\label{eq:ah:consistency}
      \sup_{\uvec{v}_h\in\Uh,\norm[1,h]{\uvec{v}_h}=1}\left|
      \int_\Omega\LAPL\vec{w}\SCAL\vec{v}_h + \mathrm{a}_h(\Ih\vec{w},\uvec{v}_h)
      \right|\lsim h^{k+1}\seminorm[H^{k+2}(\Th)^d]{\vec{w}}.
    \end{equation}
  \end{enumerate}
\end{lemma}
\begin{proof}
  See \ref{sec:convergence.analysis:preliminary.results:ah}.
\end{proof}
\begin{remark}[Variations for the Nitsche terms]\label{rem:weak.bc.variations}
  Several variations are possible for the terms responsible of the weak enforcement of boundary conditions in \eqref{eq:ah}.
  Specifically, a symmetric variation is obtained replacing inside the sum over $F\in\Fhb$ the term $\int_F\vec{w}_F\SCAL\GRAD\rT[T_F]\uvec{v}_{T_F}\normal_F$ by $-\int_F\vec{w}_F\SCAL\GRAD\rT[T_F]\uvec{v}_{T_F}\normal_F$.
  This term can also be removed altogether, leading to the so-called incomplete variation.
  One can also add a penalty coefficient $\eta>0$ in front of the penalty term, writing $\int_F\eta h_F^{-1}\vec{w}_F\SCAL\vec{v}_F$.
  The lower threshold for $\eta$ leading to a stable method depends on whether the skew-symmetric, symmetric or incomplete versions are considered.
  For the skew-symmetric version, in particular, stability is obtained provided $\eta>0$, and the expression \eqref{eq:ah} corresponds to $\eta=1$.
  The penalty parameter is present in our implementation, and values larger than 1 have sometimes been required in numerical experiments; see the comments in Section \ref{sec:tests:kovasznay}.
\end{remark}

\subsection{Pressure-velocity coupling}\label{sec:disc.div}

The pressure-velocity coupling hinges on the bilinear form $\mathrm{b}_h$ on $\Uh\times\Poly{k}(\Th)$ such that
\begin{equation}\label{eq:bh}
  \mathrm{b}_h(\uvec{v}_h,q_h)
  \coloneq -\sum_{T\in\Th}\int_T\DT\uvec{v}_T~q_T
  +\sum_{F\in\Fhb}\int_F(\vec{v}_F\SCAL\normal_F) q_{T_F},
\end{equation}
where $q_T\coloneq q_{h|T}$.
The second term in the right-hand side of \eqref{eq:bh} accounts for the weak enforcement of boundary conditions.
The relevant properties of $\mathrm{b}_h$ are summarised in the following lemma.
\begin{lemma}[Properties of $\mathrm{b}_h$]\label{lem:bh}
  The bilinear form $\mathrm{b}_h$ has the following properties:
  \begin{enumerate}[(i)]
  \item \emph{Consistency/1.} It holds, for all $\vec{v}\in H^1_0(\Omega)^d$,
    \begin{equation}\label{eq:bh:consistency.2}
      \mathrm{b}_h(\Ih\vec{v},q_h) = b(\vec{v},q_h)\qquad\forall q_h\in\Poly{k}(\Th).
    \end{equation}
  \item \emph{Stability.} It holds:
    For all $q_h\in\Ph\coloneq\Poly{k}(\Th)\cap L^2_0(\Omega)$,
    \begin{equation}\label{eq:bh:stability}
      \norm[L^2(\Omega)]{q_h}\lsim\sup_{\uvec{v}_h\in\Uh,\norm[1,h]{\uvec{v}_h}=1}\mathrm{b}_h(\uvec{v}_h,q_h).
    \end{equation}
  \item \emph{Consistency/2.} It holds:
    For all $q\in H^1(\Omega)\cap H^{k+1}(\Th)$,
    \begin{equation}\label{eq:bh:consistency}
      \sup_{\uvec{v}_h\in\Uh,\norm[1,h]{\uvec{v}_h}=1}\left|
      \int_\Omega\GRAD q\SCAL\vec{v}_h - \mathrm{b}_h(\uvec{v}_h,\lproj{k} q)
      \right|\lsim h^{k+1}\seminorm[H^{k+1}(\Th)]{q}.
    \end{equation}
  \end{enumerate}
\end{lemma}
\begin{proof}
  See \ref{sec:convergence.analysis:preliminary.results:bh}.
\end{proof}

\subsection{Convective term}\label{sec:disc.prob:adv.rea}

Let $\uvec{w}_h\in\Uh$.
To discretise the convective term, inspired by \eqref{eq:tilde.t}, we introduce the global trilinear form $\mathrm{t}_h$ on $\Uh\times\Uh\times\Uh$ such that
\begin{equation}\label{eq:th}
  \mathrm{t}_h(\uvec{w}_h,\uvec{v}_h,\uvec{z}_h)
  \coloneq\sum_{T\in\Th}\mathrm{t}_T(\uvec{w}_T,\uvec{v}_T,\uvec{z}_T)
  -\frac12\sum_{F\in\Fhb}\int_F(\vec{w}_F\SCAL\normal_F)(\vec{v}_F\SCAL\vec{z}_F),
\end{equation}
where the second term in the right-hand side accounts for weakly enforced boundary conditions while, for all $T\in\Th$, the local trilinear form $\mathrm{t}_T$ on $\UT\times\UT\times\UT$ is such that
\begin{multline}\label{eq:tT}
  \mathrm{t}_T(\uvec{w}_T,\uvec{v}_T,\uvec{z}_T)\coloneq
  \int_T\GwT{\uvec{v}_T}\SCAL\vec{z}_T
  + \frac{1}{2}\int_T\DT[2k]\uvec{w}_T(\vec{v}_T\SCAL\vec{z}_T)
  \\
  +\frac{1}{2}\sum_{F\in\Fh[T]}\int_F(\vec{w}_F\SCAL\normal_{TF})(\vec{v}_F-\vec{v}_T)\SCAL(\vec{z}_F-\vec{z}_T).
\end{multline}
The second and third terms embody Temam's device \cite{Temam:79}, and are crucial to obtain skew-symmetry and non-dissipation properties for $\mathrm{t}_h$ detailed in the following lemma.
\begin{lemma}[Properties of $\mathrm{t}_h$]\label{lem:th}
  The trilinear form $\mathrm{t}_h$ has the following properties:
  \begin{enumerate}[(i)]
  \item \emph{Skew-symmetry and non-dissipation.} For all $\uvec{w}_h,\uvec{v}_h,\uvec{z}_h\in\Uh$, it holds that
    \begin{equation}\label{eq:th:skew-symmetry}
      \mathrm{t}_h(\uvec{w}_h,\uvec{v}_h,\uvec{z}_h)
      =\frac12\sum_{T\in\Th}
      \int_T\left(
      \GwT{\uvec{v}_T}\SCAL\vec{z}_T - \vec{v}_T\SCAL\GwT{\uvec{z}_T}
      \right),
    \end{equation}
    so that, in particular, for all $\uvec{w}_h,\uvec{v}_h\in\Uh$,
    \begin{equation}\label{eq:th:non-dissipation}
      \mathrm{t}_h(\uvec{w}_h,\uvec{v}_h,\uvec{v}_h)=0.
    \end{equation}
  \item \emph{Boundedness.} It holds:
    For all $\uvec{w}_h,\uvec{v}_h,\uvec{z}_h\in\Uh$,
    \begin{equation}\label{eq:th:boundedness}
      \left| \mathrm{t}_h(\uvec{w}_h,\uvec{v}_h,\uvec{z}_h)\right|
      \lsim[C_{\mathrm{t}}]\norm[1,h]{\uvec{w}_h}\norm[1,h]{\uvec{v}_h}\norm[1,h]{\uvec{z}_h}.
    \end{equation}
  \item \emph{Consistency.} It holds:
    For all $\vec{w}\in H_0^1(\Omega)^d\cap W^{k+1,4}(\Th)^d$,
    \begin{multline}\label{eq:th:consistency}
      \sup_{\uvec{z}_h\in\Uh,\norm[1,h]{\uvec{z}_h}=1}\left|
      \int_\Omega\left(
      (\vec{w}\SCAL\GRAD)\vec{w}\SCAL\vec{z}_h
      +\frac12(\DIV\vec{w})\vec{w}\SCAL\vec{z}_h
      \right)
      - \mathrm{t}_h(\Ih\vec{w},\Ih\vec{w},\uvec{z}_h)
      \right|
      \\
      \lsim h^{k+1}
      \norm[W^{1,4}(\Omega)^d]{\vec{w}}
      \seminorm[W^{k+1,4}(\Th)^d]{\vec{w}}.
    \end{multline}
  \end{enumerate}
\end{lemma}
\begin{proof}
  See \ref{sec:convergence.analysis:preliminary.results:th}.
\end{proof}
\begin{remark}[Reformulation of $\mathrm{t}_h$]
  Expanding the discrete directional derivatives appearing in \eqref{eq:th:skew-symmetry} according to their definition \eqref{eq:GwT}, we arrive at the following reformulation of $\mathrm{t}_h$, which shows that, in the computer implementation, one does not need to actually compute $\GwT[\cdot\,]{\cdot}$ nor $\DT[2k]$:
  \begin{multline*}
  \mathrm{t}_h(\uvec{w}_h,\uvec{v}_h,\uvec{z}_h)\\
  =\frac12\sum_{T\in\Th}\left[
    \int_T(\vec{w}_T\SCAL\GRAD)\vec{v}_T\SCAL\vec{z}_T
    -\hspace{-0.5ex} \int_T\vec{v}_T\SCAL(\vec{w}_T\SCAL\GRAD)\vec{z}_T
    +\hspace{-1ex} \sum_{F\in\Fh[T]}\int_F(\vec{w}_F\SCAL\normal_{TF})\left(\vec{v}_F\SCAL\vec{z}_T - \vec{z}_F\SCAL\vec{v}_T\right)
    \right].
  \end{multline*}
\end{remark}

\subsection{Convective stabilisation}\label{sec:conv.stab}

When dealing with high-Reynolds flows, it is sometimes necessary to strengthen stability by penalising the difference between face and element unknowns.
Fix $\rho:\Real\rightarrow [0,\infty)$ a Lipschitz-continuous function and, for $T\in\Th$ and a given $\uvec{w}_T\in\UT$, define the local stabilisation bilinear form $\mathrm{j}_T(\uvec{w}_T;\cdot,\cdot):\UT\times\UT\to\Real$ by
\begin{equation}
  \label{eq:def.jT}
  \mathrm{j}_T(\uvec{w}_T;\uvec{v}_T,\uvec{z}_T)\coloneq
  \sum_{F\in\Fh[T]}\int_F\frac{\nu}{h_F}\rho(\Pe(\vec{w}_F))(\vec{v}_F-\vec{v}_T)\SCAL(\vec{z}_F-\vec{z}_T).
\end{equation}
Here, the local (oriented) Reynolds number $\Pe:\Poly{k}(F)^d\to\Real$ is such that, for all $\vec{w}\in\Poly{k}(F)^d$,
\begin{equation*}
  \Pe(\vec{w})\coloneq h_F \frac{\vec{w}\SCAL\normal_{TF}}{\nu}.
\end{equation*}
On boundary faces $F\in\Fhb$, we simply write $\Pe[F](\vec{w})$ instead of $\Pe[T_FF](\vec{w})$.
As already pointed out in \cite{Beirao-da-Veiga.Droniou.ea:10,Chainais-Hillairet.Droniou:11,Di-Pietro.Droniou.ea:15},
using the generic function $\rho$ in the definition of the convective stabilisation terms enables a unified treatment of several classical discretisations (in the notations of \cite{Beirao-da-Veiga.Droniou.ea:10}, $A(s)=\rho(s)+\frac12 s$ and $B(s)=-\rho(s)+\frac12 s$; in the notations of \cite{Di-Pietro.Droniou.ea:15}, $\rho=\frac12|A|$).
Specifically, the HHO version of classical convective stabilisations is obtained with the following choices of $\rho$:
\begin{itemize}
\item \emph{Centered scheme}: $\rho=0$.
\item \emph{Upwind scheme}: $\rho(s)=\frac{1}{2}|s|$. 
  In this case, the definition of $\mathrm{j}_T(\uvec{w}_T;\cdot,\cdot)$ simplifies to
  \begin{equation*}
    \label{eq:sT.vel}
    \mathrm{j}_T(\uvec{w}_T;\uvec{v}_T,\uvec{v}_T)\coloneq
    \sum_{F\in\Fh[T]}\int_F\frac{|\vec{w}_F\SCAL\normal_{TF}|}{2}(\vec{u}_F-\vec{u}_T)\SCAL(\vec{z}_F-\vec{z}_T).
  \end{equation*}
\item \emph{Locally upwinded $\theta$-scheme}: $\rho(s)=\frac{1}{2}(1-\theta(s))|s|$, where $\theta\in C^1_c(-1,1)$, $0\le \theta\le 1$ and
  $\theta\equiv 1$ on $[-\frac12,\frac12]$. This choice in  \eqref{eq:def.jT} corresponds to the centered scheme if $|\Pe(\vec{w}_F)|\le \frac12$ (dominating diffusion) and to the upwind scheme if $|\Pe(\vec{w}_F)|\ge 1$ (dominating advection).
\item \emph{Scharfetter--Gummel scheme}: $\rho(s)=\frac{s}2\coth(\frac{s}2)-1$.
\end{itemize}
The advantage of the locally upwinded $\theta$-scheme and the Scharfetter--Gummel scheme over the upwind scheme is that they behave as the centered scheme, and thus introduce less numerical diffusion, when $\Pe(\vec{w}_F)$ is close to zero (dominating viscosity). See, e.g., the discussion in \cite[Section 4.1]{Droniou:10} for the Scharfetter--Gummel scheme.

For $\uvec{w}_h\in\Uh$, the global stabilisation bilinear form $\mathrm{j}_h(\uvec{w}_h;\cdot,\cdot):\Uh\times \Uh\to \Real$ is obtained by assembling the local contributions and by adding a penalisation at the boundary to account for weakly enforced boundary conditions:
\begin{equation}\label{eq:jh}
  \mathrm{j}_h(\uvec{w}_h;\uvec{v}_h,\uvec{z}_h)\coloneq
  \sum_{T\in\Th}\mathrm{j}_T(\uvec{w}_T;\uvec{v}_T,\uvec{z}_T) + \sum_{F\in\Fhb} \int_F\frac{\nu}{h_F}\rho(\Pe[F](\vec{w}_F))\vec{v}_F\SCAL \vec{z}_F.
\end{equation}

\begin{lemma}[Properties of $\mathrm{j}_h$]\label{lem:prop.jT}
The stabilisation term $\mathrm{j}_h$ has the following properties:
\begin{enumerate}[(i)]
\item \emph{Continuity.} It holds, for all $\uvec{v}_h,\uvec{w}_h,\uvec{z}_h,\uvec{z}'_h\in\Uh$,
\begin{equation}\label{est:jh.cont}
\left|\mathrm{j}_h(\uvec{v}_h;\uvec{z}_h,\uvec{z}'_h)-\mathrm{j}_h(\uvec{w}_h;\uvec{z}_h,\uvec{z}'_h)\right|\lsim[C_{\mathrm{j}}]h^{1-\frac{d}{4}}\norm[1,h]{\uvec{v}_h-\uvec{w}_h}\norm[1,h]{\uvec{z}_h}\norm[1,h]{\uvec{z}'_h}.
\end{equation}

\item \emph{Consistency.} It holds:
  For all $\vec{w}\in W^{1,4}(\Omega)^d\cap W^{k+1,4}(\Th)^d$ and all $\uvec{z}_h\in\Uh$,
\begin{equation}\label{est:jh.const}
\sup_{\uvec{z}_h\in\Uh,\norm[1,h]{\uvec{z}_h}=1}|\mathrm{j}_h(\Ih \vec{w};\Ih\vec{w},\uvec{z}_h)|\lsim
h^{k+1}\norm[W^{1,4}(\Omega)^d]{\vec{w}}\seminorm[W^{k+1,4}(\Th)^d]{\vec{w}}.
\end{equation}
\end{enumerate}
\end{lemma}
\begin{proof}
  See \ref{sec:convergence.analysis:preliminary.results:jh}.
\end{proof}


\subsection{Discrete problem}\label{sec:discrete.problem}

Recall that we have set $\Ph\coloneq\Poly{k}(\Th)\cap L^2_0(\Omega)$ (see Lemma \ref{lem:bh}).
The HHO discretisation with weakly enforced boundary conditions of problem \eqref{eq:weak} reads:
Find $(\uvec{u}_h,p_h)\in\Uh\times\Ph$ such that
\begin{subequations}\label{eq:discrete}
  \begin{alignat}{2}\label{eq:discrete:momentum}
    \nu\mathrm{a}_h(\uvec{u}_h,\uvec{v}_h)
    + \mathrm{t}_{h}(\uvec{u}_h,\uvec{u}_h,\uvec{v}_h)
    + \mathrm{j}_h(\uvec{u}_h;\uvec{u}_h,\uvec{v}_h)
    + \mathrm{b}_h(\uvec{v}_h,p_h)
    &=\int_\Omega \vec{f}\SCAL\vec{v}_h&\qquad&\forall \uvec{v}_h\in \Uh
    \\ \label{eq:discrete:mass}
    -\mathrm{b}_h(\uvec{u}_h,q_h)&=0&\qquad&\forall q_h\in\Ph.
  \end{alignat}
\end{subequations}
As usual in HHO methods, boundary conditions can also be strongly enforced seeking the velocity approximation in the following subspace of $\Uh$:
\begin{equation}\label{eq:Uhz}
  \Uhz\coloneq\left\{
  \uvec{v}_h=((\vec{v}_T)_{T\in\Th},(\vec{v}_F)_{F\in\Fh})\in\Uh\st
  \vec{v}_F=\vec{0}\quad\forall F\in\Fhb
  \right\}.
\end{equation}
The HHO discretisation with strongly enforced boundary conditions of problem \eqref{eq:weak} then reads:
Find $(\uvec{u}_h,p_h)\in\Uhz\times\Ph$ such that
\begin{subequations}\label{eq:discrete.strong.bc}
  \begin{alignat}{2}\label{eq:discrete.strong.bc:momentum}
    \nu\mathrm{a}_h(\uvec{u}_h,\uvec{v}_h)
    + t_{h}(\uvec{u}_h,\uvec{u}_h,\uvec{v}_h)
    + \mathrm{j}_h(\uvec{u}_h;\uvec{u}_h,\uvec{v}_h)
    + \mathrm{b}_h(\uvec{v}_h,p_h)
    &=\int_\Omega \vec{f}\SCAL\vec{v}_h&\qquad&\forall \uvec{v}_h\in \Uhz
    \\ \label{eq:discrete.strong.bc:mass}
    -\mathrm{b}_h(\uvec{u}_h,q_h)&=0&\qquad&\forall q_h\in\Poly{k}(\Th).
  \end{alignat}
\end{subequations}
Some remarks are of order.
\begin{remark}[Simplifications]
  Since both the discrete velocity $\uvec{u}_h$ and the test function $\uvec{v}_h$ in \eqref{eq:discrete.strong.bc:momentum} are in $\Uhz$, the terms involving sums over $F\in\Fhb$ in the bilinear forms $\mathrm{a}_h$ and $\mathrm{b}_h$ (see, respectively, \eqref{eq:ah} and \eqref{eq:bh}), in the trilinear form $\mathrm{t}_h$ (see \eqref{eq:th}), and in the convective stabilisation term $\mathrm{j}_h$ (see \eqref{eq:jh}) vanish.
\end{remark}
\begin{remark}[Mass equation]\label{rem:mass.eq}
  For all $\uvec{v}_h\in\Uhz$, we have that
  $$
  \mathrm{b}_h(\uvec{v}_h,1)
  = -\sum_{T\in\Th}\int_T\DT\uvec{v}_T
  = -\sum_{T\in\Th}\sum_{F\in\Fh[T]}\int_F\vec{v}_F\SCAL\normal_{TF}
  = 0,
  $$
  where we have used the definition \eqref{eq:bh} of $\mathrm{b}_h$ and the strongly enforced boundary condition in the first equality,
  the relation \eqref{eq:DT} after integrating by parts the first term in the right-hand side in the second equality,
  and the single-valuedness of interface unknowns together with the strongly enforced boundary condition to conclude.
  As a consequence, \eqref{eq:discrete.strong.bc:mass} was written for any $q_h\in\Poly{k}(\Th)$, and not only $q_h\in\Ph$ as in \eqref{eq:discrete:mass}.
  This is a key point to prove the local mass balance \eqref{eq:discrete.strong.bc:mass.balance} below, which requires to take $q_h$ equal to the characteristic function of one element (which, of course, does not have zero average on $\Omega$).
\end{remark}
\begin{remark}[Incompressibility constraint]\label{rem:incompressibility}
  Equation \eqref{eq:discrete.strong.bc:mass} is equivalent to $\DT\uvec{u}_T=0$ for all $T\in\Th$, and expresses the fact that the HHO velocity field solution to \eqref{eq:discrete.strong.bc} is incompressible.
  Notice, however, that the fact that $\DT\uvec{u}_T=0$ for all $T\in\Th$ does not imply, in general, that $\DT[2k]\uvec{u}_T=0$, which justifies the introduction of the second term in the expression \eqref{eq:tT} of the local convective trilinear form.
\end{remark}

\subsection{Convergence analysis}

We investigate here the convergence of the method.
We focus on the version \eqref{eq:discrete} with weakly enforced boundary conditions.
The proofs carry out unchanged to the version \eqref{eq:discrete.strong.bc} with strongly enforced boundary conditions after accounting for the fact that face unknowns on the boundary vanish for vectors of discrete unknowns in $\Uhz$.
We estimate the error defined as the difference between the solution to the HHO scheme and the interpolant of the exact solution,
denoted for short by
\begin{equation}\label{def:hat.u.p}
(\hat{\uvec{u}}_h,\hat p_h)\coloneq (\Ih\vec{u},\lproj{k} p)\in\Uhz\times P_h^k.
\end{equation}
As usual for the Navier--Stokes equations, the error estimate is obtained under a smallness
assumption on the data. To specify this smallness assumption, we denote by $C_P$ a Poincar\'e constant in $H^1_0(\Omega)^d$ and, using e.g.~\cite[Proposition 7.1]{Di-Pietro.Droniou:17}, we take $C_I$ such that, for all $\vec{w}\in H^1_0(\Omega)^d$, 
\begin{equation}\label{est:CI}
	\norm[1,h]{\Ih \vec{w}}\le C_I\norm[H^1_0(\Omega)^d]{\vec{w}}.
\end{equation}

\begin{theorem}[Discrete error estimate for small data]\label{th:error.estimates}
Assume that the forcing term $\vec{f}$ satisfies, for some $\alpha \in (0,1)$,
\begin{equation}\label{f:small}
	\norm[L^2(\Omega)^d]{\vec{f}}\le\alpha\frac{\nu^2 C_{\mathrm{a}}}{\left(C_{\mathrm{t}} +C_{\mathrm{j}} h^{1-\frac{d}{4}}\right)C_IC_P},
\end{equation}
where $C_{\mathrm{a}}$, $C_{\mathrm{t}}$ and $C_{\mathrm{j}}$ are defined in \eqref{eq:ah:stability}, \eqref{eq:th:boundedness} and \eqref{est:jh.cont}, respectively.
Let $(\vec{u},p)\in H^1_0(\Omega)^d\times L^2_0(\Omega)$ be a solution to the Navier--Stokes equations \eqref{eq:weak} and $(\uvec{u},p_h)\in \Uh\times P_h^k$ be a solution to the HHO scheme \eqref{eq:discrete} with weakly enforced boundary conditions. Assume that $\vec{u}\in W^{k+1,4}(\Th)^d\cap H^{k+2}(\Th)^d$ and that $p\in H^1(\Omega)\cap H^{k+1}(\Th)$, and let $(\hat{\uvec{u}}_h,\hat p_h)$ be defined by \eqref{def:hat.u.p}.
Then, it holds with hidden constant independent of $h$, $\nu$ and $\alpha$,
\begin{multline}\label{error:discrete}
\norm[1,h]{\uvec{u}_h-\hat{\uvec{u}}_h}+(\nu+1)^{-1}\norm[L^2(\Omega)]{p_h-\hat{p}_h}\\
\lsim (1-\alpha)^{-1}\nu^{-1} h^{k+1}\left(
\nu \seminorm[H^{k+2}(\Th)^d]{\vec{u}}
+\norm[W^{1,4}(\Omega)^d]{\vec{u}}\seminorm[W^{k+1,4}(\Th)^d]{\vec{u}}
+\seminorm[H^{k+1}(\Th)]{p}\right).
\end{multline}
\end{theorem}

\begin{remark}[Rates of convergence and convergence without regularity assumption]
As in \cite[Corollary 16]{Di-Pietro.Krell:18}, one can deduce from \eqref{error:discrete} the estimate
\begin{multline*}
\norm[L^2(\Omega)^{d\times d}]{\GRADh\rh\uvec{u}_h-\GRAD\vec{u}}+(\nu+1)^{-1}\norm[L^2(\Omega)]{p_h-p}\\
\lsim
(1-\alpha)^{-1}\nu^{-1} h^{k+1}\left(\nu \seminorm[H^{k+2}(\Th)^d]{\vec{u}}
+\norm[W^{1,4}(\Omega)^d]{\vec{u}}\seminorm[W^{k+1,4}(\Th)^d]{\vec{u}}
+\seminorm[H^{k+1}(\Th)]{p}\right),
\end{multline*}
where $\GRADh$ denotes the broken gradient on $\Th$ and $\rh\uvec{u}_h$ is defined by patching the local velocity reconstructions: $(\rh\uvec{u}_h)_{|T}\coloneq\rT\uvec{u}_T$ for all $T\in\Th$.
Also following the ideas in \cite{Di-Pietro.Krell:18}, we could prove the convergence of the
solution to the HHO scheme towards the solution to \eqref{eq:weak} without requiring any smallness assumption on $\vec{f}$ or any regularity property on the solution other than $(\vec{u},p)\in H^1_0(\Omega)^d\times L^2_0(\Omega)$ (see \cite[Theorem 14]{Di-Pietro.Krell:18}). 
\end{remark}

\begin{proof}[Proof of Theorem \ref{th:error.estimates}]
  (i) \emph{Estimate on the velocity.}
  Set
  \begin{equation}\label{eq:errors}
    (\uvec{e}_h,\epsilon_h)\coloneq (\uvec{u}_h - \hat{\uvec{u}}_h, p_h - \hat{p}_h).
  \end{equation}
Defining the consistency error $\mathcal E_h:\Uh\to\Real$ such that
\[
\mathcal E_h(\uvec{v}_h)\coloneq \int_\Omega \vec{f}\SCAL\vec{v}_h-\nu\mathrm{a}_h(\hat{\uvec{u}}_h,\uvec{v}_h)-\mathrm{t}_h(\hat{\uvec{u}}_h,\hat{\uvec{u}}_h,\uvec{v}_h)-\mathrm{j}_h(\hat{\uvec{u}}_h;\hat{\uvec{u}}_h,\uvec{v}_h)-\mathrm{b}_h(\uvec{v}_h,\hat{p}_h),
\]
we have, substituting $\int_\Omega \vec{f}\SCAL\vec{v}_h$ in \eqref{eq:discrete:momentum},
\begin{multline}
\nu \mathrm{a}_h(\uvec{e}_h,\uvec{v}_h) + \mathrm{t}_h(\uvec{u}_h,\uvec{u}_h,\uvec{v}_h)-
\mathrm{t}_h(\hat{\uvec{u}}_h,\hat{\uvec{u}}_h,\uvec{v}_h)\\
+\mathrm{j}_h(\uvec{u}_h;\uvec{u}_h,\uvec{v}_h)-\mathrm{j}_h(\hat{\uvec{u}}_h;\hat{\uvec{u}}_h,\uvec{v}_h)
+\mathrm{b}_h(\uvec{v}_h,\epsilon_h)=\mathcal E_h(\uvec{v}_h).
\label{eq:error}\end{multline}
Make $\uvec{v}_h=\uvec{e}_h$.
The skew-symmetry property \eqref{eq:th:non-dissipation} of $\mathrm{t}_h$ together with linearity in its second argument yield
$0=\mathrm{t}_h(\uvec{u}_h,\uvec{e}_h,\uvec{e}_h)=\mathrm{t}_h(\uvec{u}_h,\uvec{u}_h,\uvec{e}_h)
-\mathrm{t}_h(\uvec{u}_h,\hat{\uvec{u}}_h,\uvec{e}_h)$ and thus, by the boundedness \eqref{eq:th:boundedness}, 
\[
\left|\mathrm{t}_h(\uvec{u}_h,\uvec{u}_h,\uvec{e}_h)-
\mathrm{t}_h(\hat{\uvec{u}}_h,\hat{\uvec{u}}_h,\uvec{e}_h)\right|=
\left|\mathrm{t}_h(\uvec{e}_h,\hat{\uvec{u}}_h,\uvec{e}_h)\right|\le
C_{\mathrm{t}} \norm[1,h]{\uvec{e}_h}^2\norm[1,h]{\hat{\uvec{u}}_h}.
\]
Moreover, by the continuity property \eqref{est:jh.cont} of $\mathrm{j}_h$ and the positivity of $\mathrm{j}_h(\uvec{u}_h;\uvec{e}_h,\uvec{e}_h)$,
\begin{align}
\mathrm{j}_h(\uvec{u}_h;\uvec{u}_h,\uvec{e}_h)-\mathrm{j}_h(\hat{\uvec{u}}_h;\hat{\uvec{u}}_h,\uvec{e}_h)={}&
\mathrm{j}_h(\uvec{u}_h;\uvec{e}_h,\uvec{e}_h)+\mathrm{j}_h(\uvec{u}_h;\hat{\uvec{u}}_h,\uvec{e}_h)-
\mathrm{j}_h(\hat{\uvec{u}}_h;\hat{\uvec{u}}_h,\uvec{e}_h)\nonumber\\
\ge{}&
-C_{\mathrm{j}} h^{1-\frac{d}{4}}\norm[1,h]{\uvec{u}_h-\hat{\uvec{u}}_h}\norm[1,h]{\hat{\uvec{u}}_h}
\norm[1,h]{\uvec{e}_h}.
\label{eq:est.jh}\end{align}
Finally, by \eqref{eq:discrete:mass}, \eqref{eq:bh:consistency.2} and \eqref{eq:weak:mass}, we have
$$
\mathrm{b}_h(\vec{e}_h,\epsilon_h)
=\mathrm{b}_h(\uvec{u}_h,\epsilon_h)-\mathrm{b}_h(\hat{\uvec{u}}_h,\epsilon_h)
=\mathrm{b}_h(\uvec{u}_h,\epsilon_h)-b(\vec{u},\epsilon_h)
=0.
$$
Hence, coming back to \eqref{eq:error} with $\uvec{v}_h=\uvec{e}_h$ and using the stability \eqref{eq:ah:stability} of $\mathrm{a}_h$,
\begin{align}
	\mathcal E_h(\uvec{e}_h)\ge{}& \left[\nu C_{\mathrm{a}}-C_{\mathrm{t}} \norm[1,h]{\hat{\uvec{u}}_h}-C_{\mathrm{j}} h^{1-\frac{d}{4}}\norm[1,h]{\hat{\uvec{u}}_h}\right]\norm[1,h]{\uvec{e}_h}^2\nonumber\\
	\ge{}& \left[\nu C_{\mathrm{a}}-\left(C_{\mathrm{t}} +C_{\mathrm{j}} h^{1-\frac{d}{4}}\right)C_I\nu^{-1}C_P\norm[L^2(\Omega)^d]{\vec{f}}\right]\norm[1,h]{\uvec{e}_h}^2\nonumber\\
	\ge{}& (1-\alpha)\nu C_{\mathrm{a}}\norm[1,h]{\uvec{e}_h}^2,
\label{Eh.coer}
\end{align} 
where we have used the definition \eqref{est:CI} of $C_I$, the basic estimate $\norm[H^1_0(\Omega)^d]{\vec{u}}\le
\nu^{-1}C_P\norm[L^2(\Omega)^d]{\vec{f}}$ (obtained by using $\vec{v}=\vec{u}$ in the weak formulation \eqref{eq:weak}), and \eqref{f:small}.

To estimate $\mathcal E_h(\vec{e}_h)$ from above, we recall that $\vec{f}=-\nu\Delta \vec{u} + (\vec{u}\SCAL\GRAD)\vec{u}+\GRAD p$, so that, for any $\uvec{v}_h\in \Uh$,
\begin{align*}
\mathcal E_h(\vec{v}_h)={}&-\nu\left(\int_\Omega \Delta\vec{u}\SCAL\vec{v}_h+\mathrm{a}_h(\hat{\uvec{u}}_h,\uvec{v}_h)\right)+\left(\int_\Omega (\vec{u}\SCAL\GRAD)\vec{u}\SCAL \vec{v}_h-\mathrm{t}_h(\hat{\uvec{u}}_h,\hat{\uvec{u}}_h,\uvec{v}_h)\right)\\
&-\mathrm{j}_h(\hat{\uvec{u}}_h;\hat{\uvec{u}}_h,\uvec{v}_h)+\left(\int_\Omega \GRAD p\SCAL\vec{v}_h
-\mathrm{b}_h(\uvec{v}_h,\hat{p}_h)\right).
\end{align*}
Using the consistency estimates \eqref{eq:ah:consistency}, \eqref{eq:th:consistency} (together with $\DIV\vec{u}=0$), \eqref{est:jh.const} and \eqref{eq:bh:consistency}, all the elementary components of $\mathcal E_h(\uvec{v}_h)$ can be estimated and we find
\begin{equation}\label{eq:est.Eh}
  |\mathcal E_h(\uvec{v}_h)|\lsim h^{k+1}\left(\nu \seminorm[H^{k+2}(\Th)^d]{\vec{u}}
  +\norm[W^{1,4}(\Omega)^d]{\vec{u}}\seminorm[W^{k+1,4}(\Th)^d]{\vec{u}}
  +\seminorm[H^{k+1}(\Th)]{p}\right)\norm[1,h]{\uvec{v}_h}.
\end{equation}
Applied to $\uvec{v}_h=\uvec{e}_h$ and combined with \eqref{Eh.coer}, this proves the estimate on the velocity error in \eqref{error:discrete}.
\medskip\\
(ii) \emph{Estimate on the pressure.}
To estimate the error on the pressure, we start from the stability property \eqref{eq:bh:stability} of $\mathrm{b}_h$ and we use the error equation \eqref{eq:error} to write
\begin{align}
	\norm[L^2(\Omega)]{\epsilon_h}\lsim{}& \sup_{\uvec{v}_h\in\Uh,\,\norm[1,h]{\uvec{v}_h}=1}\mathrm{b}_h(\uvec{v}_h,\epsilon_h)\nonumber\\
	={}&\sup_{\uvec{v}_h\in\Uh,\,\norm[1,h]{\uvec{v}_h}=1}\Big(\mathcal E_h(\uvec{v}_h)-\nu \mathrm{a}_h(\uvec{e}_h,\uvec{v}_h) - \mathrm{t}_h(\uvec{u}_h,\uvec{u}_h,\uvec{v}_h)+\mathrm{t}_h(\hat{\uvec{u}}_h,\hat{\uvec{u}}_h,\uvec{v}_h)\\
	&\qquad\qquad \qquad\qquad -\mathrm{j}_h(\uvec{u}_h;\uvec{u}_h,\uvec{v}_h)+\mathrm{j}_h(\hat{\uvec{u}}_h;\hat{\uvec{u}}_h,\uvec{v}_h)\Big).
\label{est:pressure.0}
\end{align}
By trilinearity of $\mathrm{t}_h$, 
\begin{equation}\label{est:press.1}
	\mathrm{t}_h(\uvec{u}_h,\uvec{u}_h,\uvec{v}_h)-\mathrm{t}_h(\hat{\uvec{u}}_h,\hat{\uvec{u}}_h,\uvec{v}_h)=\mathrm{t}_h(\uvec{e}_h,\uvec{u}_h,\uvec{v}_h)+\mathrm{t}_h(\hat{\uvec{u}}_h,\uvec{e}_h,\uvec{v}_h).
\end{equation}
By bilinearity of $\mathrm{j}_h(\uvec{u}_h;\cdot,\cdot)$,
\begin{equation}\label{est:press.2}
\mathrm{j}_h(\hat{\uvec{u}}_h;\hat{\uvec{u}}_h,\uvec{v}_h)-\mathrm{j}_h(\uvec{u}_h;\uvec{u}_h,\uvec{v}_h)
=
\mathrm{j}_h(\hat{\uvec{u}}_h;\hat{\uvec{u}}_h,\uvec{v}_h)-\mathrm{j}_h(\uvec{u}_h;\hat{\uvec{u}}_h,\uvec{v}_h)
-\mathrm{j}_h(\uvec{u}_h;\uvec{e}_h,\uvec{v}_h).
\end{equation}
Hence, the estimate \eqref{eq:est.Eh} on $\mathcal E_h(\uvec{v}_h)$ and the boundedness properties \eqref{eq:ah:stability} of $\mathrm{a}_h$ (together with a Cauchy--Schwarz inequality) and \eqref{eq:th:boundedness} of $\mathrm{t}_h$, along with the continuity estimate \eqref{est:jh.cont} on $\mathrm{j}_h$ (with $(\uvec{v}_h,\uvec{w}_h,\uvec{z}_h,\uvec{z}'_h)=(\hat{\uvec{u}}_h,\uvec{u}_h,\hat{\uvec{u}}_h,\uvec{v}_h)$ and $(\uvec{v}_h,\uvec{w}_h,\uvec{z}_h,\uvec{z}'_h)=(\uvec{u}_h,\uvec{0},\uvec{e}_h,\uvec{v}_h)$) yield
\begin{align*}
  \norm[L^2(\Omega)]{\epsilon_h}\lsim{}& h^{k+1}\left(\nu \seminorm[H^{k+2}(\Th)^d]{\vec{u}}
  +\norm[W^{1,4}(\Omega)^d]{\vec{u}}\seminorm[W^{k+1,4}(\Th)^d]{\vec{u}}
  +\seminorm[H^{k+1}(\Th)]{p}\right)\\
  &+\nu \norm[1,h]{\uvec{e}_h}
  + (1+h^{1-\frac{d}{4}})
  \norm[1,h]{\uvec{e}_h}\left(\norm[1,h]{\uvec{u}_h}+\norm[1,h]{\hat{\uvec{u}}_h}\right).
\end{align*}
The proof of the estimate on $\norm[L^2(\Omega)]{p_h-\hat{p}_h}$ is completed by plugging into the expression above the estimate on $\norm[1,h]{\uvec{e}_h}$ already established in \eqref{error:discrete}, by using the definition \eqref{est:CI} of $C_I$, and by writing $h\le {\rm diam}(\Omega)$ and $\norm[1,h]{\uvec{u}_h}\lsim \nu^{-1}\norm[L^2(\Omega)^d]{\vec{f}}$ (obtained using $\uvec{v}_h=\uvec{u}_h$ in \eqref{eq:discrete:momentum} and invoking \eqref{eq:ah:stability}, \eqref{eq:th:non-dissipation}, $\mathrm{j}_h(\uvec{u}_h;\uvec{u}_h,\uvec{u}_h)\ge 0$, and \eqref{eq:sobolev.embeddings} with $q=2$).
\end{proof}

\subsection{Links with other methods}

As noticed in \cite[Section 2.5]{Di-Pietro.Ern.ea:14} for pure diffusion equations, the lowest order HHO method (that is, the discretisation of the diffusive bilinear form presented in Section \ref{sec:diff} with $k{=}0$) is a particular case of the Hybrid Finite Volume method, which is itself a specific instance of the Hybrid Mixed Mimetic (HMM) method \cite{Droniou.Eymard.ea:10}. This comparison was extended in \cite[Section 5.4]{Di-Pietro.Droniou.ea:15} to advection--diffusion--reaction equations, where it was shown that, with a convective term discretised using the directional derivative \eqref{eq:GwT} (with a known velocity field instead of $\uvec{w}_T$) and a stabilisation term as in Section \ref{sec:conv.stab}, the HHO method for $k{=}0$ corresponds to the HMM discretisation of advection--diffusion--reaction equations presented in \cite{Beirao-da-Veiga.Droniou.ea:10}.

An initial version of the HMM method, in its Mixed Finite Volume form, was applied to the Navier--Stokes equations in \cite{Droniou.Eymard:09}, with upwinding of the convective term, albeit using a stabilisation of the diffusion that did not exactly reproduce linear solutions (contrary to $\mathrm{s}_T$ defined by \eqref{eq:def.sT} for $k{=}0$). This scheme was modified in \cite{Droniou:10} to use the standard HMM stabilisation that includes \eqref{eq:def.sT} (with $k{=}0$) as a specific case, and to include all the various discretisations of the convective term as in  \cite{Beirao-da-Veiga.Droniou.ea:10} or Section \ref{sec:conv.stab}. The HHO method we present here can therefore be considered as a higher-order extension of the HMM scheme for Navier--Stokes as in \cite{Droniou.Eymard:09,Droniou:10}

As shown in \cite[Section 4]{Di-Pietro.Droniou.ea:18}, the HHO method for diffusion equations is very close to the non-conforming Virtual Element Method (VEM) \cite{Ayuso-de-Dios.Lipnikov.ea:16} and the high-order mixed Mimetic Finite Difference method \cite{Lipnikov.Manzini:14}. The nonconforming VEM has been proposed and analysed for the Stokes equations in \cite{Cangiani.Gyrya.ea:16}, with strongly enforced boundary conditions.
The HHO methods of Section \ref{sec:discrete.problem} differs from the one presented in this reference, among others, in the choice of the degree for element-based unknowns ($k$ instead of $k-1$), which is crucial for the optimal convergence of the convective term appearing in the complete Navier--Stokes system. Another relevant difference is the possibility to weakly enforce boundary conditions.


\section{Numerical tests}\label{sec:tests}

This section contains an extensive numerical validation of the proposed HHO methods.
All the steady-state computations presented hereafter are performed by 
means of the pseudo-transient-continuation algorithm analyzed by \cite{Kelley.Keyes:98}
employing the Selective Evolution Relaxation (SER) strategy \cite{Mulder.Van-Leer:85} for evolving 
the pseudo-time step according to the Newton's equations residual.
Convergence to steady-state is achieved when the Euclidean norm of the momentum equation residual 
drops below $10^{-12}$.
At each pseudo-time step, the linearised equations are exactly solved
by means of the direct solver Pardiso \cite{Schenk.Gartner.ea:01}, distributed 
as part of the Intel Math Kernel Library (Intel MKL). 
Accordingly, the Euclidean norm of the continuity equation residual 
is comparable to the machine epsilon at all pseudo-time steps. 
The code implementation makes extensive use of the linear 
algebra Eigen open-source library \cite{Guennebaud.Jacob.ea:10}.

\subsection{Kovasznay flow}\label{sec:tests:kovasznay}

We start by assessing the convergence properties of the method using the analytical solution of \cite{Kovasznay:48}.
Specifically, in dimension $d=2$ we solve on the square domain $\Omega\coloneq(-0.5,1.5)\times(0,2)$ 
the Dirichlet problem corresponding to the exact solution $(\vec{u},p)$ such that, 
introducing the global Reynolds number $\Reynolds\coloneq \nu^{-1}$ and letting $\lambda\coloneq\frac{\Reynolds}{2}-\left(\frac{\Reynolds^2}{4} + 4\pi^2\right)^{\frac12}$, the velocity components are given by
$$
u_1(\vec{x}) \coloneq 1-\exp(\lambda x_1)\cos(2\pi x_2),\qquad
u_2(\vec{x}) \coloneq \frac{\lambda}{2\pi}\exp(\lambda x_1)\sin(2\pi x_2),
$$
while the pressure is given by
$$
p(\vec{x}) \coloneq -\frac12\exp(2\lambda x_1) + \frac{\lambda}{2}\left(\exp(4\lambda)-1\right).
$$
We take here $\nu=0.025$, corresponding to $\Reynolds=40$.

We consider computations with polynomial degrees $k\in\{0,\ldots,5\}$ over a sequence of uniformly $h$-refined Cartesian grids 
having $2^i$ ($i=2,3,\ldots,7$) elements in each direction.
We report in Tables \ref{tab:test:kovasznay.strong} and \ref{tab:test:kovasznay.weak}, respectively, the results for the methods \eqref{eq:discrete.strong.bc} 
with upwind stabilisation and \eqref{eq:discrete} without convective stabilisation.
While any other combination is possible, this setup is preferred here since we noticed that adding the convective stabilisation term with weakly enforced boundary conditions
can require in some cases to introduce a penalty parameter larger than 1 in the last term of \eqref{eq:ah}; see Remark \ref{rem:weak.bc.variations} on this point.
We have observed that this can also be avoided by simply removing convective stabilisation on the boundary from \eqref{eq:jh}.

The following quantities are monitored (see \eqref{eq:errors} for the definition of the errors):
the energy norm of the error on the velocity $\norm[\nu,h]{\uvec{e}_h}\coloneq\left(\nu\mathrm{a}_h(\uvec{e}_h,\uvec{e}_h)\right)^{\frac12}$,
the $L^2$-error on the velocity $\norm[L^2(\Omega)^d]{\vec{e}_h}$,
the $L^2$-error on the pressure $\norm[L^2(\Omega)]{\epsilon_h}$,
and the assembly and solution times (respectively denoted by $\tau_{\rm ass}$ and $\tau_{\rm sol}$) running a serial version of the code on a 2017 quad-core CPU laptop.
The \emph{assembly time} takes into account:
(i) the element-by-element computation of bilinear and trilinear forms;
(ii) the element-by-element static condensation;
(iii) the assembly of the statically condensed matrix blocks into the global matrix;
(iv) the introduction of a Lagrange multiplier to fix the mean value of pressure over $\Omega$;
(v) in case of strong boundary conditions, the elimination of boundary face unknowns from the global matrix.
We remark that the assembly of the trilinear form requires to revert static condensation in order to
compute the velocity solution over mesh elements. This \emph{back solve} post-processing can be
performed element-by-element, but its computational expense is comparable to that of matrix assembly: 
indeed, all bilinear and trilinear forms need to be recomputed. 
To avoid incurring this cost, in our implementation the back solve is performed element-by-element 
during the matrix assembly, meaning that the bilinear and trilinear forms are computed once and twice, respectively. 
The \emph{solution time} refers to the wall-clock-cpu time required by the direct solver to perform the LU factorization of 
the global system matrix and compute the solution increment to update the globally coupled unknowns.

Denoting by $e_i$ and $h_i$, respectively, the error in a given norm and the meshsize corresponding 
to a refinement iteration $i$, the estimated order of convergence (EOC) is obtained according to the following formula:
$$
{\rm EOC} = \frac{\log e_i - \log e_{i+1}}{\log h_i - \log h_{i+1}}.
$$
Besides discretization errors, EOCs and computation times, in Tables \ref{tab:test:kovasznay.strong} and \ref{tab:test:kovasznay.weak} 
we also report the size of the statically condensed global system matrix ($N_{\rm dof}$) and its number of non-zero entries ($N_{\rm nz}$).

\begin{table}
  \caption{Convergence results for the Kovasznay problem at $\Reynolds=40$ with strongly enforced boundary conditions and convective term stabilisation. \label{tab:test:kovasznay.strong}}
  \begin{footnotesize}
    \begin{tabular}{cccccccccc}
      \toprule
      $N_{\rm dof}$  & $N_{\rm nz}$ & $\norm[\nu,h]{\uvec{e}_h}$ & EOC & $\norm[L^2(\Omega)^d]{\vec{e}_h}$ & EOC & $\norm[L^2(\Omega)]{\epsilon_h}$ & EOC & $\tau_{\rm ass}$ & $\tau_{\rm sol}$ \\
      \midrule
      \multicolumn{10}{c}{$k{=}0$} \\
      \midrule
      65         & 736        & 9.37e-01   & --         & 1.40e-01   & --         & 6.84e-01   & --         & 1.31e-02   & 8.52e-03   \\ 
      289        & 3808       & 1.13e+00   & -0.27      & 5.50e-01   & -1.98      & 1.96e-01   & 1.80       & 5.92e-02   & 4.90e-02   \\ 
      1217       & 17056      & 9.14e-01   & 0.31       & 2.26e-01   & 1.28       & 1.02e-01   & 0.94       & 1.02e-01   & 1.06e-01   \\ 
      4993       & 71968      & 6.26e-01   & 0.55       & 7.89e-02   & 1.52       & 3.52e-02   & 1.54       & 3.10e-01   & 4.46e-01   \\ 
      20225      & 295456     & 3.87e-01   & 0.70       & 2.47e-02   & 1.68       & 9.78e-03   & 1.85       & 1.02e+00   & 2.17e+00   \\ 
      81409      & 1197088    & 2.47e-01   & 0.65       & 8.06e-03   & 1.61       & 3.09e-03   & 1.66       & 3.73e+00   & 1.49e+01   \\ 
      \midrule
      \multicolumn{10}{c}{$k{=}1$} \\
      \midrule
      113        & 2464       & 7.31e-01   & --         & 5.37e-01   & --         & 2.49e-01   & --         & 2.51e-02   & 1.72e-02   \\ 
      513        & 13056      & 3.83e-01   & 0.93       & 1.54e-01   & 1.80       & 4.29e-02   & 2.54       & 4.77e-02   & 4.72e-02   \\ 
      2177       & 59008      & 1.02e-01   & 1.90       & 2.13e-02   & 2.85       & 3.98e-03   & 3.43       & 1.29e-01   & 1.79e-01   \\ 
      8961       & 249984     & 2.93e-02   & 1.80       & 2.97e-03   & 2.84       & 6.54e-04   & 2.61       & 5.13e-01   & 1.01e+00   \\ 
      36353      & 1028224    & 8.23e-03   & 1.83       & 3.99e-04   & 2.90       & 1.28e-04   & 2.35       & 2.05e+00   & 5.28e+00   \\ 
      146433     & 4169856    & 2.26e-03   & 1.86       & 5.21e-05   & 2.94       & 2.65e-05   & 2.27       & 7.25e+00   & 2.97e+01   \\ 
      \midrule
      \multicolumn{10}{c}{$k{=}2$} \\
      \midrule
      161        & 5216       & 3.50e-01   & --         & 2.09e-01   & --         & 6.42e-02   & --         & 3.44e-02   & 2.26e-02   \\ 
      737        & 27872      & 3.76e-02   & 3.22       & 1.34e-02   & 3.96       & 2.07e-03   & 4.95       & 6.95e-02   & 6.88e-02   \\ 
      3137       & 126368     & 6.96e-03   & 2.43       & 1.31e-03   & 3.36       & 1.48e-04   & 3.80       & 2.66e-01   & 3.60e-01   \\ 
      12929      & 536096     & 1.06e-03   & 2.72       & 9.48e-05   & 3.79       & 1.77e-05   & 3.07       & 1.11e+00   & 2.02e+00   \\ 
      52481      & 2206496    & 1.55e-04   & 2.77       & 6.36e-06   & 3.90       & 2.27e-06   & 2.96       & 4.16e+00   & 1.13e+01   \\ 
      211457     & 8951072    & 2.21e-05   & 2.81       & 4.13e-07   & 3.95       & 2.72e-07   & 3.06       & 1.51e+01   & 6.02e+01   \\ 
      \midrule
      \multicolumn{10}{c}{$k{=}3$} \\
      \midrule
      209        & 8992       & 7.93e-02   & --         & 4.41e-02   & --         & 7.58e-03   & --         & 4.59e-02   & 3.00e-02   \\ 
      961        & 48256      & 6.23e-03   & 3.67       & 1.98e-03   & 4.48       & 2.97e-04   & 4.67       & 1.20e-01   & 1.13e-01   \\ 
      4097       & 219136     & 4.16e-04   & 3.90       & 6.43e-05   & 4.95       & 1.32e-05   & 4.49       & 5.05e-01   & 6.10e-01   \\ 
      16897      & 930304     & 3.09e-05   & 3.75       & 2.20e-06   & 4.87       & 8.19e-07   & 4.01       & 1.83e+00   & 3.27e+00   \\ 
      68609      & 3830272    & 2.28e-06   & 3.76       & 7.40e-08   & 4.89       & 5.12e-08   & 4.00       & 7.04e+00   & 1.79e+01   \\ 
      276481     & 15540736   & 1.63e-07   & 3.81       & 2.42e-09   & 4.93       & 3.14e-09   & 4.03       & 2.81e+01   & 1.09e+02   \\ 
      \midrule
      \multicolumn{10}{c}{$k{=}4$} \\
      \midrule
      257        & 13792      & 1.42e-02   & --         & 7.89e-03   & --         & 1.83e-03   & --         & 7.29e-02   & 4.23e-02   \\ 
      1185       & 74208      & 4.24e-04   & 5.07       & 1.14e-04   & 6.11       & 2.05e-05   & 6.48       & 2.29e-01   & 1.87e-01   \\ 
      5057       & 337312     & 1.81e-05   & 4.55       & 2.57e-06   & 5.48       & 6.39e-07   & 5.00       & 9.31e-01   & 9.60e-01   \\ 
      20865      & 1432608    & 6.90e-07   & 4.71       & 4.55e-08   & 5.82       & 2.28e-08   & 4.81       & 3.64e+00   & 5.71e+00   \\ 
      84737      & 5899552    & 2.59e-08   & 4.74       & 7.59e-10   & 5.91       & 7.64e-10   & 4.90       & 1.43e+01   & 3.34e+01   \\ 
      341505     & 23938848   & 9.53e-10   & 4.76       & 1.23e-11   & 5.95       & 2.42e-11   & 4.98       & 5.75e+01   & 2.05e+02   \\ 
      \midrule
      \multicolumn{10}{c}{$k{=}5$} \\
      \midrule
      305        & 19616      & 2.28e-03   & --         & 1.05e-03   & --         & 1.70e-04   & --         & 1.28e-01   & 5.63e-02   \\ 
      1409       & 105728     & 4.01e-05   & 5.83       & 1.05e-05   & 6.65       & 2.05e-06   & 6.37       & 3.95e-01   & 2.19e-01   \\ 
      6017       & 480896     & 7.21e-07   & 5.80       & 8.98e-08   & 6.87       & 3.21e-08   & 6.00       & 1.60e+00   & 1.32e+00   \\ 
      24833      & 2043008    & 1.37e-08   & 5.72       & 7.89e-10   & 6.83       & 5.43e-10   & 5.88       & 6.45e+00   & 8.29e+00   \\ 
      100865     & 8414336    & 2.56e-10   & 5.74       & 6.72e-12   & 6.88       & 9.14e-12   & 5.89       & 2.54e+01   & 5.01e+01   \\ 
      \bottomrule
    \end{tabular}
  \end{footnotesize}
\end{table}

\begin{table}
  \caption{Convergence results for the Kovasznay problem at $\Reynolds=40$ with weakly enforced boundary conditions and no convective term stabilisation. \label{tab:test:kovasznay.weak}}
  \begin{footnotesize}
    \begin{tabular}{cccccccccc}
      \toprule
      $N_{\rm dof}$  & $N_{\rm nz}$ & $\norm[\nu,h]{\uvec{e}_h}$ & EOC & $\norm[L^2(\Omega)^d]{\vec{e}_h}$ & EOC & $\norm[L^2(\Omega)]{\epsilon_h}$ & EOC & $\tau_{\rm ass}$ & $\tau_{\rm sol}$ \\
      \midrule
      \multicolumn{10}{c}{$k{=}0$} \\
      \midrule
      97         & 1216       & 1.07e+00   & --         & 3.93e-01   & --         & 6.80e-01   & --         & 2.68e-02   & 2.31e-02   \\ 
      353        & 4800       & 1.70e+00   & -0.67      & 9.58e-01   & -1.28      & 2.79e-01   & 1.28       & 3.41e-02   & 3.71e-02   \\ 
      1345       & 19072      & 1.44e+00   & 0.24       & 3.89e-01   & 1.30       & 1.32e-01   & 1.09       & 6.68e-02   & 8.04e-02   \\ 
      5249       & 76032      & 8.77e-01   & 0.72       & 1.18e-01   & 1.72       & 4.93e-02   & 1.42       & 2.15e-01   & 3.52e-01   \\ 
      20737      & 303616     & 4.78e-01   & 0.88       & 3.23e-02   & 1.87       & 1.49e-02   & 1.72       & 8.07e-01   & 1.95e+00   \\ 
      82433      & 1213440    & 2.46e-01   & 0.96       & 8.32e-03   & 1.96       & 4.08e-03   & 1.87       & 3.19e+00   & 1.47e+01   \\ 
      \midrule
      \multicolumn{10}{c}{$k{=}1$} \\
      \midrule
      177        & 4256       & 1.02e+00   & --         & 7.27e-01   & --         & 2.69e-01   & --         & 1.44e-02   & 1.60e-02   \\ 
      641        & 16768      & 4.20e-01   & 1.28       & 1.66e-01   & 2.13       & 4.96e-02   & 2.44       & 3.59e-02   & 4.25e-02   \\ 
      2433       & 66560      & 1.40e-01   & 1.58       & 2.66e-02   & 2.64       & 8.60e-03   & 2.53       & 1.09e-01   & 1.70e-01   \\ 
      9473       & 265216     & 4.06e-02   & 1.79       & 3.55e-03   & 2.91       & 1.29e-03   & 2.74       & 4.62e-01   & 1.10e+00   \\ 
      37377      & 1058816    & 1.03e-02   & 1.97       & 4.37e-04   & 3.02       & 1.79e-04   & 2.85       & 1.91e+00   & 5.64e+00   \\ 
      148481     & 4231168    & 2.61e-03   & 1.99       & 5.46e-05   & 3.00       & 2.96e-05   & 2.60       & 7.07e+00   & 3.32e+01   \\ 
      \midrule
      \multicolumn{10}{c}{$k{=}2$} \\
      \midrule
      257        & 9152       & 5.50e-01   & --         & 3.16e-01   & --         & 1.20e-01   & --         & 2.23e-02   & 2.33e-02   \\ 
      929        & 36032      & 7.58e-02   & 2.86       & 2.46e-02   & 3.68       & 6.03e-03   & 4.31       & 6.11e-02   & 7.47e-02   \\ 
      3521       & 142976     & 1.23e-02   & 2.62       & 1.84e-03   & 3.74       & 3.69e-04   & 4.03       & 2.41e-01   & 3.90e-01   \\ 
      13697      & 569600     & 1.70e-03   & 2.86       & 1.12e-04   & 4.03       & 3.63e-05   & 3.35       & 1.02e+00   & 2.21e+00   \\ 
      54017      & 2273792    & 2.21e-04   & 2.95       & 6.87e-06   & 4.03       & 3.84e-06   & 3.24       & 3.62e+00   & 1.17e+01   \\ 
      214529     & 9085952    & 2.80e-05   & 2.98       & 4.28e-07   & 4.00       & 3.72e-07   & 3.37       & 1.40e+01   & 6.76e+01   \\ 
      \midrule
      \multicolumn{10}{c}{$k{=}3$} \\
      \midrule
      337        & 15904      & 1.10e-01   & --         & 6.02e-02   & --         & 2.90e-02   & --         & 3.85e-02   & 3.26e-02   \\ 
      1217       & 62592      & 9.17e-03   & 3.58       & 2.30e-03   & 4.71       & 7.22e-04   & 5.33       & 1.05e-01   & 1.23e-01   \\ 
      4609       & 248320     & 6.93e-04   & 3.73       & 7.74e-05   & 4.89       & 2.38e-05   & 4.92       & 4.65e-01   & 6.74e-01   \\ 
      17921      & 989184     & 4.81e-05   & 3.85       & 2.44e-06   & 4.99       & 1.18e-06   & 4.34       & 1.82e+00   & 3.73e+00   \\ 
      70657      & 3948544    & 3.13e-06   & 3.94       & 7.88e-08   & 4.95       & 5.79e-08   & 4.35       & 6.79e+00   & 2.01e+01   \\ 
      280577     & 15777792   & 1.99e-07   & 3.97       & 2.51e-09   & 4.97       & 2.68e-09   & 4.43       & 2.68e+01   & 1.20e+02   \\ 
      \midrule
      \multicolumn{10}{c}{$k{=}4$} \\
      \midrule
      417        & 24512      & 2.46e-02   & --         & 7.32e-03   & --         & 5.12e-03   & --         & 6.26e-02   & 4.68e-02   \\ 
      1505       & 96448      & 9.27e-04   & 4.73       & 2.17e-04   & 5.08       & 7.04e-05   & 6.19       & 1.93e-01   & 1.89e-01   \\ 
      5697       & 382592     & 3.61e-05   & 4.68       & 3.62e-06   & 5.91       & 1.11e-06   & 5.98       & 8.13e-01   & 1.02e+00   \\ 
      22145      & 1523968    & 1.24e-06   & 4.87       & 5.36e-08   & 6.08       & 3.07e-08   & 5.18       & 3.13e+00   & 6.02e+00   \\ 
      87297      & 6083072    & 4.01e-08   & 4.95       & 8.21e-10   & 6.03       & 8.08e-10   & 5.25       & 1.19e+01   & 3.37e+01   \\ 
      346625     & 24306688   & 1.27e-09   & 4.98       & 1.28e-11   & 6.00       & 2.03e-11   & 5.31       & 4.68e+01   & 2.02e+02   \\ 
      \midrule
      \multicolumn{10}{c}{$k{=}5$} \\
      \midrule
      497        & 34976      & 6.48e-03   & --         & 1.76e-03   & --       & 1.02e-03   & --         & 1.23e-01   & 7.22e-02   \\ 
      1793       & 137600     & 7.07e-05   & 6.52       & 1.34e-05   & 7.04     & 4.58e-06   & 7.81       & 4.06e-01   & 2.95e-01   \\ 
      6785       & 545792     & 1.28e-06   & 5.79       & 1.10e-07   & 6.94     & 4.40e-08   & 6.70       & 1.51e+00   & 1.56e+00   \\ 
      26369      & 2173952    & 2.20e-08   & 5.87       & 8.84e-10   & 6.95     & 5.86e-10   & 6.23       & 5.67e+00   & 8.48e+00   \\ 
      103937     & 8677376    & 3.56e-10   & 5.95       & 7.20e-12   & 6.94     & 7.42e-12   & 6.30       & 2.28e+01   & 5.14e+01   \\ 
      \bottomrule
    \end{tabular}
  \end{footnotesize}
\end{table}

Numerical results confirm the theoretical $h$-convergence rates estimates; 
the EOC for the pressure error in $L^2$ norm is around (for strongly enforced boundary conditions) 
or exceeds (for weakly imposed boundary conditions) 
$(k{+}1)$, and we approach an EOC of $(k{+}1)$ and $(k{+}2)$
for the velocity error in the energy and the $L^2$ norm, respectively.
Focusing on the velocity errors, we remark that the convergence rates provided 
by method \eqref{eq:discrete.strong.bc} with strongly enforced boundary conditions are slightly sub-optimal,
while method \eqref{eq:discrete} with weakly enforced boundary conditions is very close to providing optimal rates of convergence.
We numerically verified that this occurrence is to be ascribed to convective term stabilization, active in the former while 
switched off in the latter configuration. 
Note that the EOC of $k{=}0$ degree discretizations is severely impacted, probably due to the higher relative significance 
of the upwind stabilization with respect to the centered contribution in the discretisation of the convective term.
Indeed, the highest jumps between element and faces unknowns are observed for under-resolved low degree computations. 

\subsection{Two- and three-dimensional lid-driven cavity flow}

We next use the HHO method \eqref{eq:discrete.strong.bc} to solve the lid-driven cavity flow, one of the most extensively studied problems in fluid mechanics.
The computational domain is either the unit square $\Omega=(0,1)^2$ or the unit cube $\Omega=(0,1)^3$, depending on the flow space dimensions.
Homogeneous (wall) boundary conditions are enforced at all but the top horizontal wall (at $x_2=1$), 
where we enforce a unit tangential velocity (that is, $\vec{u}=(1,0)$ if $d=2$, $\vec{u}=(1,0,0)$ if $d=3$).
We note that this boundary condition is incompatible with the formulation \eqref{eq:weak}, 
even modified to account for non-homogeneous boundary conditions, since the solution to the lid driven cavity does not 
belong to $H^1(\Omega)^d$; it is however, as mentioned, a very classical and well-understood test that informs on the quality of the numerical scheme.
Due to the crucial importance of suitably enforcing the velocity discontinuity at the top corners of the cavity we rely on strong imposition of boundary conditions.

In Figures \ref{fig:cavity:Re1000}, \ref{fig:cavity:Re5000}, \ref{fig:cavity:Re10000}, and \ref{fig:cavity:Re20000} 
we report the horizontal component $u_1$ of the velocity along the vertical centerline $x_1=\frac12$ and the vertical component $u_2$ 
of the velocity along the horizontal centerline $x_2=\frac12$ for the two dimensional flow at global Reynolds numbers $\Reynolds\coloneq\frac1{\nu}$
respectively equal to $\pgfmathprintnumber{1000}$, $\pgfmathprintnumber{5000}$, $\pgfmathprintnumber{10000}$, and $\pgfmathprintnumber{20000}$.
The reference computation is carried out on a $128\times128$ Cartesian mesh with $k{=}1$.
For the sake of comparison, we also include very high-order computations with $k{=}7$ on progressively finer Cartesian grids: 
$16\times16$ for $\Reynolds=\pgfmathprintnumber{1000}$, $32\times32$ for $\Reynolds=\pgfmathprintnumber{5000}$, 
$64\times64$ for $\Reynolds=\pgfmathprintnumber{10000}$, and $128\times 128$ for $\Reynolds=\pgfmathprintnumber{20000}$.
The high-order solution corresponding to $\Reynolds=\pgfmathprintnumber{1000}$ and $\Reynolds=\pgfmathprintnumber{20000}$ are displayed in Figure \ref{fig:cavity:u.magnitude}.
When available, references solutions from the literature \cite{Ghia.Ghia.ea:82,Erturk.Corke.ea:05} are also plotted for the sake of comparison.

We remark that the solid gray and black lines outlining the behavior of 
low-order ($k{=}1$) and high-order ($k{=}7$) velocity approximations, respectively, are perfectly superimposed 
at low Reynolds numbers, while significant differences are visible starting from $\Reynolds=\pgfmathprintnumber{10000}$.
In particular, at $\Reynolds=\pgfmathprintnumber{20000}$, $k{=}1$ computations are in better agreement with reference solutions by Erturk \emph{et al} \cite{Erturk.Corke.ea:05}.
Nevertheless, since high-polynomial degrees over coarse meshes provide accurate results at low Reynolds numbers, 
we are led to think that $k{=}1$ HHO computations are over-dissipative at high Reynolds. 
Indeed, strong velocity gradients observed close to cavity walls and multiple counter-rotating vortices developing at the bottom corners 
are known to be very demanding, both from the stability and the accuracy viewpoints.
Note that the thin jet originating at the top-right corner is contained in exactly one mesh element, both on the $16\times16$ grid for $\Reynolds=\pgfmathprintnumber{1000}$ and 
on the $128\times128$ grid for $\Reynolds=\pgfmathprintnumber{20000}$, see Figure \ref{fig:cavity:u.magnitude}.

The three-dimensional lid-driven cavity flow is computed at $\Reynolds=\pgfmathprintnumber{1000}$, see Figure \ref{fig:cavity:stream3d}.
In Figure \ref{fig:cavity:Re1000_3d} we report $k{=}1,2,4$ HHO computations over $32^3, 16^3, 8^3$ hexahedral element grids of the unit cube, respectively 
(we double both the mesh step size $h$ and the polynomial degree $k$ in order to perform high-order accurate computations at reasonable computational costs).
Comparing the velocity at horizontal and vertical centerlines (passing through the centroid of the unit cube $c=(0.5,0.5,0.5)$) with reference solution from the literature \cite{Albensoeder.Kuhlmann:05},
we demonstrate the ability to accurately reproduce the flow behaviour with coarse meshes and relatively high polynomial degrees ($k{=}4$).
As opposite, the mismatch region observed in a neighborhood of the negative peak of the $u_1$ 
velocity component distribution suggests that both $k{=}1$ and $k{=}2$ HHO computations are over-dissipative.
For the sake of comparison, in Figure \ref{fig:cavity:Re1000_3d_ho}, 
we report higher-order accurate $k{=}4$ and $k{=}8$ HHO computations over $16^3$ and $8^3$ hexahedral grids of the unit cube, respectively.
These latter velocity solutions are in very good agreement with both the reference solutions of \cite{Albensoeder.Kuhlmann:05} and the $k{=}4$ HHO computation over the $8^3$ hexahedral grid.
The three-dimensional lid-driven cavity computations were run on a dual 18 cores Xeon CPU cluster node exploiting a shared-memory, thread-based implementation, 
with both matrix assembly and LU factorization performed in parallel exploiting 36 concurrent compute units. 

In Table \ref{tab:cavity:dspecs} it is possible to evaluate the computational expense
of three-dimensional lid-driven cavity flow computations 
in terms of global system matrix properties and degrees of freedom count.
It is interesting to remark that doubling the mesh step size together with the 
polynomial degree is beneficial, not only from the accuracy viewpoint, see Figure \ref{fig:cavity:Re1000_3d}, 
but also from the computational costs viewpoint. 
Indeed, the global matrix size (equal to the number of unknowns after static condensation), 
the number of non-zero entries of the global statically condensed system matrix,
and the number of element unknowns (recall that only pressure averages survive after static condensation) 
decreases on coarser meshes with higher polynomial degrees.  
Note that the same $h$-coarsening plus $p$-refinement strategy employing discontinuous Galerkin (dG)
instead of HHO would have led to a significant increase of the number of Jacobian matrix non-zeroes entries,
in particular a $k=8$ dG discretization would have topped at 1.5 billion non-zeroes.  
Indeed, when considering a $d$-dimensional flow problem, 
the leading block size of the global sparse matrix grows as 
the size of polynomial spaces in $d{-}1$ and $d$ variables for HHO and dG methods, respectively.
This crucial difference suggests that significant efficiency gains 
might be obtained in the context of implicit time discretizations 
employing high-order HHO discretizations.

\begin{figure}[!htb]
\centering
\includegraphics[width=0.48\textwidth]{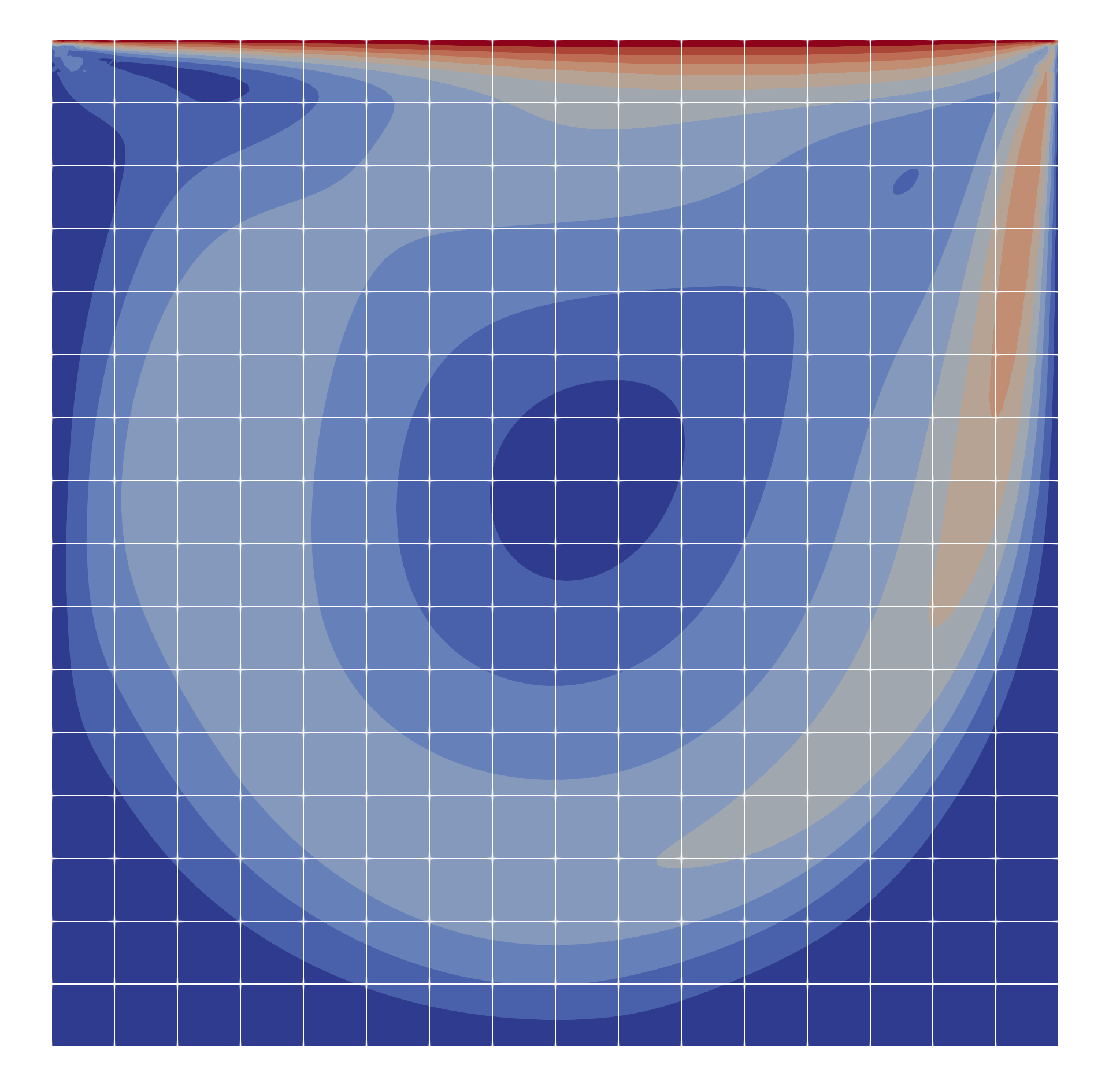} 
\includegraphics[width=0.48\textwidth]{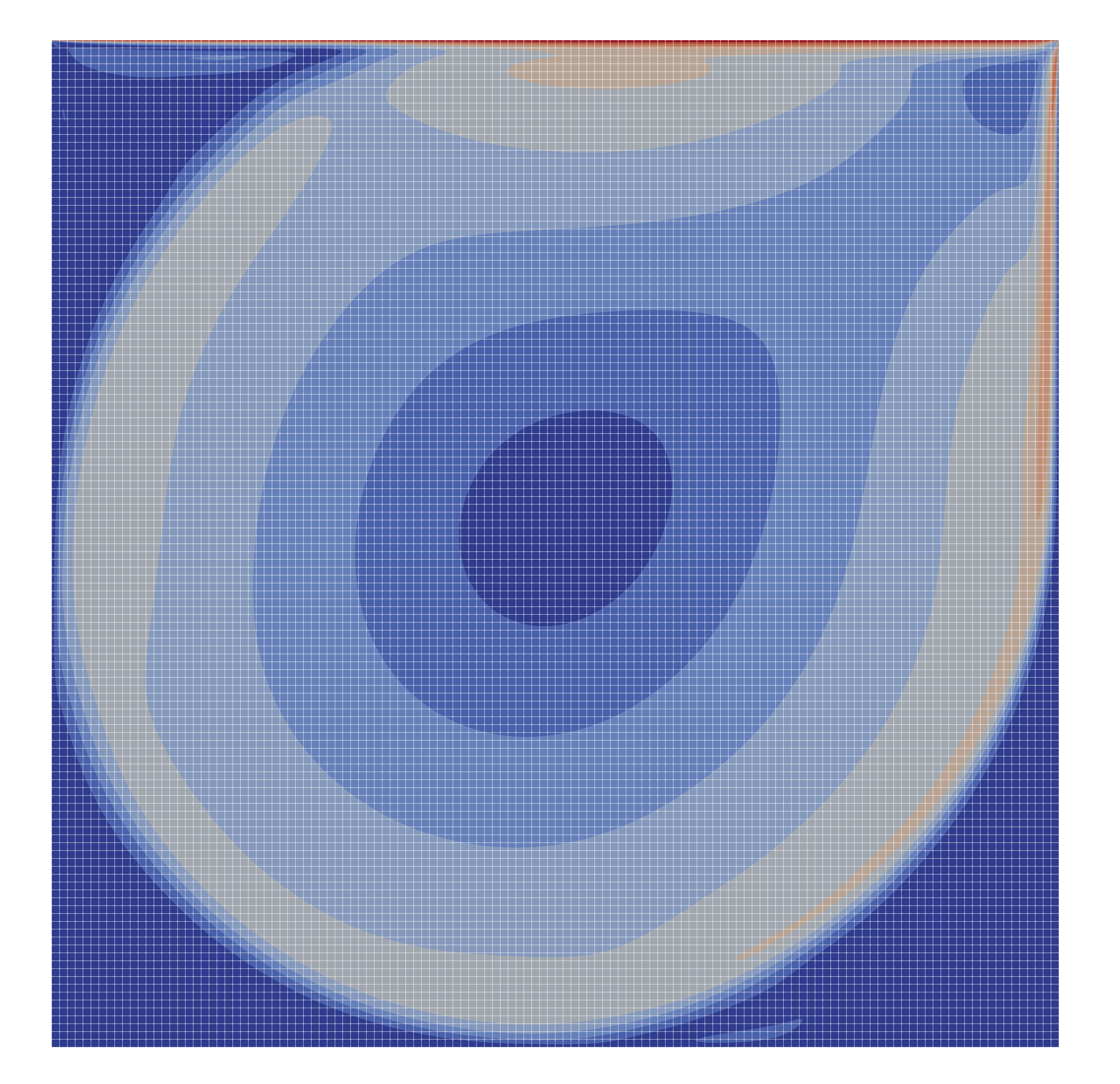} 
\caption{Two-dimensional lid-driven cavity flow, velocity magnitude contours ($10$ equispaced values in the range $[0,1]$) for $k{=}7$ computations at $\Reynolds=\pgfmathprintnumber{1000}$ (\emph{left}: 16x16 grid) and $\Reynolds=\pgfmathprintnumber{20000}$ (\emph{right}: 128x128 grid).\label{fig:cavity:u.magnitude}
}
\end{figure}

\begin{figure}[!htb]
\centering
\includegraphics[width=0.48\textwidth]{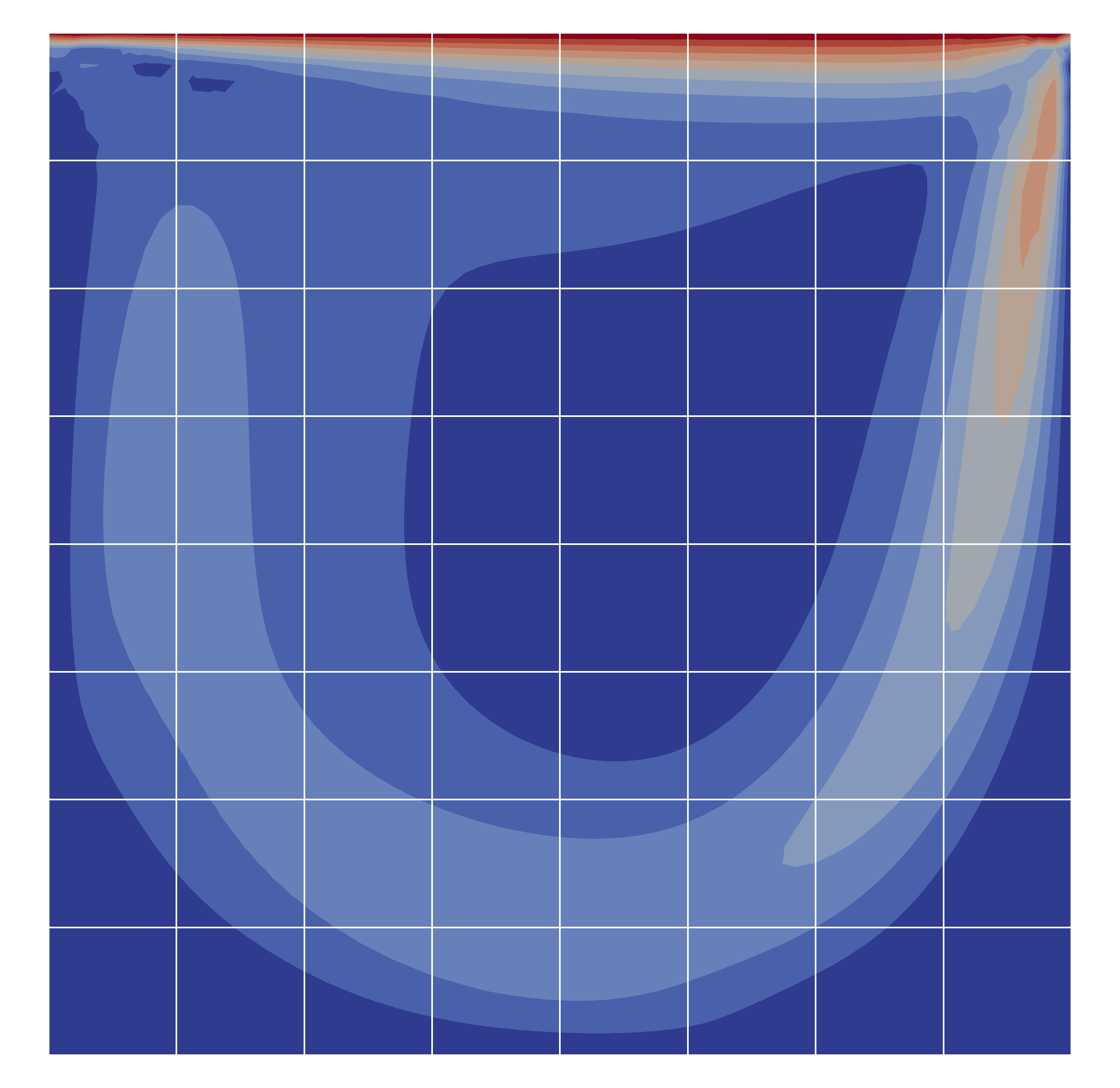} 
\includegraphics[width=0.48\textwidth]{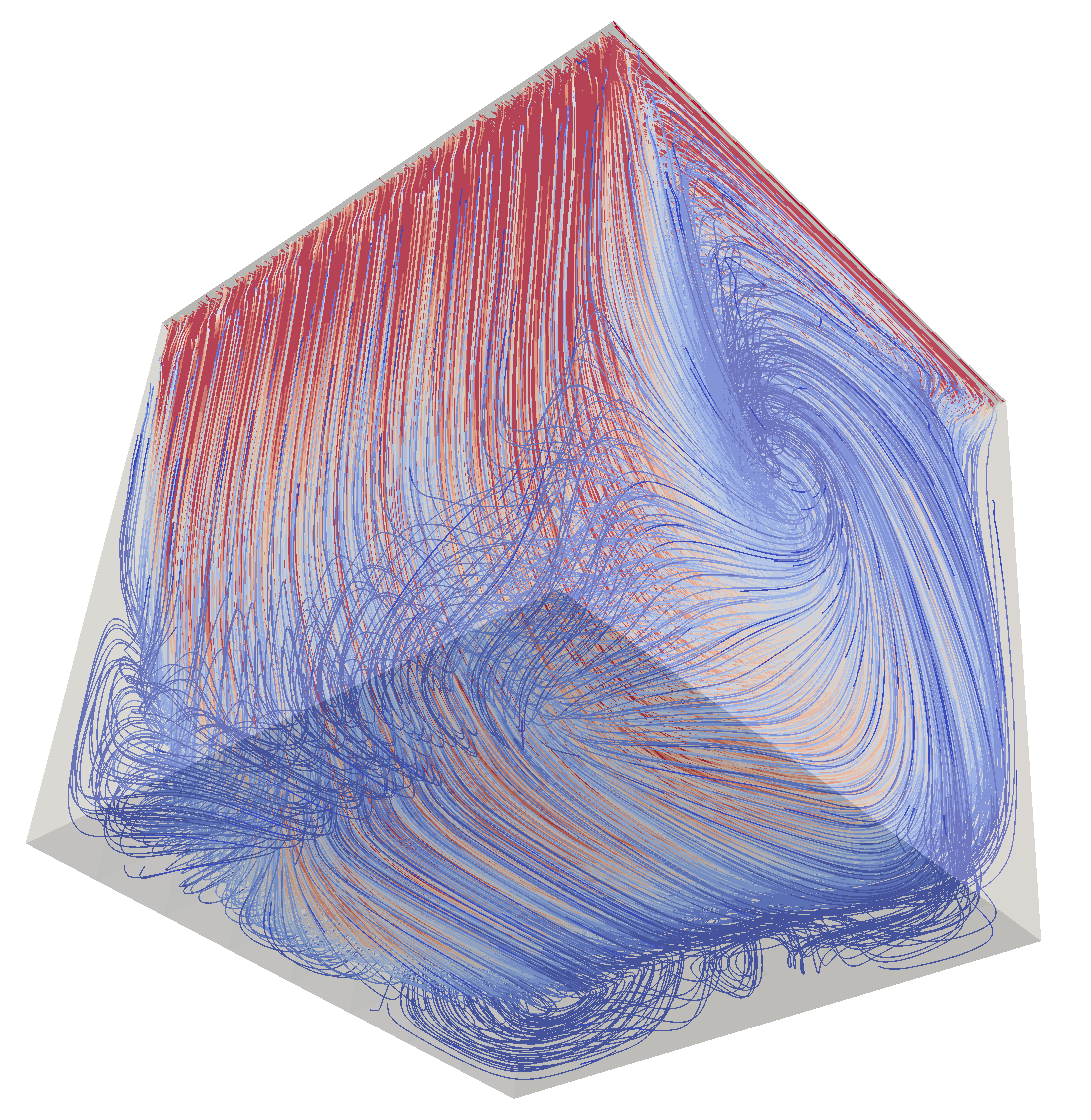} 
\caption{Three-dimensional lid-driven cavity flow at $\Reynolds=\pgfmathprintnumber{1000}$ computed by means of a $k{=}8$, $8^3$ hexahedral elements grid HHO discretization. \emph{Left,} velocity magnitude contours ($10$ equispaced values in the range $[0,1]$) over the vertical mid-plane whose top edge is aligned with the velocity vector. \emph{Right}: Streamlines color-coded by velocity magnitude. \label{fig:cavity:stream3d}
}
\end{figure}


\begin{figure}\centering
  \begin{tikzpicture}[font=\footnotesize, %
      spy using outlines={magnification=4, size=3cm, connect spies, fill=none} %
    ]
    \begin{axis}[height=9cm, width=9cm, %
        xmin=-1, xmax=1, ymin=0, ymax=1, %
        xlabel={$u_1$}, ylabel={$x_2$}, %
        legend style = { at={(0,1)}, anchor=north west, draw=none, fill=none }, %
        axis x line=top, ytick pos=bottom]
      \addplot[mark=o, only marks, mark size=1.4, color=black!50] table [x=Re1000, y=y] {ghia-ghia-shin-ux-vertical-centerline.txt};
      \addplot[mark=+, only marks, color=black!75] table [x=Re1000, y=y] {erturk-corke-gokcol-ux-vertical-centerline.txt};
      \addplot[mark=none, thick, color=black!50] table [x="a", y="Points:1", col sep=comma] {re1000-2d-p1-128-strong-bc-ux-vertical-centerline.txt};
      \addplot[mark=none, thick, color=black] table [x="a", y="Points:1", col sep=comma] {re1000-2d-p7-strong-bc-ux-vertical-centerline.txt};
      \legend{Ghia et al.,Erturk et al.,{$k{=}1$, $128\times128$, strong bc},{$k{=}7$, $16\times 16$, strong bc}}
    \end{axis}
    \begin{axis}[height=9cm, width=9cm, %
        xmin=0, xmax=1, ymin=-1, ymax=1,%
        xlabel={$x_1$}, ylabel={$u_2$}, %
        axis y line=right, xtick pos=left]
      \addplot[mark=o, only marks, mark size=1.4, color=black!50] table [x=x, y=Re1000] {ghia-ghia-shin-uy-horizontal-centerline.txt};
      \addplot[mark=+, only marks, color=black!75] table [x=x, y=Re1000] {erturk-corke-gokcol-uy-horizontal-centerline.txt};      
      \addplot[mark=none, thick, color=black!50] table [x="Points:0", y="b", col sep=comma] {re1000-2d-p1-128-strong-bc-uy-horizontal-centerline.txt};
      \addplot[mark=none, thick, color=black] table [x="Points:0", y="b", col sep=comma] {re1000-2d-p7-strong-bc-uy-horizontal-centerline.txt};
    \end{axis}
    \begin{scope}
      \spy on (0.85,5.00) in node [fill=none, anchor=south east] at (-0.75,4.25);
      \spy on (5.50,7.00) in node [fill=none, anchor=south west] at (8.35,4.25); 
      \spy on (6.70,2.05) in node [fill=none, anchor=south west] at (8.35,0.25);
      \spy on (2.30,1.15) in node [fill=none, anchor=south east] at (-0.75,0.25);
    \end{scope}
  \end{tikzpicture}
  \caption{2D Lid-driven cavity flow, horizontal component $u_1$ of the velocity along the vertical centerline $x_1=\frac12$ and the vertical component $u_2$ of the velocity along the horizontal centerline $x_2=\frac12$ for $\Reynolds=\pgfmathprintnumber{1000}$\label{fig:cavity:Re1000}}
\end{figure}


\begin{figure}\centering
  \begin{tikzpicture}[font=\footnotesize, %
      spy using outlines={magnification=4, size=3cm, connect spies, fill=none} %
    ]
    \begin{axis}[height=9cm, width=9cm, %
        xmin=-1, xmax=1, ymin=0, ymax=1, %
        xlabel={$u_1$}, ylabel={$x_2$}, %
        legend style = { at={(0,1)}, anchor=north west, draw=none, fill=none }, %
        axis x line=top, ytick pos=bottom]
      \addplot[mark=o, only marks, mark size=1.4, color=black!50] table [x=Re5000, y=y] {ghia-ghia-shin-ux-vertical-centerline.txt};
      \addplot[mark=+, only marks, color=black!75] table [x=Re5000, y=y] {erturk-corke-gokcol-ux-vertical-centerline.txt};
      \addplot[mark=none, thick, color=black!50] table [x="a", y="Points:1", col sep=comma] {re5000-2d-p1-128-strong-bc-ux-vertical-centerline.txt};
      \addplot[mark=none, thick, color=black] table [x="a", y="Points:1", col sep=comma] {re5000-2d-p7-strong-bc-ux-vertical-centerline.txt};
      \legend{Ghia et al.,Erturk et al.,{$k{=}1$, $128\times128$, strong bc},{$k{=}7$, $32\times 32$, strong bc}}
    \end{axis}
    \begin{axis}[height=9cm, width=9cm, %
        xmin=0, xmax=1, ymin=-1, ymax=1,%
        xlabel={$x_1$}, ylabel={$u_2$}, %
        axis y line=right, xtick pos=left]
      \addplot[mark=o, only marks, mark size=1.4, color=black!50] table [x=x, y=Re5000] {ghia-ghia-shin-uy-horizontal-centerline.txt};
      \addplot[mark=+, only marks, color=black!75] table [x=x, y=Re5000] {erturk-corke-gokcol-uy-horizontal-centerline.txt};      
      \addplot[mark=none, thick, color=black!50] table [x="Points:0", y="b", col sep=comma] {re5000-2d-p1-128-strong-bc-uy-horizontal-centerline.txt};
      \addplot[mark=none, thick, color=black] table [x="Points:0", y="b", col sep=comma] {re5000-2d-p7-strong-bc-uy-horizontal-centerline.txt};
    \end{axis}
    \begin{scope}
      \spy on (0.50,5.15) in node [fill=none, anchor=south east] at (-0.75,4.25);  
      \spy on (5.50,7.00) in node [fill=none, anchor=south west] at (8.35,4.25); 
      \spy on (7.00,1.85) in node [fill=none, anchor=south west] at (8.35,0.25);
      \spy on (2.25,0.70) in node [fill=none, anchor=south east] at (-0.75,0.25);
    \end{scope}
  \end{tikzpicture}
  \caption{2D Lid-driven cavity flow, horizontal component $u_1$ of the velocity along the vertical centerline $x_1=\frac12$ and the vertical component $u_2$ of the velocity along the horizontal centerline $x_2=\frac12$ for $\Reynolds=\pgfmathprintnumber{5000}$\label{fig:cavity:Re5000}}
\end{figure}


\begin{figure}\centering
  \begin{tikzpicture}[font=\footnotesize, %
      spy using outlines={magnification=4, size=3cm, connect spies, fill=none} %
    ]
    \begin{axis}[height=9cm, width=9cm, %
        xmin=-1, xmax=1, ymin=0, ymax=1, %
        xlabel={$u_1$}, ylabel={$x_2$}, %
        legend style = { at={(0,1)}, anchor=north west, draw=none, fill=none }, %
        axis x line=top, ytick pos=bottom]
      \addplot[mark=o, only marks, mark size=1.4, color=black!50] table [x=Re10000, y=y] {ghia-ghia-shin-ux-vertical-centerline.txt};
      \addplot[mark=+, only marks, color=black!75] table [x=Re10000, y=y] {erturk-corke-gokcol-ux-vertical-centerline.txt};
      \addplot[mark=none, thick, color=black!50] table [x="a", y="Points:1", col sep=comma] {re10000-2d-p1-128-strong-bc-ux-vertical-centerline.txt};
      \addplot[mark=none, thick, color=black] table [x="a", y="Points:1", col sep=comma] {re10000-2d-p7-strong-bc-ux-vertical-centerline.txt};
      \legend{Ghia et al.,Erturk et al.,{$k{=}1$, $128\times128$, strong bc},{$k{=}7$, $64\times 64$, strong bc}}
    \end{axis}
    \begin{axis}[height=9cm, width=9cm, %
        xmin=0, xmax=1, ymin=-1, ymax=1,%
        xlabel={$x_1$}, ylabel={$u_2$}, %
        axis y line=right, xtick pos=left]
      \addplot[mark=o, only marks, mark size=1.4, color=black!50] table [x=x, y=Re10000] {ghia-ghia-shin-uy-horizontal-centerline.txt};
      \addplot[mark=+, only marks, color=black!75] table [x=x, y=Re10000] {erturk-corke-gokcol-uy-horizontal-centerline.txt};      
      \addplot[mark=none, thick, color=black!50] table [x="Points:0", y="b", col sep=comma] {re10000-2d-p1-128-strong-bc-uy-horizontal-centerline.txt};
      \addplot[mark=none, thick, color=black] table [x="Points:0", y="b", col sep=comma] {re10000-2d-p7-strong-bc-uy-horizontal-centerline.txt};
    \end{axis}
    \begin{scope}
      \spy on (0.50,5.25) in node [fill=none, anchor=south east] at (-0.75,4.25);
      \spy on (5.50,7.00) in node [fill=none, anchor=south west] at (8.35,4.25);  
      \spy on (7.00,1.85) in node [fill=none, anchor=south west] at (8.35,0.25);
      \spy on (2.25,0.50) in node [fill=none, anchor=south east] at (-0.75,0.25);
    \end{scope}
  \end{tikzpicture}
  \caption{2D Lid-driven cavity flow, horizontal component $u_1$ of the velocity along the vertical centerline $x_1=\frac12$ and the vertical component $u_2$ of the velocity along the horizontal centerline $x_2=\frac12$ for $\Reynolds=\pgfmathprintnumber{10000}$\label{fig:cavity:Re10000}}
\end{figure}


\begin{figure}\centering
  \begin{tikzpicture}[font=\footnotesize, %
      spy using outlines={magnification=4, size=3cm, connect spies, fill=none} %
    ]
    \begin{axis}[height=9cm, width=9cm, %
        xmin=-1, xmax=1, ymin=0, ymax=1, %
        xlabel={$u_1$}, ylabel={$x_2$}, %
        legend style = { at={(0,1)}, anchor=north west, draw=none, fill=none }, %
        axis x line=top, ytick pos=bottom]
      \addplot[mark=+, only marks, color=black!75] table [x=Re20000, y=y] {erturk-corke-gokcol-ux-vertical-centerline.txt};
      \addplot[mark=none, thick, color=black!50] table [x="a", y="Points:1", col sep=comma] {re20000-2d-p1-128-strong-bc-ux-vertical-centerline.txt};
      \addplot[mark=none, thick, color=black] table [x="a", y="Points:1", col sep=comma] {re20000-2d-p7-strong-bc-ux-vertical-centerline.txt};
      \legend{Erturk et al.,{$k{=}1$, $128\times128$, strong bc},{$k{=}7$, $128\times 128$, strong bc}}
    \end{axis}
    \begin{axis}[height=9cm, width=9cm, %
        xmin=0, xmax=1, ymin=-1, ymax=1,%
        xlabel={$x_1$}, ylabel={$u_2$}, %
        axis y line=right, xtick pos=left]
      \addplot[mark=+, only marks, color=black!75] table [x=x, y=Re20000] {erturk-corke-gokcol-uy-horizontal-centerline.txt};      
      \addplot[mark=none, thick, color=black!50] table [x="Points:0", y="b", col sep=comma] {re20000-2d-p1-128-strong-bc-uy-horizontal-centerline.txt};
      \addplot[mark=none, thick, color=black] table [x="Points:0", y="b", col sep=comma] {re20000-2d-p7-strong-bc-uy-horizontal-centerline.txt};
    \end{axis}
    \begin{scope}
      \spy on (0.50,5.25) in node [fill=none, anchor=south east] at (-0.75,4.25);
      \spy on (5.50,7.00) in node [fill=none, anchor=south west] at (8.35,4.25); 
      \spy on (7.00,1.85) in node [fill=none, anchor=south west] at (8.35,0.25);
      \spy on (2.25,0.50) in node [fill=none, anchor=south east] at (-0.75,0.25);
    \end{scope}
  \end{tikzpicture}
  \caption{2D Lid-driven cavity flow, horizontal component $u_1$ of the velocity along the vertical centerline $x_1=\frac12$ and the vertical component $u_2$ of the velocity along the horizontal centerline $x_2=\frac12$ for $\Reynolds=\pgfmathprintnumber{20000}$\label{fig:cavity:Re20000}}
\end{figure}


\begin{figure}\centering
  \begin{tikzpicture}[font=\footnotesize, %
      spy using outlines={magnification=3, size=3cm, connect spies, fill=none} %
    ]
    \begin{axis}[height=9cm, width=9cm, %
        xmin=-1, xmax=1, ymin=0, ymax=1, %
        xlabel={$u_1$}, ylabel={$x_2$}, %
        legend style = { at={(0,1)}, anchor=north west, draw=none, fill=none }, %
        axis x line=top, ytick pos=bottom]
      \addplot[mark=+, only marks, color=black!75] table [x=Re1000, y=y] {albensoeder-kuhlmann-ux-vertical-centerline.txt};
      \addplot[mark=none, thick, color=red] table [x="a", y="Points:1", col sep=comma] {re1000-3d-p1-32-strong-bc-ux-vertical-centerline.txt};
      \addplot[mark=none, thick, color=blue] table [x="a", y="Points:1", col sep=comma] {re1000-3d-p2-16-strong-bc-ux-vertical-centerline.txt};
      \addplot[mark=none, thick, color=black] table [x="a", y="Points:1", col sep=comma] {re1000-3d-p4-8-strong-bc-ux-vertical-centerline.txt};
      \legend{Albensoeder et al.,{$k{=}1$, $32^3$, strong bc},{$k{=}2$, $16^3$, strong bc},{$k{=}4$, $8^3$, strong bc}}
    \end{axis}
    \begin{axis}[height=9cm, width=9cm, %
        xmin=0, xmax=1, ymin=-1, ymax=1,%
        xlabel={$x_1$}, ylabel={$u_2$}, %
        axis y line=right, xtick pos=left]
      \addplot[mark=+, only marks, color=black!75] table [x=x, y=Re1000] {albensoeder-kuhlmann-uy-horizontal-centerline.txt};      
      \addplot[mark=none, thick, color=red] table [x="Points:0", y="b", col sep=comma] {re1000-3d-p1-32-strong-bc-uy-horizontal-centerline.txt};
      \addplot[mark=none, thick, color=blue] table [x="Points:0", y="b", col sep=comma] {re1000-3d-p2-16-strong-bc-uy-horizontal-centerline.txt};
      \addplot[mark=none, thick, color=black] table [x="Points:0", y="b", col sep=comma] {re1000-3d-p4-8-strong-bc-uy-horizontal-centerline.txt};
    \end{axis}
    \begin{scope}
      \spy on (0.85,4.50) in node [fill=none, anchor=south east] at (-0.75,4.25);
      \spy on (4.50,6.70) in node [fill=none, anchor=south west] at (8.35,4.25); 
      \spy on (6.70,2.45) in node [fill=none, anchor=south west] at (8.35,0.25);
      \spy on (2.70,0.85) in node [fill=none, anchor=south east] at (-0.75,0.25);
    \end{scope}
  \end{tikzpicture}
  \caption{3D Lid-driven cavity flow, horizontal component $u_1$ of the velocity along the vertical centerline $x_1,x_3=\frac12$ and the vertical component $u_2$ of the velocity along the horizontal centerline $x_2,x_3=\frac12$ for $\Reynolds=\pgfmathprintnumber{1000}$\label{fig:cavity:Re1000_3d}}
\end{figure}


\begin{figure}\centering
  \begin{tikzpicture}[font=\footnotesize, %
      spy using outlines={magnification=3, size=3cm, connect spies, fill=none} %
    ]
    \begin{axis}[height=9cm, width=9cm, %
        xmin=-1, xmax=1, ymin=0, ymax=1, %
        xlabel={$u_1$}, ylabel={$x_2$}, %
        legend style = { at={(0,1)}, anchor=north west, draw=none, fill=none }, %
        axis x line=top, ytick pos=bottom]
      \addplot[mark=+, only marks, color=black!75] table [x=Re1000, y=y] {albensoeder-kuhlmann-ux-vertical-centerline.txt};
      \addplot[mark=none, thick, color=blue] table [x="a", y="Points:1", col sep=comma] {re1000-3d-p4-16-strong-bc-ux-vertical-centerline.txt};
      \addplot[mark=none, thick, color=black] table [x="a", y="Points:1", col sep=comma] {re1000-3d-p8-8-strong-bc-ux-vertical-centerline.txt};
      \legend{Albensoeder et al.,{$k{=}4$, $16^3$, strong bc},{$k{=}8$, $8^3$, strong bc}}
    \end{axis}
    \begin{axis}[height=9cm, width=9cm, %
        xmin=0, xmax=1, ymin=-1, ymax=1,%
        xlabel={$x_1$}, ylabel={$u_2$}, %
        axis y line=right, xtick pos=left]
      \addplot[mark=+, only marks, color=black!75] table [x=x, y=Re1000] {albensoeder-kuhlmann-uy-horizontal-centerline.txt};      
      \addplot[mark=none, thick, color=blue] table [x="Points:0", y="b", col sep=comma] {re1000-3d-p4-16-strong-bc-uy-horizontal-centerline.txt};
      \addplot[mark=none, thick, color=black] table [x="Points:0", y="b", col sep=comma] {re1000-3d-p8-8-strong-bc-uy-horizontal-centerline.txt};
    \end{axis}
    \begin{scope}
      \spy on (0.85,4.50) in node [fill=none, anchor=south east] at (-0.75,4.25);
      \spy on (4.50,6.70) in node [fill=none, anchor=south west] at (8.35,4.25); 
      \spy on (6.70,2.45) in node [fill=none, anchor=south west] at (8.35,0.25);
      \spy on (2.70,0.85) in node [fill=none, anchor=south east] at (-0.75,0.25);
    \end{scope}
  \end{tikzpicture}
  \caption{3D Lid-driven cavity flow, horizontal component $u_1$ of the velocity along the vertical centerline $x_1,x_3=\frac12$ and the vertical component $u_2$ of the velocity along the horizontal centerline $x_2,x_3=\frac12$ for $\Reynolds=\pgfmathprintnumber{1000}$\label{fig:cavity:Re1000_3d_ho}}
\end{figure}

\begin{table}
\centering
\begin{tabular}{cc|ccc||cc|ccc}
  \toprule
  \multicolumn{5}{c||}{Figure \ref{fig:cavity:Re1000_3d} HHO discretizations specs} & \multicolumn{5}{c}{Figure \ref{fig:cavity:Re1000_3d_ho} HHO discretizations specs} \\
  \midrule
  degree & grid       & $N_{\rm dof}$ & $N_{\rm nz}$ & $N_{\rm edof}$ &  
  degree & grid       & $N_{\rm dof}$ & $N_{\rm nz}$ & $N_{\rm edof}$ \\ 
  $k{=}1$  & $32^3$   & 890k  & 86M   & 524k  &  
  ($k{=}2$)& ($32^3$) & (1747k) & (343M)  & (1310k) \\ 
  $k{=}2$  & $16^3$   & 211k  & 39M   & 164k  &  
  $k{=}4$  & $16^3$   & 522k  & 244M  & 573k  \\ 
  $k{=}4$  & $8^3$    & 61k   & 27M   & 72k   &  
  $k{=}8$  & $8^3$    & 182k  & 239M  & 338k  \\ 
  \bottomrule
\end{tabular}
\caption{Three-dimensional lid-driven cavity flow. For each HHO discretization of Figure \ref{fig:cavity:Re1000_3d} and \ref{fig:cavity:Re1000_3d_ho}, we report the size of the statically condensed global system matrix ($N_{\rm dof}$), its number of non-zero entries ($N_{\rm nz}$) and the number of elemental degrees of freedom ($N_{\rm edof} = (d+1) \, \dim(\Poly{k}(\Th))$). The $k=2$, $32^3$ grid HHO discretization is included for the sake of comparison. \label{tab:cavity:dspecs}} 
\end{table}


\section{Flux formulation}\label{sec:flux}

We show here that the discrete problem \eqref{eq:discrete.strong.bc} with strongly enforced boundary conditions admits a reformulation in terms of conservative numerical fluxes.
\begin{proposition}[Flux formulation]\label{prop:flux}
  Denote by $(\uvec{u}_h,p_h)\in\Uhz\times\Ph$ the unique solution to \eqref{eq:discrete.strong.bc} and, for all $T\in\Th$ and all $F\in\Fh[T]$, define the numerical normal traces of the viscous and convective momentum fluxes as follows:
  $$
  \begin{aligned}
    \flux_{TF}^\visc(\uvec{u}_T)
    &\coloneq
    \nu\left(-\GRAD\rT\uvec{u}_T\normal_{TF} + \vec{R}_{TF}^k\uvec{u}_T\right),
    \\
    \flux_{TF}^\conv(\uvec{u}_T)&\coloneq
    \vlproj[F]{k}\left[%
      \frac{\vec{u}_F\SCAL\normal_{TF}}{2}(\vec{u}_F+\vec{u}_T)%
      - \frac{\nu}{h_F}\rho(\Pe(\vec{u}_F))(\vec{u}_F-\vec{u}_T)%
      \right],
  \end{aligned}
  $$
  where, letting $\uvec{D}_{\partial T}^k\coloneq\left\{\uvec{\alpha}_{\partial T}\coloneq(\vec{\alpha}_F)_{F\in\Fh[T]}\st\vec{v}_F\in\Poly{k}(F)^d\quad\forall F\in\Fh[T]\right\}$, the boundary residual operator $\underline{\vec{R}}_{\partial T}^k\coloneq(\vec{R}_{TF}^k)_{F\in\Fh[T]}:\UT\to\uvec{D}_{\partial T}^k$ is such that, for all $\uvec{v}_T\in\UT$ and all $\uvec{\alpha}_{\partial T}\in \uvec{D}_{\partial T}^k$,
  $$
  -\sum_{T\in\Th}(\vec{R}_{TF}^k\uvec{v}_T,\vec{\alpha}_F)_F
  = \mathrm{s}_T(\uvec{v}_T,(\vec{0},\uvec{\alpha}_{\partial T})).
  $$
  Then, for all $T\in\Th$, it holds:
  For any $(\vec{v}_T,q_T)\in\Poly{k}(T)^d\times\Poly{k}(T)$,%
  \begin{subequations}%
    \begin{align}
      \nonumber\int_T\nu\GRAD\rT\uvec{u}_T\SSCAL\GRAD\vec{v}_T
      - \int_T\vec{u}_T\SCAL(\vec{u}_T\SCAL\GRAD)\vec{v}_T
      - \frac12\int_T\DT[2k]\uvec{u}_T(\vec{u}_T\SCAL\vec{v}_T)
      - \int_T p_T(\DIV\vec{v}_T)      
      \hspace{-2em}&
      \\ \label{eq:discrete.strong.bc:momentum.balance}
      + \sum_{F\in\Fh[T]}\int_F\left(\flux_{TF}^\visc(\uvec{u}_T)+\flux_{TF}^\conv(\uvec{u}_T)+p_T\normal_{TF}\right)\SCAL\vec{v}_T
      &= \int_T\vec{f}\SCAL\vec{v}_T,
      \\ \label{eq:discrete.strong.bc:mass.balance}
      \int_T\vec{u}_T\SCAL\GRAD q_T
      - \sum_{F\in\Fh[T]}\int_F(\vec{u}_F\SCAL\normal_{TF})q_T
      &= 0.      
    \end{align}
  \end{subequations}
  Moreover, the numerical normal trace of the global momentum and mass fluxes are conservative, i.e., for any interface $F\in\Fhi$ such that $F\in\Fh[T_1]\cap\Fh[T_2]$ for distinct mesh elements $T_1,T_2\in\Th$,
  \begin{subequations}
    \begin{align}
      \label{eq:discrete.strong.bc:momentum.flux.continuity}
      \left(
      \flux_{T_1F}^\visc(\uvec{u}_{T_1}) + \flux_{T_1F}^\conv(\uvec{u}_{T_1}) + p_{T_1}\normal_{T_1F}
      \right)
      + \left(
      \flux_{T_2F}^\visc(\uvec{u}_{T_2}) + \flux_{T_2F}^\conv(\uvec{u}_{T_2}) + p_{T_2}\normal_{T_2F}
      \right)&=\vec{0},
      \\ \label{eq:discrete.strong.bc:mass.flux.continuity}
      \vec{u}_F\SCAL\normal_{T_1F} + \vec{u}_F\SCAL\normal_{T_2F} &= 0.
    \end{align}
  \end{subequations}
\end{proposition}
\begin{proof}
  The following equivalent expression for the viscous bilinear form $\mathrm{a}_h$ defined by \eqref{eq:ah} descends from \cite[Lemma 3.3]{Di-Pietro.Tittarelli:18}, where the scalar case is considered: For all $\uvec{v}_h\in\Uhz$,
  \begin{equation}\label{eq:flux:ah}
    \mathrm{a}_h(\uvec{u}_h,\uvec{v}_h)
    =\sum_{T\in\Th}\left[
      \int_T\GRAD\rT\uvec{u}_T\SSCAL\GRAD \vec{v}_T
      - \sum_{F\in\Fh[T]}\int_F \flux_{TF}^\visc(\uvec{u}_T)\SCAL(\vec{v}_F-\vec{v}_T)
      \right].
  \end{equation}
  Using the global discrete integration by parts formula \eqref{eq:ibp} and accounting for the strongly enforced boundary condition, for any $\uvec{v}_h\in\Uhz$ the convective trilinear form defined by \eqref{eq:th} evaluated at $(\uvec{u}_h,\uvec{u}_h,\uvec{v}_h)$ can be rewritten as follows:
  \begin{multline*}
    \mathrm{t}_h(\uvec{u}_h,\uvec{u}_h,\uvec{v}_h)
    =
    \\
    -\sum_{T\in\Th}\left[
      \int_T\vec{u}_T\SCAL\GwT[\uvec{u}_T]{\uvec{v}_T}
      + \frac12\int_T\DT[2k]\uvec{u}_T(\vec{u}_T\SCAL\vec{v}_T)
      + \frac12\sum_{F\in\Fh[T]}\int_F(\vec{u}_F\SCAL\normal_{TF})(\vec{u}_F-\vec{u}_T)\SCAL(\vec{v}_F-\vec{v}_T)
      \right].
  \end{multline*}
  Hence, expanding $\GwT[\uvec{u}_T]{\uvec{v}_T}$ according to its definition \eqref{eq:GwT} with $\uvec{w}_T=\uvec{u}_T$ and $\vec{z}=\vec{u}_T$ for all $T\in\Th$ in the previous expression, adding $\mathrm{j}_h(\uvec{u}_h;\uvec{u}_h,\uvec{v}_h)$, and rearranging the terms we obtain
  \begin{multline}\label{eq:flux:th}
    \mathrm{t}_h(\uvec{u}_h,\uvec{u}_h,\uvec{v}_h)
    + \mathrm{j}_h(\uvec{u}_h;\uvec{u}_h,\uvec{v}_h)
    =
    \\
    -\sum_{T\in\Th}\left[
      \int_T\vec{u}_T\SCAL(\vec{u}_T\SCAL\GRAD)\vec{v}_T
      + \frac12\int_T\DT[2k]\uvec{u}_T(\vec{u}_T\SCAL\vec{v}_T)
      + \sum_{F\in\Fh[T]}\int_F\flux_{TF}^\conv(\uvec{u}_T)\SCAL(\vec{v}_F-\vec{v}_T)
    \right],
  \end{multline}
  where we have further observed that $(\vec{v}_F-\vec{v}_{T|F})\in\Poly{k}(F)^d$ to insert $\vlproj[F]{k}$ into the expression of the convective flux.
  Finally, writing the definition \eqref{eq:bh} of $\mathrm{b}_h$ for $q_h=p_h$ and expanding, for all $T\in\Th$, $\DT\uvec{v}_T$ according to its definition \eqref{eq:DT} with $\ell=k$ and $q=p_T$, we obtain
  \begin{equation}\label{eq:flux:bh}
    \mathrm{b}_h(\uvec{u}_h,p_h)
    =-\sum_{T\in\Th}\left[
      \int_T p_T\DIV\vec{v}_T + \sum_{F\in\Fh[T]}\int_Fp_T\normal_{TF}\SCAL(\vec{v}_F-\vec{v}_T)
      \right].
  \end{equation}
  Plugging \eqref{eq:flux:ah}, \eqref{eq:flux:th}, and \eqref{eq:flux:bh} into the discrete momentum equation \eqref{eq:discrete.strong.bc:momentum}, we obtain, for all $\uvec{v}_h\in\Uhz$,
  \begin{multline*}
    \sum_{T\in\Th}\Bigg[
      \int_T\nu\GRAD\rT\uvec{u}_T\SSCAL\GRAD\vec{v}_T
    - \int_T\vec{u}_T\SCAL(\vec{u}_T\SCAL\GRAD)\vec{v}_T
    - \frac12\int_T\DT[2k]\uvec{u}_T(\vec{u}_T\SCAL\vec{v}_T)
    - \int_T p_T\DIV\vec{v}_T
    \\
    + \sum_{F\in\Fh[T]}\int_F\left(
    \flux_{TF}^\visc(\uvec{u}_T)+\flux_{TF}^\conv(\uvec{u}_T)+p_T\normal_{TF}
    \right)\SCAL(\vec{v}_T-\vec{v}_F)
    \Bigg]
    = \sum_{T\in\Th}\int_T\vec{f}\SCAL\vec{v}_T.
  \end{multline*}
  Selecting $\uvec{v}_h$ such that $\vec{v}_T$ spans $\Poly{k}(T)^d$ for a selected mesh element $T\in\Th$ while $\vec{v}_{T'} = \vec{0}$ for all $T'\in\Th\setminus\{T\}$ and $\vec{v}_F = \vec{0}$ for all $F\in\Fh$, we obtain \eqref{eq:discrete.strong.bc:momentum.balance}.
  On the other hand, selecting $\uvec{v}_h$ such that $\vec{v}_T = \vec{0}$ for all $T\in\Th$, $\vec{v}_F$ spans $\Poly{k}(F)^d$ for a selected interface $F\in\Fhi$ such that $F\subset\partial T_1\cap\partial T_2$ for distinct mesh elements $T_1,T_2\in\Th$, and $\vec{v}_{F'} = \vec{0}$ for all $F'\in\Fh\setminus\{F\}$, we deduce \eqref{eq:discrete.strong.bc:momentum.flux.continuity}.
  
  The discrete local mass balance \eqref{eq:discrete.strong.bc:mass.balance} is a straightforward consequence of \eqref{eq:discrete.strong.bc:mass} expanding $\DT$ according to its definition \eqref{eq:DT}.
  As a matter of fact, accounting for Remark \ref{rem:mass.eq}, we can take as a test function $q_h$ such that, for a selected mesh element $T\in\Th$, $q_T$ spans $\Poly{k}(T)$ while $q_{T'}=0$ for all $T'\in\Th\setminus\{T\}$.
  Finally, the continuity of the mass fluxes expressed by \eqref{eq:discrete.strong.bc:mass.flux.continuity} is an immediate consequence of the single-valuedness of face unknowns.
\end{proof}

\begin{remark}[Finite volume local mass and momentum balances]
  Let an element $T\in\Th$ be fixed.
  We start by observing that, taking in \eqref{eq:discrete.strong.bc:momentum.balance} $\vec{v}_T$ constant in $T$, the terms in the first line of this expression vanish.
  As a matter of fact, the first, second, and fourth terms involve derivatives of $\vec{v}_T$, while, recalling \eqref{eq:DT2k.DTk}, we can write for the third term $-\frac12\int_T\DT[2k]\uvec{u}_T(\vec{u}_T\SCAL\vec{v}_T) = -\frac12\int_T\DT\uvec{u}_T(\vec{u}_T\SCAL\vec{v}_T)=0$, where we have used \eqref{eq:discrete:mass} to infer $\DT\uvec{u}_T=0$ and conclude.
  Hence, letting now $\vec{v}_T$ be successively equal to the vectors of the canonical basis of $\Real^d$, we have the following finite volume-like local momentum balance:
  \begin{subequations}\label{eq:discrete.strong.bc:balance:fv}
    \begin{equation}\label{eq:discrete.strong.bc:momentum.balance:fv}
      \sum_{F\in\Fh[T]}\int_F\left(\flux_{TF}^\visc(\uvec{u}_T)+\flux_{TF}^\conv(\uvec{u}_T)+p_T\normal_{TF}\right) = \int_T\vec{f}.
    \end{equation}
    Similarly, taking in \eqref{eq:discrete.strong.bc:mass.balance} $q_T=1$, we have the following finite volume-like local mass balance:
    \begin{equation}\label{eq:discrete.strong.bc:mass.balance:fv}
       \sum_{F\in\Fh[T]}\int_F(\vec{u}_F\SCAL\normal_{TF}) = 0.
    \end{equation}
  \end{subequations}
  The relations \eqref{eq:discrete.strong.bc:balance:fv} are relevant from the engineering point of view, as they guarantee that both momentum and mass are preserved at the local level.
  From the mathematical point of view, they can be used, e.g., to derive \emph{a posteriori} error estimators by flux equilibration.
\end{remark}


\appendix

\section{Proofs of intermediate results}\label{sec:proofs}

This section collects the proofs of the intermediate results required in the analysis.

\subsection{Discrete Sobolev embeddings}\label{sec:convergence.analysis:preliminary.results:sobolev.embeddings}

\begin{proof}[Proof of Proposition \ref{prop:sobolev.embeddings}]
    Let $\uvec{v}_h\in\Uh$.
  It follows from \cite[Theorem 2.1]{Di-Pietro.Ern:10} (see also \cite[Theorem 5.3]{Di-Pietro.Ern:12}) that
  \begin{equation}\label{eq:sobolev.embeddings:dG}
    \norm[L^q(\Omega)^d]{\vec{v}_h}
    \lsim\left(
    \sum_{T\in\Th} \norm[L^2(T)^{d\times d}]{\GRAD\vec{v}_T}^2
    + \sum_{F\in\Fhi}h_F^{-1}\norm[L^2(F)^d]{\jump{\vec{v}_h}}^2
    + \sum_{F\in\Fhb}h_F^{-1}\norm[L^2(F)^d]{\vec{v}_{T_F}}^2
    \right)^{\frac12},
  \end{equation}
  where, for all $F\in\Fhi$ such that $F\in\Fh[T_1]\cap\Fh[T_2]$ for distinct mesh elements $T_1,T_2\in\Th$, we have introduced the jump operator such that $\jump{\vec{v}_h}\coloneq\vec{v}_{h|T_1}-\vec{v}_{h|T_2}$.
  For all $F\in\Fhi$, inserting $\pm\vec{v}_F$ and using a triangle inequality, it is readily inferred that
  $
  \norm[L^2(F)^d]{\jump{\vec{v}_h}}
  \le\norm[L^2(F)^d]{\vec{v}_{T_1}-\vec{v}_F}
  + \norm[L^2(F)^d]{\vec{v}_{T_2}-\vec{v}_F}.
  $
  Similarly, for all $F\in\Fhb$ we can write
  $
  \norm[L^2(F)^d]{\vec{v}_{T_F}}
  \le\norm[L^2(F)^d]{\vec{v}_{T_F}-\vec{v}_F}
  + \norm[L^2(F)^d]{\vec{v}_F}.
  $
  Plugging the previous inequalities into \eqref{eq:sobolev.embeddings:dG}, we obtain
  \begin{multline*}
  \norm[L^q(\Omega)^d]{\vec{v}_h}\\
  \lsim\left(
  \sum_{T\in\Th}\norm[L^2(T)^{d\times d}]{\GRAD\vec{v}_T}^2
  + \sum_{F\in\Fh}\sum_{T\in\Th[F]} h_F^{-1}\norm[L^2(F)^d]{\vec{v}_F-\vec{v}_T}^2
  + \sum_{F\in\Fhb}h_F^{-1}\norm[L^2(F)^d]{\vec{v}_F}^2
  \right)^{\frac12}.
  \end{multline*}
  Using the relation $\sum_{F\in\Fh}\sum_{T\in\Th[F]}\bullet=\sum_{T\in\Th}\sum_{F\in\Fh[T]}\bullet$ and recalling the definitions \eqref{eq:norm.1h} of $\norm[1,h]{{\cdot}}$ and \eqref{eq:norm.1T} of $\norm[1,T]{{\cdot}}$ yields \eqref{eq:sobolev.embeddings}.
\end{proof}

\subsection{Approximation properties of the discrete directional derivative}\label{sec:convegence.analysis:preliminary.results:GwT}

\begin{proof}[Proof of Proposition \ref{prop:GwT}]
  Set, for the sake of brevity $\uhvec{v}_T\coloneq\IT\vec{v}$.
  Writing \eqref{eq:GwT.GT} for $\uvec{w}_T=\uvec{v}_T=\uhvec{v}_T$ and summing and subtracting $\int_T(\hvec{v}_T\SCAL\GRAD)\vec{v}\SCAL\vec{z}$, we infer that it holds
  \begin{equation}\label{eq:approx.GwT:basic}
    \begin{aligned}
      \int_T\left(\GwT[\uhvec{v}_T]{\uhvec{v}_T}-(\vec{v}\SCAL\GRAD)\vec{v}\right)\SCAL\vec{z}
      &= \int_T[(\hvec{v}_T\SCAL\GT[2k])\uhvec{v}_T-(\hvec{v}_T\SCAL\GRAD)\vec{v}]\SCAL\vec{z}
      + \int_T\left[(\hvec{v}_T-\vec{v})\SCAL\GRAD\right]\vec{v}\SCAL\vec{z}
      \\
      &\quad
      + \sum_{F\in\Fh[T]}\int_F(\hvec{v}_F-\hvec{v}_T)\SCAL\normal_{TF}(\hvec{v}_F-\hvec{v}_T)\SCAL\vec{z}
      \eqcolon\term_1+\term_2+\term_3.
    \end{aligned}
  \end{equation}
  For the first term, using the orthogonality property \eqref{eq:GT.orth} with $\matr{\tau}=(\vec{z}\otimes\vlproj[T]{0}\vec{v})$, we can use a H\"{o}lder inequality with exponents $(4,2,4)$ to get
  \begin{align*}
  |\term_1|
  ={}& \left|
  \int_T(\GT[2k]\uhvec{v}_T-\GRAD\vec{v})
  \SSCAL\vec{z}\otimes(\hvec{v}_T-\vlproj[T]{0}\vec{v})
  \right|\\
  \le{}&\norm[L^4(T)^d]{\vec{v}-\vlproj[T]{0}\vec{v}}
  \norm[L^2(T)^{d\times d}]{\GT[2k]\uhvec{v}_T-\GRAD\vec{v}}
  \norm[L^4(T)^d]{\vec{z}},
  \end{align*}
  where we have additionally used the linearity, idempotency, and $L^4$-boundedness  of $\vlproj[T]{k}$ (see \cite[Lemma 3.2]{Di-Pietro.Droniou:17}) to write 
	$$
	\norm[L^4(T)^d]{\hvec{v}_T-\vlproj[T]{0}\vec{v}}=\norm[L^4(T)^d]{\vlproj[T]{k}(\vec{v}-\vlproj[T]{0}\vec{v})}\lsim\norm[L^4(T)^d]{\vec{v}-\vlproj[T]{0}\vec{v}}.
	$$
  Using the approximation properties \eqref{eq:lproj.approx:T} of the $L^2$-orthogonal projector with $l=0$, $p=4$, $s=1$, and $m=0$ for the first factor and
  the approximation properties \eqref{eq:approx.GT} of the gradient reconstruction with $\ell=2k$ for the second factor, we get
  \begin{equation}\label{eq:approx.GwT:T1}
    |\term_1|\lsim h_T^{k+1}\seminorm[H^{k+1}(T)^d]{\vec{v}}\seminorm[W^{1,4}(T)^d]{\vec{v}}\norm[L^4(T)^d]{\vec{z}}.
  \end{equation}

  For the second term, a H\"{o}lder inequality with exponents $(2,4,4)$ gives
  \begin{equation}\label{eq:approx.GwT:T2}
    |\term_2|\lsim\norm[L^2(T)^d]{\vec{v}-\hvec{v}_T}\seminorm[W^{1,4}(T)^d]{\vec{v}}\norm[L^4(T)^d]{\vec{z}}
    \lsim h_T^{k+1}\seminorm[H^{k+1}(T)^d]{\vec{v}}\seminorm[W^{1,4}(T)^d]{\vec{v}}\norm[L^4(T)^d]{\vec{z}},
  \end{equation}
  where we have used the optimal approximation properties \eqref{eq:lproj.approx:T} of the $L^2$-orthogonal projector with $l=k$, $p=2$, $s=k+1$, and $m=0$ to conclude.

  For the third term, we will need the following discrete trace inequality, valid for all $\alpha\in[1,\infty]$, all $T\in\Th$, and all $F\in\Fh[T]$:
  \begin{equation}\label{ineq:discrete_tr}
    \norm[L^\alpha(F)]{v}\lsim h_T^{-\frac{1}{\alpha}}\norm[L^\alpha(T)]{v}\qquad\forall v\in\Poly{l}(T),
  \end{equation}
	with hidden constant additionally depending on $\alpha$ and $l$ (but not on $v$ or $T$).
  After observing that, owing to the linearity, idempotency, and boundedness of $\vlproj[F]{k}$, it holds, for $\alpha\in\{2,4\}$,
\[
\norm[L^\alpha(F)^d]{\uhvec{v}_F-\uhvec{v}_T}=\norm[L^\alpha(F)^d]{\vlproj[F]{k}(\vec{v}-\vlproj[T]{k}\vec{v})}\lsim\norm[L^\alpha(F)^d]{\vec{v}-\vlproj[T]{k}\vec{v}},
\]
we can write
  \begin{align}
    |\term_3|
    \lsim{}&\sum_{F\in\Fh[T]}
    \norm[L^2(F)^d]{\vec{v}-\vlproj[T]{k}\vec{v}}    
    \norm[L^4(F)^d]{\vec{v}-\vlproj[T]{k}\vec{v}}
    \norm[L^4(F)^d]{\vec{z}}\nonumber\\
    \lsim{}& h_T^{k+1}\seminorm[H^{k+1}(T)^d]{\vec{v}}\seminorm[W^{1,4}(T)^d]{\vec{v}}\norm[L^4(T)^d]{\vec{z}},
\label{eq:approx.GwT:T3}
  \end{align}
  where we have used the approximation properties \eqref{eq:lproj.approx:FT} of the $L^2$-orthogonal projector with $l=k$ and, respectively, $\alpha=2$, $m=0$, and $s=k+1$ and $\alpha=4$, $m=0$, and $s=1$ to bound the first two factors inside the summation, and the discrete trace inequality \eqref{ineq:discrete_tr} with $\alpha=4$ to bound the third one.

  Plugging \eqref{eq:approx.GwT:T1}, \eqref{eq:approx.GwT:T2}, and \eqref{eq:approx.GwT:T3} into \eqref{eq:approx.GwT:basic} yields the conclusion.
\end{proof}

\subsection{Viscous term}\label{sec:convergence.analysis:preliminary.results:ah}

\begin{proof}[Proof of Lemma \ref{lem:ah}]
  (i) \emph{Stability and boundedness.}
  Let $\uvec{v}_h\in\Uh$.
  Summing \eqref{eq:aT:stability} over $T\in\Th$ and using the resulting equivalence in combination with \eqref{eq:ah}, we can write
  $$
  \mathrm{a}_h(\uvec{v}_h,\uvec{v}_h)
  = \sum_{T\in\Th}\mathrm{a}_T(\uvec{v}_T,\uvec{v}_T) + \sum_{F\in\Fhb}h_F^{-1}\norm[L^2(F)^d]{\vec{v}_F}^2
  \esim[C_{\mathrm{a}}^{-1}]\sum_{T\in\Th}\norm[1,T]{\uvec{v}_T}^2 + \sum_{F\in\Fhb}h_F^{-1}\norm[L^2(F)^d]{\vec{v}_F}^2.
  $$
  \medskip\\
  (ii) \emph{Consistency.} Let, for the sake of brevity, $\uhvec{w}_h\coloneq\Ih\vec{w}$ and set, for all $T\in\Th$, $\check{\vec{w}}_T\coloneq\rT\IT\vec{w}$.
  It follows from \cite[Lemma 3]{Di-Pietro.Ern.ea:14} (see also \cite[Theorems 1.1 and 1.2]{Di-Pietro.Droniou:17*1}) that it holds: For all $T\in\Th$,
  \begin{equation}\label{eq:eproj.approx}
    \norm[L^2(T)^{d\times d}]{\GRAD(\check{\vec{w}}_T-\vec{w})}
    + h_T^{\frac12}\norm[L^2(\partial T)^{d\times d}]{\GRAD(\check{\vec{w}}_T-\vec{w})}
    \lsim h_T^{k+1}\seminorm[H^{k+2}(T)^d]{\vec{w}}.
  \end{equation}
  Integrating by parts element by element and using the fact that $\vec{v}_F$ is single-valued for $F\in\Fhi$ to insert it into the second term, we can write
  \begin{equation}\label{eq:ah:consistency:1}
		\begin{aligned}
    \int_\Omega\LAPL\vec{w}\SCAL\vec{v}_h
    ={}& -\sum_{T\in\Th}\left(
    \int_T\GRAD\vec{w}\SSCAL\GRAD\vec{v}_T
    + \sum_{F\in\Fh[T]}\int_F\GRAD\vec{w}\normal_{TF}\SCAL(\vec{v}_F - \vec{v}_T)
    \right)\\
    &+ \sum_{F\in\Fhb}\int_F\GRAD\vec{w}\normal_F\SCAL\vec{v}_F.
		\end{aligned}
  \end{equation}
  On the other hand, plugging the definition \eqref{eq:aT} of $\mathrm{a}_T$ into \eqref{eq:ah}, and expanding, for all $T\in\Th$, $\rT\uvec{v}_T$ according to \eqref{eq:characterisation.rT} with $\vec{w}=\check{\vec{w}}_T$, we can write
  \begin{align}
      \mathrm{a}_h(\uhvec{w}_h,\uvec{v}_h)
      ={}&\sum_{T\in\Th}\left(
      \int_T\GRAD\check{\vec{w}}_T\SSCAL\GRAD\vec{v}_T
      + \sum_{F\in\Fh[T]}\int_F\GRAD\check{\vec{w}}_T\normal_{TF}\SCAL(\vec{v}_F - \vec{v}_T)
      \right)\nonumber\\
      &
      - \sum_{F\in\Fhb}\int_F\GRAD\check{\vec{w}}_{T_F}\normal_F\SCAL\vec{v}_F    
      + \sum_{T\in\Th}\mathrm{s}_T(\uhvec{w}_T,\uvec{v}_T),
	\label{eq:ah:consistency:2}
  \end{align}
  where, to cancel the remaining terms inside the sum over $F\in\Fhb$, we have used the fact that $\hvec{w}_F=\vec{0}$ for all $F\in\Fhb$ owing to the assumed regularity $\vec{w}\in H_0^1(\Omega)^d$.
  Summing \eqref{eq:ah:consistency:1} and \eqref{eq:ah:consistency:2}, and using Cauchy--Schwarz inequalities for the first and last terms, the H\"{o}lder inequality with exponents $(2,\infty,2)$ together with $\norm[L^\infty(F)^d]{\normal_{TF}}\le 1$ for the second term, and the definition \eqref{eq:norm.1T} of $\norm[1,T]{{\cdot}}$ and $h_F\le h_T$, it is inferred that
  $$
  \begin{aligned}
    \Big|
    \int_\Omega\LAPL\vec{w}\SCAL\vec{v}_h{}& + \mathrm{a}_h(\uhvec{w}_h,\uvec{v}_h)
    \Big|\\
    &\lsim
    \sum_{T\in\Th}\left(
    \norm[L^2(T)^{d\times d}]{\GRAD(\check{\vec{w}}_T-\vec{w})}
    + h_T^{\frac12}\norm[L^2(\partial T)^{d\times d}]{\GRAD(\check{\vec{w}}_T-\vec{w})}
    \right)\norm[1,T]{\uvec{v}_T}
    \\
    &\quad
    + \sum_{F\in\Fhb}h_{T_F}^{\frac12}\norm[L^2(F)^{d\times d}]{\GRAD(\check{\vec{w}}_{T_F}-\vec{w})}~h_F^{-\frac12}\norm[L^2(F)^d]{\vec{v}_F}
    \\
    &\quad
    + \sum_{T\in\Th}\mathrm{s}_T(\uhvec{w}_T,\uhvec{w}_T)^{\frac12}\mathrm{s}_T(\uvec{v}_T,\uvec{v}_T)^{\frac12}.
  \end{aligned}
  $$
  Using the approximation properties of \eqref{eq:eproj.approx} of $\uhvec{w}_T$ to estimate the terms in the first two lines and the consistency \eqref{eq:sT:consistency} of the viscous stabilisation bilinear form $\mathrm{s}_T$ together with the fact that $\mathrm{s}_T(\uvec{v}_T,\uvec{v}_T)^{\frac12}\lsim\norm[1,T]{\uvec{v}_T}$ owing to the local seminorm equivalence \eqref{eq:aT:stability} to estimate the term in the third line, we arrive at
  $$
  \begin{aligned}
    \left|
    \int_\Omega\LAPL\vec{w}\SCAL\vec{v}_h + \mathrm{a}_h(\uhvec{w}_h,\uvec{v}_h)
    \right|
    \lsim{}&
    \sum_{T\in\Th} h_T^{k+1}\seminorm[H^{k+2}(T)^d]{\vec{w}}\norm[1,T]{\uvec{v}_T}\\
    &+ \sum_{F\in\Fhb}  h_{T_F}^{k+1}\seminorm[H^{k+2}(T_F)^d]{\vec{w}}~h_F^{-\frac12}\norm[L^2(F)^d]{\vec{v}_F}
    \\
    \lsim{}&
    h^{k+1}\seminorm[H^{k+2}(\Th)^d]{\vec{w}}\norm[1,h]{\uvec{v}_h},
  \end{aligned}
  $$
  where the conclusion follows from Cauchy--Schwarz inequalities on the sums over $T\in\Th$ and $F\in\Fhb$.
  Passing to the supremum over the set $\left\{\uvec{v}_h\in\Uh\st\norm[1,h]{\uvec{v}_h}=1\right\}$ yields \eqref{eq:ah:consistency}.
\end{proof}

\subsection{Pressure-velocity coupling}\label{sec:convergence.analysis:preliminary.results:bh}

\begin{proof}[Proof of Lemma \ref{lem:bh}]
  (i) \emph{Consistency/1.}
  Using the definition \eqref{eq:bh} of the bilinear form $\mathrm{b}_h$, the commuting property \eqref{eq:DT:commuting} of the discrete divergence together with the fact that boundary unknowns in $\Ih\vec{v}$ vanish since $\vec{v}\in H_0^1(\Omega)^d$, we obtain, for all $q_h\in\Poly{k}(\Th)$, letting $q_T\coloneq q_{h|T}$ for all $T\in\Th$,
  $$
  \mathrm{b}_h(\Ih\vec{v},q_h)
  = -\sum_{T\in\Th}\int_T\lproj[T]{k}(\DIV\vec{v})q_T
  = -\sum_{T\in\Th}\int_T(\DIV\vec{v})q_T
  = b(\vec{v},q_h).
  $$
  \medskip\\
  (ii) \emph{Stability.} We proceed as in \cite[Lemma 17]{Botti.Di-Pietro.ea:18}.
  From the surjectivity of the continuous divergence operator from $H_0^1(\Omega)^d$ to $L_0^2(\Omega)$ (see, e.g., \cite[Section 2.2]{Girault.Raviart:86}), we infer the existence of $\vec{v}_{q_h}\in H_0^1(\Omega)^d$ such that $-\DIV\vec{v}_{q_h}=q_h$ and $\norm[H^1(\Omega)^d]{\vec{v}_{q_h}}\lsim\norm{q_h}$, with hidden constant depending only on $\Omega$.
  Using the above fact, we get
  $
  \norm{q_h}^2= -\int_\Omega(\DIV\vec{v}_{q_h})q_h = \mathrm{b}_h(\Ih\vec{v}_{q_h},q_h),
  $
  where we have used \eqref{eq:bh:consistency.2} with $\vec{v}=\vec{v}_{q_h}$.
  Hence, denoting by $\$$ the supremum in the right-hand side of \eqref{eq:bh:stability} and using the boundedness property \eqref{est:CI} of the global interpolator, we can write
  $
  \norm{q_h}^2
  \le\$\norm[1,h]{\Ih\vec{v}_{q_h}}
  \lsim\$\norm[H^1(\Omega)^d]{\vec{v}_{q_h}}
  \lsim\$\norm{q_h}.
  $
  Simplifying yields \eqref{eq:bh:stability}.
  \medskip\\
  (iii) \emph{Consistency/2.}
  Integrating by parts element by element and using the regularity $q\in H^1(\Omega)$ together with the fact that $\vec{v}_F$ is single-valued for all $F\in\Fhi$ to insert it into the second term, we can write
  \begin{equation}\label{eq:bh:consistency:1}
    \int_\Omega\GRAD q\SCAL\vec{v}_h
    = -\sum_{T\in\Th}\left(
    \int_T q(\DIV\vec{v}_T) + \sum_{F\in\Fh[T]}\int_F q(\vec{v}_F-\vec{v}_T)\SCAL\normal_{TF}
    \right)
    + \sum_{F\in\Fhb}\int_F q(\vec{v}_F\SCAL\normal_F).
  \end{equation}
  On the other hand, recalling the definition \eqref{eq:bh} of $\mathrm{b}_h$ and expanding, for all $T\in\Th$, $\DT\uvec{v}_T$ according to \eqref{eq:DT} with $\ell=k$ and $q=\lproj[T]{k}q_{|T}$, we have that
	\begin{equation}\label{eq:bh:consistency:2}
  \begin{aligned}
    \mathrm{b}_h(\uvec{v}_h,\lproj{k}q)
    ={}& -\sum_{T\in\Th}\left(
    \int_T q(\DIV\vec{v}_T) + \sum_{F\in\Fh[T]}\int_F\lproj[T]{k}q_{|T}(\vec{v}_F-\vec{v}_T)\SCAL\normal_{TF}
    \right)\\
    &+ \sum_{F\in\Fhb}\int_F \lproj[T_F]{k}q_{|T_F}(\vec{v}_F\SCAL\normal_F),
  \end{aligned}
	\end{equation}
  where we have used the definition \eqref{eq:lproj} of $\lproj[T]{k}$ together with the fact that $\DIV\vec{v}_T\in\Poly{k-1}(T)\subset\Poly{k}(T)$ to remove the projector from the first term inside the summation over $T\in\Th$.
  Subtracting \eqref{eq:bh:consistency:2} from \eqref{eq:bh:consistency:1}, taking absolute values, and using H\"{o}lder inequalities with exponents $(2,2,\infty)$ together with $\norm[L^\infty(F)^d]{\normal_{TF}}\le 1$ and $h_F\le h_T$, we get
  $$
  \begin{aligned}
    \left|\int_\Omega\GRAD q\SCAL\vec{v}_h - \mathrm{b}_h(\uvec{v}_h,\lproj{k}q)\right|
    \le{}&
    \sum_{T\in\Th}\sum_{F\in\Fh[T]} h_T^{\frac12}\norm[L^2(F)]{q-\lproj[T]{k}q_{|T}}~h_F^{-\frac12}\norm[L^2(F)^d]{\vec{v}_F-\vec{v}_T}
    \\
    &
    + \sum_{F\in\Fhb} h_{T_F}^{\frac12}\norm[L^2(F)]{q-\lproj[T_F]{k}q_{|T_F}}~h_F^{-\frac12}\norm[L^2(F)^d]{\vec{v}_F}
    \\
    \lsim{}&\sum_{T\in\Th}h_T^{k+1}\seminorm[H^{k+1}(T)]{q}\norm[1,T]{\uvec{v}_T}\\
    &+ \sum_{F\in\Fhb} h_{T_F}^{k+1}\seminorm[H^{k+1}(T_F)]{q}~h_F^{-\frac12}\norm[L^2(F)^d]{\vec{v}_F},
  \end{aligned}
  $$
  where we have concluded using the optimal approximation properties \eqref{eq:lproj.approx:FT} of the $L^2$-orthogonal projector with $l=k$, $p=2$, $s=k+1$, $m=0$.
  Using Cauchy--Schwarz inequalities on the sums over $T\in\Th$ and $F\in\Fhb$ and recalling the definition \eqref{eq:norm.1h} of $\norm[1,h]{{\cdot}}$ gives
  $$
  \left|\int_\Omega\GRAD q\SCAL\vec{v}_h - \mathrm{b}_h(\uvec{v}_h,\lproj{k}q)\right|
  \lsim h^{k+1}\seminorm[H^{k+1}(\Th)]{q}\norm[1,h]{\uvec{v}_h}.
  $$
  Passing to the supremum over the set $\left\{\uvec{v}_h\in\Uh\st\norm[1,h]{\uvec{v}_h}=1\right\}$ yields \eqref{eq:bh:consistency}.
\end{proof}

\subsection{Convective term}\label{sec:convergence.analysis:preliminary.results:th}

\begin{proof}[Proof of Lemma \ref{lem:th}]
  (i) \emph{Skew-symmetry and non-dissipation.}
  To prove \eqref{eq:th:skew-symmetry}, plug the definition \eqref{eq:tT} of $\mathrm{t}_T$ into \eqref{eq:th} and use the discrete integration by parts formula \eqref{eq:ibp} to write
  \begin{multline*}
    \sum_{T\in\Th}\left(
    \frac12\int_T\DT[2k]\uvec{w}_T(\vec{v}_T\SCAL\vec{z}_T)
    + \frac12\sum_{F\in\Fh[T]}\int_F(\vec{w}_F\SCAL\normal_{TF})(\vec{v}_F-\vec{v}_T)\SCAL(\vec{z}_F-\vec{z}_T)
    \right)
    \\
    =-\frac12\sum_{T\in\Th}\int_T\left(
    \GwT{\uvec{v}_T}\SCAL\vec{z}_T + \vec{v}_T\SCAL\GwT{\uvec{z}_T}
    \right)
    +\frac12\sum_{F\in\Fhb}\int_F(\vec{w}_F\SCAL\normal_F)\vec{v}_F\SCAL\vec{z}_F.
  \end{multline*}
  The non-dissipation property \eqref{eq:th:non-dissipation} immediately follows letting $\uvec{z}_h=\uvec{v}_h$ in \eqref{eq:th:skew-symmetry}.
  \medskip\\
  (ii) \emph{Boundedness.}
  Accounting for \eqref{eq:th:skew-symmetry}, it suffices to prove that it holds, for all $\uvec{w}_h,\uvec{v}_h,\uvec{z}_h\in\Uh$,
  \begin{equation}\label{eq:th:boundedness:T}
    \term\coloneq\left| \sum_{T\in\Th}\int_T\GwT{\uvec{v}_T}\SCAL\vec{z}_T\right|
    \lsim\norm[1,h]{\uvec{w}_h}\norm[1,h]{\uvec{v}_h}\norm[1,h]{\uvec{z}_h},
  \end{equation}
  then apply this bound twice exchanging the roles of $\uvec{v}_h$ and $\uvec{z}_h$.
  Recalling \eqref{eq:GwT.GT}, we can write
  \begin{equation}\label{eq:th:boundedness:basic}
    \term
    \le\left|
    \sum_{T\in\Th}\int_T(\vec{w}_T\SCAL\GT[2k])\uvec{v}_T\SCAL\vec{z}_T  
    \right| + \left|
    \sum_{T\in\Th}\sum_{F\in\Fh[T]}\int_F(\vec{w}_F-\vec{w}_T)\SCAL\normal_{TF}(\vec{v}_F-\vec{v}_T)\SCAL\vec{z}_T
    \right|
    \eqcolon \term_1 + \term_2.
  \end{equation}
  To estimate the first term, we preliminarily observe that $\norm[L^2(T)^{d\times d}]{\GT[2k]\uvec{v}_T}\lsim\norm[1,T]{\uvec{v}_T}$, as can be easily verified letting $\matr{\tau}=\GT[2k]\uvec{v}_T$ in \eqref{eq:GT} written for $\ell=2k$, then applying Cauchy--Schwarz and discrete trace inequalities (see \eqref{ineq:discrete_tr}) to bound the right-hand side.
  Then, H\"{o}lder inequalities with exponents $(4,2,4)$, first on the integral over $T$ then on the sum over $T\in\Th$, yield
  \begin{equation}\label{eq:th:boundedness:T1}
    \begin{aligned}
      \term_1
      &\le\sum_{T\in\Th}\norm[L^4(T)^d]{\vec{w}_T}\norm[L^2(T)^{d\times d}]{\GT[2k]\uvec{v}_T}\norm[L^4(T)^d]{\vec{z}_T}
      \\
      &\le\sum_{T\in\Th}\norm[L^4(T)^d]{\vec{w}_T}\norm[1,T]{\uvec{v}_T}\norm[L^4(T)^d]{\vec{z}_T}
      \\
      &\le\norm[L^4(\Omega)^d]{\vec{w}_h}\norm[1,h]{\uvec{v}_h}\norm[L^4(\Omega)^d]{\vec{z}_h}
      \lesssim\norm[1,h]{\uvec{w}_h}\norm[1,h]{\uvec{v}_h}\norm[1,h]{\uvec{z}_h},
    \end{aligned}
  \end{equation}
  where we have used the discrete Sobolev embedding \eqref{eq:sobolev.embeddings} with $q=4$ to conclude.
  For the second term, we will need the following reverse Lebesgue embeddings, proved in \cite[Lemma 5.1]{Di-Pietro.Droniou:17}:
  For all $(\alpha,\beta)\in [1,\infty]$, all $T\in\Th$, and all $F\in\Fh[T]$,
  \begin{alignat}{2}
    \label{ineq:reverse_leb}
    \norm[L^\alpha(F)]{v}&\lsim |F|^{\frac{1}{\alpha}-\frac{1}{\beta}}\norm[L^\beta(F)]{v}&\qquad&\forall v\in\Poly{l}(F).
  \end{alignat}
  We have that
  \begin{equation*}
    \begin{aligned}
      \term_2
      &\le\sum_{T\in\Th}\sum_{F\in\Fh[T]}\norm[L^4(F)^d]{\vec{w}_F-\vec{w}_T}\norm[L^2(F)^d]{\vec{v}_F-\vec{v}_T}\norm[L^4(F)^d]{\vec{z}_T}
      \\
      &\lsim\norm[L^4(\Omega)^d]{\vec{z}_h}\sum_{T\in\Th}\sum_{F\in\Fh[T]}h_F^{-\frac14}\norm[L^4(F)^d]{\vec{w}_F-\vec{w}_T}\norm[L^2(F)^d]{\vec{v}_F-\vec{v}_T}
      \\
      &\lsim\norm[L^4(\Omega)^d]{\vec{z}_h}\sum_{T\in\Th}\sum_{F\in\Fh[T]}h_F^{-\frac14}|F|^{-\frac14}\norm[L^2(F)^d]{\vec{w}_F-\vec{w}_T}\norm[L^2(F)^d]{\vec{v}_F-\vec{v}_T}
      \\
      &\lsim\norm[L^4(\Omega)^d]{\vec{z}_h}\sum_{T\in\Th}\sum_{F\in\Fh[T]}h_F^{-\frac12}\norm[L^2(F)^d]{\vec{w}_F-\vec{w}_T}~h_F^{-\frac12}\norm[L^2(F)^d]{\vec{v}_F-\vec{v}_T}\\
      &\lsim\norm[1,h]{\uvec{z}_h}\norm[1,h]{\uvec{w}_h}\norm[1,h]{\uvec{v}_h},
    \end{aligned}
  \end{equation*}
  where we have used a H\"older inequality with exponents $(4,\infty,2,4)$ together with the bound $\norm[L^\infty(F)^d]{\normal_{TF}}\le 1$ in the first line,
  the discrete trace inequality \eqref{ineq:discrete_tr} with $\alpha=4$ followed by $\norm[L^4(T)^d]{\vec{z}_T}\le\norm[L^4(\Omega)^d]{\vec{z}_h}$ for all $T\in\Th$ in the second line, the reverse Lebesgue embedding \eqref{ineq:reverse_leb} with $(\alpha,\beta)=(4,2)$ in the third line,
  the bound $h_F^{-\frac14}|F|^{-\frac14}\lsim h_F^{-\frac14-\frac{d-1}{4}}\lsim h_F^{-1}=h_F^{-\frac12}h_F^{-\frac12}$ (valid since $d\le 3$) in the fourth line,
  and the discrete Sobolev embedding \eqref{eq:sobolev.embeddings} with $q=4$ followed by a discrete Cauchy--Schwarz inequality on the sums over $T\in\Th$ and $F\in\Fh[T]$ and the definition \eqref{eq:norm.1h} of $\norm[1,h]{{\cdot}}$ to conclude.
  \medskip\\
  (iii) \emph{Consistency.}
  Set, for the sake of brevity, $\uhvec{w}_h\coloneq\Ih\vec{w}$.
  We decompose the argument of the supremum into the sum of the following terms:
  $$
  \begin{aligned}
    \term_1
    &\coloneq
    \sum_{T\in\Th}\int_T\left[
      (\vec{w}\SCAL\GRAD)\vec{w}-\GwT[\uhvec{w}_T]{\uhvec{w}_T}
      \right]\SCAL\vec{z}_T,
    \\
    \term_2
    &\coloneq
    \frac12\sum_{T\in\Th}\int_T\left[
      (\DIV\vec{w})\vec{w} - (\DT[2k]\uhvec{w}_T)\hvec{w}_T
      \right]\SCAL\vec{z}_T
    \\
    \term_3
    &\coloneq
    \frac12\sum_{T\in\Th}\sum_{F\in\Fh[T]}
    \int_F(\hvec{w}_F\SCAL\normal_{TF})(\hvec{w}_F-\hvec{w}_T)\SCAL(\vec{z}_F-\vec{z}_T),
    \\
    \term_4
    &\coloneq
    -\frac12\sum_{F\in\Fhb}\int_F(\hvec{w}_F\SCAL\normal_F)(\hvec{w}_F\SCAL\vec{z}_F).
  \end{aligned}
  $$
  Using the approximation properties \eqref{eq:approx.GwT} of the discrete directional derivative followed by the discrete Sobolev embedding \eqref{eq:sobolev.embeddings} with $q=4$, it is inferred for the first term:
  \begin{align}
    |\term_1|
    \lsim{}&
    h^{k+1}\seminorm[H^{k+1}(\Th)^d]{\vec{w}}\seminorm[W^{1,4}(\Omega)^d]{\vec{w}}\norm[L^4(\Omega)^d]{\vec{z}_h}\nonumber\\
    \lsim{}&
    h^{k+1}\seminorm[H^{k+1}(\Th)^d]{\vec{w}}\seminorm[W^{1,4}(\Omega)^d]{\vec{w}}\norm[1,h]{\uvec{z}_h}.
		\label{eq:th:consistency:T1}
  \end{align}
  Note that the assumption $\vec{w}\in H^1_0(\Omega)^d\cap W^{k+1,4}(\Th)^d$ ensures that $\vec{w}\in W^{1,4}(\Omega)^d$ since $\vec{w}_{|T}\in W^{k+1,4}(T)\subset W^{1,4}(T)$ for all $T\in\Th$, and the traces of $\vec{w}$ at the interfaces $F\in\Fhi$ are continuous (owing to $\vec{w}\in H^1(\Omega)^d$).

  The second term is further decomposed inserting $\pm(\DIV\vec{w})\hvec{w}_T\SCAL\vec{z}_T$:
  $$
  \term_2 =
  \frac12\sum_{T\in\Th}\int_T(\DIV\vec{w}-\DT[2k]\uhvec{w}_T)\hvec{w}_T\SCAL\vec{z}_T
  + \frac12\sum_{T\in\Th}\int_T(\DIV\vec{w})(\vec{w}-\hvec{w}_T)\SCAL\vec{z}_T
  \eqcolon\term_{2,1} + \term_{2,2}.
  $$
  After observing that $(\DIV\vec{w}-\DT[2k]\uhvec{w}_T)$ is $L^2$-orthogonal to functions in $\Poly{k}(T)$ as a consequence of \eqref{eq:GT.orth}, and that $\vlproj[T]{0}\vec{w}\SCAL\vec{z}_T\in\Poly{k}(T)$, we can write
  $$
  \begin{aligned}
    |\term_{2,1}|
    &= \frac12\left|
    \sum_{T\in\Th}\int_T(\DIV\vec{w}-\DT[2k]\uhvec{w}_T)(\hvec{w}_T-\vlproj[T]{0}\vec{w})\SCAL\vec{z}_T  
    \right|
    \\
    &\lsim
    \sum_{T\in\Th}\norm[L^2(T)]{\DIV\vec{w}-\DT[2k]\uhvec{w}_T}
    \norm[L^4(T)^d]{\hvec{w}_T-\vlproj[T]{0}\vec{w}}
    \norm[L^4(T)^d]{\vec{z}_T}
    \\
    &\lsim h^{k+1}
    \seminorm[H^{k+1}(\Th)^d]{\vec{w}}\seminorm[W^{1,4}(\Omega)^d]{\vec{w}}\norm[1,h]{\uvec{z}_h}.
  \end{aligned}
  $$
  To pass from the second to the third line, we have used the approximation properties \eqref{eq:approx.DT} of the divergence reconstruction with $\ell=2k$ to bound the first factor,
  the linearity, idempotency, and $L^4$-boundedness of $\vlproj[T]{k}$ followed by the approximation properties \eqref{eq:lproj.approx:T} of the $L^2$-orthogonal projector with $l=0$, $p=4$, $m=0$, and $s=1$ to estimate the second factor as follows:
  $$
  \norm[L^4(T)^d]{\hvec{w}_T-\vlproj[T]{0}\vec{w}}
  = \norm[L^4(T)^d]{\vlproj[T]{k}(\vec{w}-\vlproj[T]{0}\vec{w})}
  \lsim \norm[L^4(T)^d]{\vec{w}-\vlproj[T]{0}\vec{w}}
  \lsim h_T\seminorm[W^{1,4}(T)^d]{\vec{w}},
  $$
  the discrete Sobolev embedding \eqref{eq:sobolev.embeddings} for $q=4$ for the third factor, and a discrete H\"{o}lder inequality on the sum over $T\in\Th$ with exponents $(2,4,4)$ to conclude.
  On the other hand, H\"{o}lder inequalities with exponents $(4,2,4)$ followed by the approximation properties \eqref{eq:lproj.approx:T} of the $L^2$-orthogonal projector with $l=k$, $p=2$, $m=0$, and $s=k+1$ give for the second contribution
  $$
  |\term_{2,2}|
  \lsim h^{k+1}\seminorm[H^{k+1}(\Th)^d]{\vec{w}}\seminorm[W^{1,4}(\Omega)^d]{\vec{w}}\norm[L^4(\Omega)^d]{\vec{z}_h}
  \lsim h^{k+1}\seminorm[H^{k+1}(\Th)^d]{\vec{w}}\seminorm[W^{1,4}(\Omega)^d]{\vec{w}}\norm[1,h]{\uvec{z}_h},
  $$
  where we have used the discrete Sobolev embedding \eqref{eq:sobolev.embeddings} for $q=4$ to conclude.
  Collecting the above bounds, we arrive at
  \begin{equation}\label{eq:th:consistency:T2}
    |\term_2|\lsim h^{k+1}\seminorm[H^{k+1}(\Th)^d]{\vec{w}}\seminorm[W^{1,4}(\Omega)^d]{\vec{w}}\norm[1,h]{\uvec{z}_h}.
  \end{equation}

  To estimate the third term, using a H\"{o}lder inequality with exponents $(4,\infty,4,2)$ we obtain
  \begin{equation}\label{est:th.consistency:T3}
  |\term_3|
  \le
  \sum_{T\in\Th}\sum_{F\in\Fh[T]}
  \norm[L^4(F)^d]{\hvec{w}_F}\norm[L^4(F)^d]{\hvec{w}_F-\hvec{w}_T}\norm[L^2(F)^d]{\vec{z}_F-\vec{z}_T}.
  \end{equation}
  For the first factor inside the summations, we use the $L^4$-boundedness of $\vlproj[F]{k}$ followed by a local trace inequality in $L^4$ (see, e.g., \cite[Eq. (A.10)]{Di-Pietro.Droniou:17}) and the fact that $h_T\le {\rm  diam}(\Omega)\lsim 1$ to write
  $$
  \norm[L^4(F)^d]{\hvec{w}_F}
  \lsim\norm[L^4(F)^d]{\vec{w}}
  \lsim h_T^{-\frac14}\left(\norm[L^4(T)^d]{\vec{w}} + h_T\norm[L^4(T)^{d\times d}]{\GRAD\vec{w}}
  \right)
  \lsim h_T^{-\frac14}\norm[W^{1,4}(T)^d]{\vec{w}}.
  $$
  For the second factor, using the linearity, idempotency, and $L^4$-boundedness of $\vlproj[F]{k}$ followed by the optimal approximation properties of $\vlproj[T]{k}$ we obtain
  $$
  \norm[L^4(F)^d]{\hvec{w}_F-\hvec{w}_T}
  = \norm[L^4(F)^d]{\vlproj[F]{k}(\vec{w}-\hvec{w}_T)}
  \lsim\norm[L^4(F)^d]{\vec{w}-\hvec{w}_T}
  \lsim h_T^{k+\frac34}\seminorm[W^{k+1,4}(T)^d]{\vec{w}}.
  $$
  Collecting the above estimates, we can go on writing
  \begin{align}
      |\term_3|
      &\lsim
      \sum_{T\in\Th}\sum_{F\in\Fh[T]} h_T^{-\frac14}\norm[W^{1,4}(T)^d]{\vec{w}}~h_T^{k+\frac34}\seminorm[W^{k+1,4}(T)^d]{\vec{w}}~\norm[L^2(F)^d]{\vec{z}_F-\vec{z}_T}
      \nonumber\\
      &\lsim h^{k+1}\norm[W^{1,4}(\Omega)^d]{\vec{w}}\seminorm[W^{k+1,4}(\Th)^d]{\vec{w}}\left(
      \sum_{T\in\Th}\sum_{F\in\Fh[T]}h_F^{-1}\norm[L^2(F)^d]{\vec{z}_F-\vec{z}_T}^2
      \right)^{\frac12}
      \nonumber\\
      &\lsim h^{k+1}\norm[W^{1,4}(\Omega)^d]{\vec{w}}\seminorm[W^{k+1,4}(\Th)^d]{\vec{w}}\norm[1,h]{\uvec{z}_h},
		\label{eq:th:consistency:T3}
  \end{align}
  where we have used H\"{o}lder inequalities with exponents $(4,4,2)$ on the sum over $T\in\Th$ and $F\in\Fh[T]$ together with $h_F\le h_T\le h$ to pass to the second line, and the definitions \eqref{eq:norm.1h} of $\norm[1,h]{{\cdot}}$ and \eqref{eq:norm.1T} of $\norm[1,T]{{\cdot}}$ to conclude.
  
  Finally, after observing that $\vec{w}$ vanishes on $\partial\Omega$ owing to the assumed regularity, we have for the fourth term
  \begin{equation}\label{eq:th:consistency:T4}
    \term_4=0.
  \end{equation}
  Collecting the bounds \eqref{eq:th:consistency:T1}, \eqref{eq:th:consistency:T2}, \eqref{eq:th:consistency:T3}, and \eqref{eq:th:consistency:T4}, and observing that $\seminorm[H^{k+1}(\Th)^d]{\vec{w}}\lsim\seminorm[W^{k+1,4}(\Th)^d]{\vec{w}}$, the conclusion follows.
\end{proof}

\subsection{Convective stabilisation}\label{sec:convergence.analysis:preliminary.results:jh}

\begin{proof}[Proof of Lemma \ref{lem:prop.jT}]
  (i) \emph{Continuity.} Since $\rho$ is Lipschitz-continuous,
\begin{align*}
\left|\frac{\nu}{h_F}\rho(\Pe(\vec{v}_F))-\frac{\nu}{h_F}\rho(\Pe(\vec{w}_F))|\right|
\lsim{}& \frac{\nu}{h_F}|\Pe(\vec{v}_F)-\Pe(\vec{w}_F)|\\
\le{}&|\vec{v}_F-\vec{w}_F|.
\end{align*}
Hence, setting $\uvec{e}_h\coloneq\uvec{v}_h-\uvec{w}_h$,
\begin{align}
\Big|\mathrm{j}_h(\uvec{v}_h;\uvec{z}_h,\uvec{z}'_h){}&-\mathrm{j}_h(\uvec{w}_h;\uvec{z}_h,\uvec{z}'_h)\Big|\nonumber\\
\lsim{}& \sum_{T\in\Th[T]}\sum_{F\in\Fh}\int_F |\vec{e}_F|\,|\vec{z}_F-\vec{z}_T|\,|\vec{z}'_F-\vec{z}'_T|
	+\sum_{F\in\Fhb} \int_F |\vec{e}_F|\,|\vec{z}_F|\,|\vec{z}'_F|\nonumber\\
\le{}& \sum_{T\in\Th[T]}\sum_{F\in\Fh}\norm[L^\infty(F)^d]{\vec{e}_F} \norm[L^2(F)^d]{\vec{z}_F-\vec{z}_T}\norm[L^2(F)^d]{\vec{z}'_F-\vec{z}'_T}\nonumber\\
&	+\sum_{F\in\Fhb} \norm[L^\infty(F)^d]{\vec{e}_F}\norm[L^2(F)^d]{\vec{z}_F}\norm[L^2(F)^d]{\vec{z}'_F}\nonumber\\
\lsim{}&\sum_{T\in\Th[T]}\sum_{F\in\Fh}|F|^{-\frac12}h_F\norm[L^2(F)^d]{\vec{e}_F}~ h_F^{-\frac12}\norm[L^2(F)^d]{\vec{z}_F-\vec{z}_T}~h_F^{-\frac12}\norm[L^2(F)^d]{\vec{z}'_F-\vec{z}'_T}\nonumber\\
&	+\sum_{F\in\Fhb} |F|^{-\frac12}h_F\norm[L^2(F)^d]{\vec{e}_F}~ h_F^{-\frac12}\norm[L^2(F)^d]{\vec{z}_F}~ h_F^{-\frac12}\norm[L^2(F)^d]{\vec{z}'_F},
\label{th:est.jh}\end{align}
where the second inequality follows from the H\"older inequality with exponents $(\infty,2,2)$, and the conclusion is obtained using the reverse Lebesgue inequality \eqref{ineq:reverse_leb} with $(\alpha,\beta)=(\infty,2)$. We then write, for $F\in\Fh[T]$,
\begin{align*}
	|F|^{-\frac12}h_F\norm[L^2(F)^d]{\vec{e}_F}\le{}&
	|F|^{-\frac12}h_F\norm[L^2(F)^d]{\vec{e}_F-\vec{e}_T}+|F|^{-\frac12}h_F\norm[L^2(F)^d]{\vec{e}_T}\\
	\le{}&|F|^{-\frac12}h_F^{\frac32}\norm[1,h]{\uvec{e}_h}+|F|^{-\frac12}h_Fh_T^{-\frac12}|T|^{\frac14}\norm[L^4(T)^d]{\vec{e}_T}\\
	\le{}& h^{1-\frac{d}{4}}\norm[1,h]{\uvec{e}_h},
\end{align*}
where we have used the triangle inequality in the first line, followed by $\norm[L^2(F)^d]{\vec{e}_F-\vec{e}_T}\le h_F^{\frac12}\norm[1,h]{\uvec{e}_h}$, the discrete trace inequality \eqref{ineq:discrete_tr} with $\alpha=2$, and the H\"older inequality $\norm[L^2(T)^d]{\vec{e}_T}\le |T|^{\frac14}\norm[L^4(T)^d]{\vec{e}_T}$ in the second line. The conclusion follows from $|F|\esim h_F^{d-1}$, $|T|\esim h_T^d$, $h_F\le h_T$, $h_T^{2-\frac{d}{2}}\le {\rm diam}(\Omega)^{1-\frac{d}{4}}h^{1-\frac{d}{4}}$ and, owing to the Sobolev embedding \eqref{eq:sobolev.embeddings}, $\norm[L^4(T)^d]{\vec{e}_T}\le \norm[L^4(\Omega)^d]{\vec{e}_T}\lesssim \norm[1,h]{\uvec{e}_h}$.

The proof of \eqref{est:jh.cont} is complete by plugging this estimate on $|F|^{-\frac12}h_F\norm[L^2(F)^d]{\vec{e}_F}$ into \eqref{th:est.jh}, by using Cauchy--Schwarz inequalities and by recalling the definition \eqref{eq:norm.1h} of $\norm[1,h]{{\cdot}}$.

\medskip

	(ii) \emph{Consistency.}
	Let $\uhvec{w}_h\coloneq\Ih\vec{w}$. Since $\rho$ is Lipschitz-continuous and $\rho(0)=0$ we have $|\rho(s)|\lsim |s|$ for all $s\in\Real$. Hence, $|\Pe(\hvec{w}_F)|\lsim h_F \frac{|\hvec{w}_F|}{\nu}$. This firstly shows that the boundary penalisation in $\mathrm{j}_h(\uhvec{w}_h;\uhvec{w}_h,\uvec{z}_h)$ vanishes (since $\uhvec{w}_F=0$ for all $F\in\Fhb$, given that $\vec{w}=0$ on $\partial\Omega$), and then, by H\"older's inequality with exponents $(4,4,2)$,
	\begin{align*}
	|\mathrm{j}_h(\uhvec{w}_h;\uhvec{w}_h,\uvec{z}_h)|\lsim{}& \sum_{T\in\Th}\sum_{F\in\Fh[T]}\int_F
	|\hvec{w}_F||\hvec{w}_F-\hvec{w}_T|\,|\vec{z}_F-\vec{z}_T|\\
	\lsim{}&\sum_{T\in\Th}\sum_{F\in\Fh[T]}
		\norm[L^4(F)^d]{\hvec{w}_F}\norm[L^4(F)^d]{\hvec{w}_F-\hvec{w}_T}\norm[L^2(F)^d]{\vec{z}_F-\vec{z}_T}. 
	\end{align*}
	This right-hand side is the same as in \eqref{est:th.consistency:T3}, and the estimates performed on $\term_3$ in the proof of Lemma \ref{lem:th} (see \eqref{eq:th:consistency:T3}) therefore show that \eqref{est:jh.const} holds.
\end{proof}


\section*{Acknowledgements}
The work of the second author was partially supported by \emph{Agence Nationale de la Recherche} grants HHOMM (ANR-15-CE40-0005) and fast4hho (ANR-17-CE23-0019).
The work of the third author was partially supported by the Australian Government through the Australian Research Council's Discovery Projects funding scheme (project number DP170100605).
The first author also acknowledges the support of Universit\`{a} di Bergamo through STaRs ``Supporting Talented Researchers'' fellowship.

\bibliographystyle{elsarticle-num}

\begin{thebibliography}{10}
\expandafter\ifx\csname url\endcsname\relax
  \def\url#1{\texttt{#1}}\fi
\expandafter\ifx\csname urlprefix\endcsname\relax\def\urlprefix{URL }\fi
\expandafter\ifx\csname href\endcsname\relax
  \def\href#1#2{#2} \def\path#1{#1}\fi

\bibitem{Temam:79}
R.~Temam, Navier-{S}tokes equations, revised Edition, Vol.~2 of Studies in
  Mathematics and its Applications, North-Holland Publishing Co., Amsterdam-New
  York, 1979, theory and numerical analysis, With an appendix by F. Thomasset.

\bibitem{Di-Pietro.Ern:15}
D.~A. Di~Pietro, A.~Ern, A hybrid high-order locking-free method for linear
  elasticity on general meshes, Comput. Meth. Appl. Mech. Engrg. 283 (2015)
  1--21.
\newblock \href {http://dx.doi.org/10.1016/j.cma.2014.09.009}
  {\path{doi:10.1016/j.cma.2014.09.009}}.

\bibitem{Botti.Di-Pietro:18}
L.~Botti, D.~A. Di~Pietro, Numerical assessment of {Hybrid High-Order} methods
  on curved meshes and comparison with discontinuous {Galerkin} methods, J.
  Comput. Phys. 370 (2018) 58--84.
\newblock \href {http://dx.doi.org/10.1016/j.jcp.2018.05.017}
  {\path{doi:10.1016/j.jcp.2018.05.017}}.

\bibitem{Karakashian.Katsaounis:00}
O.~Karakashian, T.~Katsaounis, A discontinuous {G}alerkin method for the
  incompressible {N}avier--{S}tokes equations, in: Discontinuous {G}alerkin
  methods ({N}ewport, {RI}, 1999), Vol.~11 of Lect. Notes Comput. Sci. Eng.,
  Springer, Berlin, 2000, pp. 157--166.
\newblock \href {http://dx.doi.org/10.1007/978-3-642-59721-3\_11}
  {\path{doi:10.1007/978-3-642-59721-3\_11}}.

\bibitem{Cockburn.Kanschat.ea:05}
B.~Cockburn, G.~Kanschat, D.~Sch\"{o}tzau, A locally conservative {LDG} method
  for the incompressible {N}avier-{S}tokes equations, Math. Comp. 74~(251)
  (2005) 1067--1095.
\newblock \href {http://dx.doi.org/10.1090/S0025-5718-04-01718-1}
  {\path{doi:10.1090/S0025-5718-04-01718-1}}.

\bibitem{Girault.Riviere.ea:05}
V.~Girault, B.~Rivi\`{e}re, M.~F. Wheeler, A discontinuous {G}alerkin method
  with nonoverlapping domain decomposition for the {S}tokes and
  {N}avier-{S}tokes problems, Math. Comp. 74~(249) (2005) 53--84.
\newblock \href {http://dx.doi.org/10.1090/S0025-5718-04-01652-7}
  {\path{doi:10.1090/S0025-5718-04-01652-7}}.

\bibitem{Bassi.Crivellini.ea:06}
F.~Bassi, A.~Crivellini, D.~A. Di~Pietro, S.~Rebay, An artificial
  compressibility flux for the discontinuous {Galerkin} solution of the
  incompressible {Navier-Stokes} equations, J. Comput. Phys. 218~(2) (2006)
  794--815.
\newblock \href {http://dx.doi.org/10.1016/j.jcp.2006.03.006}
  {\path{doi:10.1016/j.jcp.2006.03.006}}.

\bibitem{Cockburn.Kanschat.ea:07}
B.~Cockburn, G.~Kanschat, D.~Sch\"otzau, A note on discontinuous {G}alerkin
  divergence-free solutions of the {N}avier-{S}tokes equations, J. Sci. Comput.
  31~(1-2) (2007) 61--73.
\newblock \href {http://dx.doi.org/10.1007/s10915-006-9107-7}
  {\path{doi:10.1007/s10915-006-9107-7}}.

\bibitem{Mozolevski.Suli.ea:07}
I.~Mozolevski, E.~S\"uli, P.~R. B\"osing, Discontinuous {G}alerkin finite
  element approximation of the two-dimensional {N}avier-{S}tokes equations in
  stream-function formulation, Comm. Numer. Methods Engrg. 23~(6) (2007)
  447--459.
\newblock \href {http://dx.doi.org/10.1002/cnm.944}
  {\path{doi:10.1002/cnm.944}}.

\bibitem{Di-Pietro.Ern:10}
D.~A. Di~Pietro, A.~Ern, Discrete functional analysis tools for discontinuous
  {G}alerkin methods with application to the incompressible {N}avier--{S}tokes
  equations, Math. Comp. 79 (2010) 1303--1330.
\newblock \href {http://dx.doi.org/10.1090/S0025-5718-10-02333-1}
  {\path{doi:10.1090/S0025-5718-10-02333-1}}.

\bibitem{Botti.Di-Pietro:11}
L.~Botti, D.~A. Di~Pietro, A pressure-correction scheme for
  convection-dominated incompressible flows with discontinuous velocity and
  continuous pressure, J. Comput. Phys. 230~(3) (2011) 572--585.
\newblock \href {http://dx.doi.org/10.1016/j.jcp.2010.10.004}
  {\path{doi:10.1016/j.jcp.2010.10.004}}.

\bibitem{Riviere.Sardar:14}
B.~Rivi\`{e}re, S.~Sardar, Penalty-free discontinuous {G}alerkin methods for
  incompressible {N}avier-{S}tokes equations, Math. Models Methods Appl. Sci.
  24~(6) (2014) 1217--1236.
\newblock \href {http://dx.doi.org/10.1142/S0218202513500826}
  {\path{doi:10.1142/S0218202513500826}}.

\bibitem{Tavelli.Dumbser:14}
M.~Tavelli, M.~Dumbser, A staggered semi-implicit discontinuous {G}alerkin
  method for the two dimensional incompressible {N}avier-{S}tokes equations,
  Appl. Math. Comput. 248 (2014) 70--92.
\newblock \href {http://dx.doi.org/10.1016/j.amc.2014.09.089}
  {\path{doi:10.1016/j.amc.2014.09.089}}.

\bibitem{Di-Pietro:12}
D.~A. Di~Pietro, Cell centered {Galerkin} methods for diffusive problems,
  ESAIM: Math. Model Numer. Anal. 46~(1) (2012) 111--144.
\newblock \href {http://dx.doi.org/10.1051/m2an/2011016}
  {\path{doi:10.1051/m2an/2011016}}.

\bibitem{Nguyen.Peraire.ea:11}
N.~Nguyen, J.~Peraire, B.~Cockburn, An implicit high-order hybridizable
  discontinuous {Galerkin} method for the incompressible {Navier--Stokes}
  equations, J. Comput. Phys. 230 (2011) 1147--1170.
\newblock \href {http://dx.doi.org/10.1016/j.jcp.2010.10.032}
  {\path{doi:10.1016/j.jcp.2010.10.032}}.

\bibitem{Giorgiani.Fernandez-Mendez.ea:14}
G.~Giorgiani, S.~Fern\'{a}ndez-M\'{e}ndez, A.~Huerta, Hybridizable
  {Discontinuous Galerkin} with degree adaptivity for the incompressible
  {Navier--Stokes} equations, Computers \& Fluids 98 (2014) 196--208.
\newblock \href {http://dx.doi.org/10.1016/j.compfluid.2014.01.011}
  {\path{doi:10.1016/j.compfluid.2014.01.011}}.

\bibitem{Qiu.Shi:16}
W.~Qiu, K.~Shi, A superconvergent {HDG} method for the incompressible
  {Navier--Stokes} equations on general polyhedral meshes, IMA J. Numer. Anal.
  36~(4) (2016) 1943--1967.
\newblock \href {http://dx.doi.org/10.1093/imanum/drv067}
  {\path{doi:10.1093/imanum/drv067}}.

\bibitem{Ueckermann.Lermusiaux:16}
M.~P. Ueckermann, P.~F.~J. Lermusiaux, Hybridizable discontinuous {G}alerkin
  projection methods for {N}avier-{S}tokes and {B}oussinesq equations, J.
  Comput. Phys. 306 (2016) 390--421.
\newblock \href {http://dx.doi.org/10.1016/j.jcp.2015.11.028}
  {\path{doi:10.1016/j.jcp.2015.11.028}}.

\bibitem{Cesmelioglu.Cockburn.ea:17}
A.~\c{C}e\c{s}melio\u{g}lu, B.~Cockburn, W.~Qiu, Analysis of an {HDG} method
  for the incompressible {Navier--Stokes} equations, Math. Comp. 86 (2017)
  1643--1670.
\newblock \href {http://dx.doi.org/10.1090/mcom/3195}
  {\path{doi:10.1090/mcom/3195}}.

\bibitem{Di-Pietro.Krell:18}
D.~A. Di~Pietro, S.~Krell, A {Hybrid High-Order} method for the steady
  incompressible {Navier--Stokes} problem, J. Sci. Comput. 74~(3) (2018)
  1677--1705.
\newblock \href {http://dx.doi.org/10.1007/s10915-017-0512-x}
  {\path{doi:10.1007/s10915-017-0512-x}}.

\bibitem{Beirao-da-Veiga.Lovadina.ea:18}
L.~Beir\~{a}o~da Veiga, C.~Lovadina, G.~Vacca, Virtual {E}lements for the
  {N}avier--{S}tokes {P}roblem on {P}olygonal {M}eshes, SIAM J. Numer. Anal.
  56~(3) (2018) 1210--1242.
\newblock \href {http://dx.doi.org/10.1137/17M1132811}
  {\path{doi:10.1137/17M1132811}}.

\bibitem{Beirao-da-Veiga.Lovadina.ea:17}
L.~Beir\~{a}o~da Veiga, C.~Lovadina, G.~Vacca, Divergence free {Virtual
  Elements} for the {Stokes} problem on polygonal meshes, ESAIM: Math. Model.
  Numer. Anal. (M2AN) 51~(2) (2017) 509--535.
\newblock \href {http://dx.doi.org/10.1051/m2an/2016032}
  {\path{doi:10.1051/m2an/2016032}}.

\bibitem{Boffi.Di-Pietro:18}
D.~Boffi, D.~A. Di~Pietro, Unified formulation and analysis of mixed and primal
  discontinuous skeletal methods on polytopal meshes, ESAIM: Math. Model Numer.
  Anal. 52~(1) (2018) 1--28.
\newblock \href {http://dx.doi.org/10.1051/m2an/2017036}
  {\path{doi:10.1051/m2an/2017036}}.

\bibitem{Di-Pietro.Droniou.ea:18}
D.~A. Di~Pietro, J.~Droniou, G.~Manzini, {Discontinuous Skeletal Gradient
  Discretisation Methods} on polytopal meshes, J. Comput. Phys. 355 (2018)
  397--425.
\newblock \href {http://dx.doi.org/10.1016/j.jcp.2017.11.018}
  {\path{doi:10.1016/j.jcp.2017.11.018}}.

\bibitem{Aghili.Boyaval.ea:15}
J.~Aghili, S.~Boyaval, D.~A. Di~Pietro, Hybridization of {Mixed High-Order}
  methods on general meshes and application to the {Stokes} equations, Comput.
  Meth. Appl. Math. 15~(2) (2015) 111--134.
\newblock \href {http://dx.doi.org/10.1515/cmam-2015-0004}
  {\path{doi:10.1515/cmam-2015-0004}}.

\bibitem{Di-Pietro.Ern.ea:16*1}
D.~A. Di~Pietro, A.~Ern, A.~Linke, F.~Schieweck, A discontinuous skeletal
  method for the viscosity-dependent {Stokes} problem, Comput. Meth. Appl.
  Mech. Engrg. 306 (2016) 175--195.
\newblock \href {http://dx.doi.org/10.1016/j.cma.2016.03.033}
  {\path{doi:10.1016/j.cma.2016.03.033}}.

\bibitem{Botti.Di-Pietro.ea:18}
L.~Botti, D.~A. Di~Pietro, J.~Droniou, A {Hybrid High-Order} discretisation of
  the {Brinkman} problem robust in the {Darcy} and {Stokes} limits, Comput.
  Methods Appl. Mech. Engrg. 341 (2018) 278--310.
\newblock \href {http://dx.doi.org/10.1016/j.cma.2018.07.004}
  {\path{doi:10.1016/j.cma.2018.07.004}}.

\bibitem{Di-Pietro.Droniou.ea:15}
D.~A. Di~Pietro, J.~Droniou, A.~Ern, A discontinuous-skeletal method for
  advection-diffusion-reaction on general meshes, SIAM J. Numer. Anal. 53~(5)
  (2015) 2135--2157.
\newblock \href {http://dx.doi.org/10.1137/140993971}
  {\path{doi:10.1137/140993971}}.

\bibitem{Kovasznay:48}
L.~I.~G. Kovasznay, Laminar flow behind a two-dimensional grid, Mathematical
  Proceedings of the Cambridge Philosophical Society 44~(1) (1948) 58--62.
\newblock \href {http://dx.doi.org/10.1017/S0305004100023999}
  {\path{doi:10.1017/S0305004100023999}}.

\bibitem{Di-Pietro.Ern.ea:14}
D.~A. Di~Pietro, A.~Ern, S.~Lemaire, An arbitrary-order and compact-stencil
  discretization of diffusion on general meshes based on local reconstruction
  operators, Comput. Meth. Appl. Math. 14~(4) (2014) 461--472.
\newblock \href {http://dx.doi.org/10.1515/cmam-2014-0018}
  {\path{doi:10.1515/cmam-2014-0018}}.

\bibitem{Boffi.Botti.ea:16}
D.~Boffi, M.~Botti, D.~A. Di~Pietro, A nonconforming high-order method for the
  {Biot} problem on general meshes, SIAM J. Sci. Comput. 38~(3) (2016)
  A1508--A1537.
\newblock \href {http://dx.doi.org/10.1137/15M1025505}
  {\path{doi:10.1137/15M1025505}}.

\bibitem{Droniou.Eymard.ea:18}
J.~Droniou, R.~Eymard, T.~Gallou\"et, C.~Guichard, R.~Herbin, The gradient
  discretisation method, Vol.~82 of Math\'ematiques et Applications, Springer,
  2018.
\newblock \href {http://dx.doi.org/10.1007/978-3-319-79042-8}
  {\path{doi:10.1007/978-3-319-79042-8}}.

\bibitem{Di-Pietro.Tittarelli:18}
D.~A. Di~Pietro, R.~Tittarelli,
  \href{http://arxiv.org/abs/1703.05136}{Numerical Methods for PDEs}, no.~15,
  Springer, 2018, Ch. An introduction to Hybrid High-Order methods, sEMA-SIMAI.
\newline\urlprefix\url{http://arxiv.org/abs/1703.05136}

\bibitem{Di-Pietro.Droniou:17}
D.~A. Di~Pietro, J.~Droniou, A {Hybrid High-Order} method for {Leray--Lions}
  elliptic equations on general meshes, Math. Comp. 86~(307) (2017) 2159--2191.
\newblock \href {http://dx.doi.org/10.1090/mcom/3180}
  {\path{doi:10.1090/mcom/3180}}.

\bibitem{Dupont.Scott:80}
T.~Dupont, R.~Scott, Polynomial approximation of functions in {S}obolev spaces,
  Math. Comp. 34~(150) (1980) 441--463.
\newblock \href {http://dx.doi.org/10.2307/2006095}
  {\path{doi:10.2307/2006095}}.

\bibitem{Brenner.Scott:08}
S.~C. Brenner, L.~R. Scott, The mathematical theory of finite element methods,
  3rd Edition, Vol.~15 of Texts in Applied Mathematics, Springer, New York,
  2008.
\newblock \href {http://dx.doi.org/10.1007/978-0-387-75934-0}
  {\path{doi:10.1007/978-0-387-75934-0}}.

\bibitem{Anderson.Droniou:18}
D.~Anderson, J.~Droniou, \href{https://arxiv.org/abs/1707.04038}{An arbitrary
  order scheme on generic meshes for miscible displacements in porous media},
  SIAM J. Sci. Comput. 40 (4) (2018) B1020-B1054.
\newblock \href {http://dx.doi.org/10.1137/17M1138807}
  {\path{doi:10.1137/17M1138807}}.
\newline\urlprefix\url{https://arxiv.org/abs/1707.04038}

\bibitem{Beirao-da-Veiga.Droniou.ea:10}
L.~Beir\~{a}o~da Veiga, J.~Droniou, M.~Manzini, A unified approach for handling
  convection terms in finite volumes and mimetic discretization methods for
  elliptic problems, IMA J. Numer. Anal. 31~(4) (2010) 1357--1401.
\newblock \href {http://dx.doi.org/10.1093/imanum/drq018}
  {\path{doi:10.1093/imanum/drq018}}.

\bibitem{Chainais-Hillairet.Droniou:11}
C.~Chainais-Hillairet, J.~Droniou, Finite-volume schemes for noncoercive
  elliptic problems with {N}eumann boundary conditions, IMA J. Numer. Anal.
  31~(1) (2011) 61--85.
\newblock \href {http://dx.doi.org/10.1093/imanum/drp009}
  {\path{doi:10.1093/imanum/drp009}}.

\bibitem{Droniou:10}
J.~Droniou, Remarks on discretizations of convection terms in hybrid mimetic
  mixed methods, Netw. Heterog. Media 5~(3) (2010) 545--563, proceedings of
  ``New Trends in Model Coupling''.
\newblock \href {http://dx.doi.org/10.3934/nhm.2010.5.545}
  {\path{doi:10.3934/nhm.2010.5.545}}.

\bibitem{Droniou.Eymard.ea:10}
J.~Droniou, R.~Eymard, T.~Gallou{\"e}t, R.~Herbin, A unified approach to
  mimetic finite difference, hybrid finite volume and mixed finite volume
  methods, Math. Models Methods Appl. Sci. 20~(2) (2010) 265--295.
\newblock \href {http://dx.doi.org/10.1142/S0218202510004222}
  {\path{doi:10.1142/S0218202510004222}}.

\bibitem{Droniou.Eymard:09}
J.~Droniou, R.~Eymard, Study of the mixed finite volume method for {S}tokes and
  {N}avier-{S}tokes equations, Numer. {M}ethods {P}artial {D}ifferential
  {E}quations 25~(1) (2009) 137--171.
\newblock \href {http://dx.doi.org/10.1002/num.20333}
  {\path{doi:10.1002/num.20333}}.

\bibitem{Ayuso-de-Dios.Lipnikov.ea:16}
B.~Ayuso~de Dios, K.~Lipnikov, G.~Manzini, The nonconforming virtual element
  method, ESAIM: Math. Model Numer. Anal. 50~(3) (2016) 879--904.
\newblock \href {http://dx.doi.org/10.1051/m2an/2015090}
  {\path{doi:10.1051/m2an/2015090}}.

\bibitem{Lipnikov.Manzini:14}
K.~Lipnikov, G.~Manzini, A high-order mimetic method on unstructured polyhedral
  meshes for the diffusion equation, J. Comput. Phys. 272 (2014) 360--385.
\newblock \href {http://dx.doi.org/10.1016/j.jcp.2014.04.021}
  {\path{doi:10.1016/j.jcp.2014.04.021}}.

\bibitem{Cangiani.Gyrya.ea:16}
A.~Cangiani, V.~Gyrya, G.~Manzini, The nonconforming virtual element method for
  the {S}tokes equations, SIAM J. Numer. Anal. 54~(6) (2016) 3411--3435.
\newblock \href {http://dx.doi.org/10.1137/15M1049531}
  {\path{doi:10.1137/15M1049531}}.

\bibitem{Kelley.Keyes:98}
C.~Kelley, D.~Keyes, Convergence analysis of pseudo-transient continuation,
  SIAM Journal on Numerical Analysis 35~(2) (1998) 508--523.
\newblock \href {http://dx.doi.org/10.1137/S0036142996304796}
  {\path{doi:10.1137/S0036142996304796}}.

\bibitem{Mulder.Van-Leer:85}
W.~A. Mulder, B.~Van~Leer, Experiments with implicit upwind methods for the
  euler equations, J. Comput. Phys. 59~(2) (1985) 232 -- 246.
\newblock \href {http://dx.doi.org/10.1016/0021-9991(85)90144-5}
  {\path{doi:10.1016/0021-9991(85)90144-5}}.

\bibitem{Schenk.Gartner.ea:01}
O.~Schenk, K.~G\"{a}rtner, W.~Fichtner, A.~Stricker, Pardiso: A
  high-performance serial and parallel sparse linear solver in semiconductor
  device simulation, Future Gener. Comput. Syst. 18~(1) (2001) 69--78.
\newblock \href {http://dx.doi.org/10.1016/S0167-739X(00)00076-5}
  {\path{doi:10.1016/S0167-739X(00)00076-5}}.

\bibitem{Guennebaud.Jacob.ea:10}
G.~Guennebaud, B.~Jacob, et~al., Eigen v3, http://eigen.tuxfamily.org (2010).

\bibitem{Ghia.Ghia.ea:82}
U.~Ghia, K.~N. Ghia, C.~T. Shin, High-{Re} solutions for incompressible flow
  using the {Navier--Stokes} equations and a multigrid method, J. Comput. Phys.
  48 (1982) 387--411.
\newblock \href {http://dx.doi.org/10.1016/0021-9991(82)90058-4}
  {\path{doi:10.1016/0021-9991(82)90058-4}}.

\bibitem{Erturk.Corke.ea:05}
E.~Erturk, T.~C. Corke, G\"{o}k\c{c}\"{o}l, Numerical solutions of {2-D} steady
  incompressible driven cavity flow at high {Reynolds}, Int. J. Numer. Meth.
  Fluids 48 (2005) 747--774.
\newblock \href {http://dx.doi.org/10.1002/fld.953}
  {\path{doi:10.1002/fld.953}}.

\bibitem{Albensoeder.Kuhlmann:05}
S.~Albensoeder, H.~C. Kuhlmann, Accurate three-dimensional lid-driven cavity
  flow, J. Comput. Phys. 206~(2) (2005) 536--558.
\newblock \href {http://dx.doi.org/10.1016/j.jcp.2004.12.024}
  {\path{doi:10.1016/j.jcp.2004.12.024}}.

\bibitem{Di-Pietro.Ern:12}
D.~A. Di~Pietro, A.~Ern, Mathematical aspects of discontinuous {Galerkin}
  methods, Vol.~69 of Math\'ematiques \& Applications, Springer-Verlag, Berlin,
  2012.
\newblock \href {http://dx.doi.org/10.1007/978-3-642-22980-0}
  {\path{doi:10.1007/978-3-642-22980-0}}.

\bibitem{Di-Pietro.Droniou:17*1}
D.~A. Di~Pietro, J.~Droniou, {$W^{s,p}$}-approximation properties of elliptic
  projectors on polynomial spaces, with application to the error analysis of a
  {Hybrid High-Order} discretisation of {Leray--Lions} problems, Math. Models
  Methods Appl. Sci. 27~(5) (2017) 879--908.
\newblock \href {http://dx.doi.org/10.1142/S0218202517500191}
  {\path{doi:10.1142/S0218202517500191}}.

\bibitem{Girault.Raviart:86}
V.~Girault, P.-A. Raviart, Finite element methods for {N}avier--{S}tokes
  equations, Vol.~5 of Springer Series in Computational Mathematics,
  Springer-Verlag, Berlin, 1986, theory and algorithms.

\end{thebibliography}

\end{document}